\documentclass[11pt, twoside, reqno]{amsart}

\usepackage{amsmath}
\usepackage{amsfonts}
\usepackage{amssymb}
\usepackage{amsthm}
\usepackage{hyperref}
\usepackage{graphicx}
\usepackage{cite}
\usepackage{caption}
\usepackage{color}

\newcommand{\Addresses}{{
  \bigskip
  \small
  
  \textsc{Department of Pure Mathematics and Mathematical Statistics, University of Cambridge,
Cambridge CB3 0WB, UK}\par\nopagebreak
  \textit{E-mail address:} \texttt{m.cekic@dpmms.cam.ac.uk}
  
}}

\usepackage[top=2.5cm, bottom=2.5cm, left=2.5cm, right=2.5cm, headsep=0.2in]{geometry}

\usepackage{fancyhdr}

\DeclareMathOperator{\Tr}{tr}
\DeclareMathOperator{\im}{Im}
\DeclareMathOperator{\re}{Re}

\theoremstyle{plain}
\newtheorem{theorem}{Theorem}[section]
\theoremstyle{definition}
\newtheorem{definition}[theorem]{Definition}
\theoremstyle{plain}
\newtheorem{lemma}[theorem]{Lemma}
\newtheorem{rem}[theorem]{Remark}
\newtheorem{prop}[theorem]{Proposition}
\newtheorem{conj}[theorem]{Conjecture}

\newtheorem{que}[theorem]{Question}

\numberwithin{equation}{section}

\newcommand{\Lapl}{\mathcal{L}}

\begin{document}

\title{The Calder\'{o}n problem for connections}
\author{Mihajlo Ceki\'{c}}
\date{}

\maketitle

\begin{abstract}
In this paper we consider the problem of identifying a connection $\nabla$ on a vector bundle up to gauge equivalence from the Dirichlet-to-Neumann map of the connection Laplacian $\nabla^*\nabla$ over conformally transversally anisotropic (CTA) manifolds. This was proved in \cite{LCW} for line bundles in the case of the transversal manifold being simple -- we generalise this result to the case where the transversal manifold only has an injective ray transform. Moreover, the construction of suitable Gaussian beam solutions on vector bundles is given for the case of the connection Laplacian and a potential, following the works of \cite{CTA}. This in turn enables us to construct the Complex Geometrical Optics (CGO) solutions and prove our main uniqueness result. We also reduce the problem to a new non-abelian X-ray transform for the case of simple transversal manifolds and higher rank vector bundles. Finally, we prove the recovery of a flat connection in general from the DN map, up to gauge equivalence, using an argument relating the Cauchy data of the connection Laplacian and the holonomy.
\end{abstract}

\section{Introduction}

The full Calder\'{o}n problem consists in determining a metric $g$ on a manifold up to an isometry that fixes every point of the boundary from the Dirichlet-to-Neumann (DN) map. It has been one of the main drives in the area of geometric inverse problems. In this generality the problem is still open, but considerable partial results exist under suitable assumptions on $g$. There are variations of this problem that are physically motivated and which have received a lot of attention -- namely, one can consider the operator $\Delta + X + q$, where $X$ is a first order term related to the magnetic potential and $q$ is a zero order term related to the electric potential. 


Moreover, a very interesting case is the one of the ``twisted" or connection Laplacian $\Lapl = \nabla^*\nabla$, where $\nabla$ is the covariant derivative. Let us consider a Hermitian vector bundle $E$ over a Riemannian manifold $(M, g)$ (equipped with a fibrewise Hermitian inner product) and a unitary connection $A$ on $E$. A \textit{gauge equivalence} $\psi$ is a section of the automorphism bundle $\text{Aut}E$, that is a bundle isomorphism that preserves the Hermitian structure. One then has a natural gauge invariance of the DN map (denoted by $\Lambda_A$ for the corresponding operator $\Lapl$) associated with the connection Laplacian on the vector bundle $E$; more precisely, if we denote the pullback connection by $\nabla_B = \psi^* \nabla_A = \psi \nabla_A \psi^{-1}$ and in addition we assume $\psi|_{\partial M} = Id$, then $\Lambda_A = \Lambda_B$. As with many similar problems, the question is: is this the only obstruction to injectivity? One can then pose the following:

\begin{conj}\label{conjecture1}
Given two unitary connections $A$ and $B$ on $E$, we have the equivalence: $\Lambda_A = \Lambda_B$ if and only if there exists a gauge equivalence that is the identity at the boundary that pulls back $B$ to $A$.
\end{conj}


This problem is solved completely for the case of surfaces in \cite{2D}. In higher dimensions there are several partial cases considered: \cite{LCW, Eskin, MagU}. The two most relevant for us are Eskin's result in \cite{Eskin} which solves the Conjecture \ref{conjecture1} when $M$ is a domain in Euclidean space with Euclidean metric and the result in \cite{LCW} which considers the line bundle case, where $(M, g)$ is assumed to have the admissible property. The approach we will take is the one initiated by Sylvester and Uhlmann \cite{SU} and later generalised by \cite{LCW} and others -- it can be briefly described in steps as:
\begin{enumerate}
\item Prove a suitable integral identity based on integration by parts.
\item Prove the necessary Carleman estimates and obtain the existence of the Complex Geometrical Optics (CGO) solutions.
\item Insert these solutions in the identity and use their density to make a global conclusion about the involved quantities.
\item Reduce the problem to a question of injectivity of an X-ray transform (or some other transform).
\end{enumerate}

In this work, we have mostly restricted our attention to the special class of manifolds defined below (this is the setting discussed in \cite{LCW} and \cite{CTA}):

\begin{definition}
Let $(M, g)$ be a smooth compact, oriented Riemannian manifold of dimension $n \geq 3$, with boundary and let $T = (\mathbb{R} \times M_0, e \oplus g_0)$, where $e$ is the Euclidean metric and $(M_0, g_0)$ a Riemannian manifold with boundary of dimension $n-1$. We say that $(M, g)$ is \textit{conformally transversally anisotropic} (CTA) if $(M, cg)$ is isometrically embedded into $T$ for some positive function $c$ on $M$.
\end{definition}
In this paper, we have completely covered and proved the conjecture for line bundles, in the case of CTA manifolds and with the hypothesis of injectivity of the ray transform on the transversal manifold $M_0$ (see Theorem \ref{maintheorem}) -- this result is new in the sense that we have significantly weakened the hypothesis on $M_0$.

In order to state the Main Theorem, we need to set up some notation: let $F(-\infty) = F = \{x \in \partial M \mid \langle{\frac{\partial}{\partial x_1}, \nu(x)}\rangle = c(x) \nu_1(x) \leq 0\}$, which we call the \emph{front side} and the analogous set with $\leq$ replaced with $>$ we call the \emph{back side}; here $\nu(x)$ is the outer normal. We also use the notation $\partial M_{-} = F$ and $\partial M_{+} = \overline{B}$ (see Figure~\ref{frontandback}). Moreover, we remark that this setup was used in \cite{MagU} in the connections setting in order to prove a suitable partial data result in Euclidean domains; the analogy with our case is that we are considering rays from the ``point at infinity", rather than from the points near the boundary. This approach for partial data problems (with the front and back face structure) originates from Bukhgeim and Uhlmann \cite{BU02} and was further developed developed by Kenig, Sj\"ostrand and Uhlmann \cite{LCW1}.

Furthermore, lets us spell out some basic definitions about the $X$-ray transform. Let $SM_0 = \{(x, \xi) \in\mid x \in M_0 \text{ and } |\xi| = 1\}$ denote the sphere bundle of $M_0$ and consider the set of all inward and outward pointing vectors:
\[\partial_{\pm}SM_0 = \{(x, \xi) \in SM_0 \mid x \in \partial M_0 \text{ and } \pm \langle{\xi, \nu(x)}\rangle \leq 0\}\]
Then, let us denote by $\gamma_{x, \xi}$ the unique geodesic in $M_0$ with $\gamma_{x, \xi}(0) = x$ and $\dot{\gamma}_{x, \xi}(0) = \xi$ for any $(x, \xi) \in TM$; we define the exit time $\tau(x, \xi)$ as the first time when $\gamma_{x, \xi}$ hits the boundary $\partial M_0$ (possibly infinite). Then we denote the set of \emph{trapped} geodesics by:
\[\Gamma_{+} = \{(x, \xi) \in \partial_+SM_0 \mid \tau(x, \xi) = \infty\}\]
With this in mind, we may define the geodesic $X$-ray transform of a smooth $1$-form $\alpha$ and a function $f$ on $M_0$, for all $(x, \xi) \in \partial_+ SM_0 \setminus \Gamma_+$:
\[I(f, \alpha)(x, \xi) = \int_0^{\tau(x, \xi)} \Big(f(\gamma_{x, \xi}(t)) + \alpha\big(\gamma_{x, \xi}(t), \dot{\gamma}_{x, \xi}(t)\big)\Big) dt\]
There is an obstruction to injectivity of this transform:
\begin{definition}
We say that the $X$-ray transform is injective on functions and $1$-forms if $I(f, \alpha) = 0$ implies that $f = 0$ and the existence of a smooth function $p$ on $M_0$ with $p|_{\partial M_0} = 0$ and $\alpha = dp$.
\end{definition}

We will need another definition -- this time it is about the ``admissible" vector bundles over $M$, which is a necessary topological condition to construct the CGO solutions.

\begin{definition}\label{vecadmissible}
Let $M \Subset \mathbb{R}\times M_0$ be a CTA manifold. A vector bundle $E$ over $M$ is called \emph{admissible} if it is isomorphic to a pullback bundle $\pi^*E_0$, where $E_0$ is a vector bundle over $M_0$ and $\pi: M \to M_0$ is the projection along the $x_1$-direction.
\end{definition}

Notice the condition of admissibility of the vector bundle $E$ is a necessary and sufficient condition for the bundle $E$ to have an extension $E'$ to $\mathbb{R} \times M_0$ such that $E'|_{M} = E$. We prove the following result:



\begin{theorem}[Main Theorem]\label{maintheorem}
Let $(M, g)$ be a CTA manifold. Let $E$ be an admissible Hermitian line bundle over $M$, equipped with unitary connections $A_1$ and $A_2$. Assume furthermore the injectivity of the ray transform on functions and 1-forms on $M_0$. If $\Gamma$ is a neighbourhood of the front face of $M$, then $\Lambda_{A_1}(f)|_{\Gamma} = \Lambda_{A_2}(f)|_{\Gamma}$\footnote{Alternatively, given a connection $A$ and a subset $\Gamma \subset \partial M$ of the boundary, the partial Cauchy data space are defined as $C^{\Gamma}_{A} = \{(u|_{\partial M}, d_A u (\nu)|_{\Gamma}) \big| d_A^*d_Au = 0 \text{ and } u\in H^1(M)\}$, where $\nu$ is the outward normal; then by definition $C^{\Gamma}_{A_1} = C^{\Gamma}_{A_2}$ if and only if $\Lambda_{A_1}(f)|_{\Gamma} = \Lambda_{A_2}(f)|_{\Gamma}$ for all $f$.} for all $f \in C^{\infty}(\partial M; E|_{\partial M})$ implies the existence of a gauge equivalence that is the identity on $\Gamma$ and which pulls back $A_2$ to $A_1$.
\end{theorem}

Firstly, we remark that the CGO solutions supported in a front or a back face were constructed by Chung in \cite{FC} for Euclidean domains -- this probably implies such solutions could be constructed in our setting. The existence of such CGOs would reduce the assumption of the theorem to $\Lambda_{A_1}(f)|_{\Gamma} = \Lambda_{A_2}(f)|_{\Gamma}$ for all $f \in C_0^\infty(\Gamma)$; however, due to technical reasons and simplicity we will deal only with the full Dirichlet data.

\begin{figure}[h]
    \centering
    \hspace*{-2.8cm}
    \includegraphics[width=0.8\textwidth]{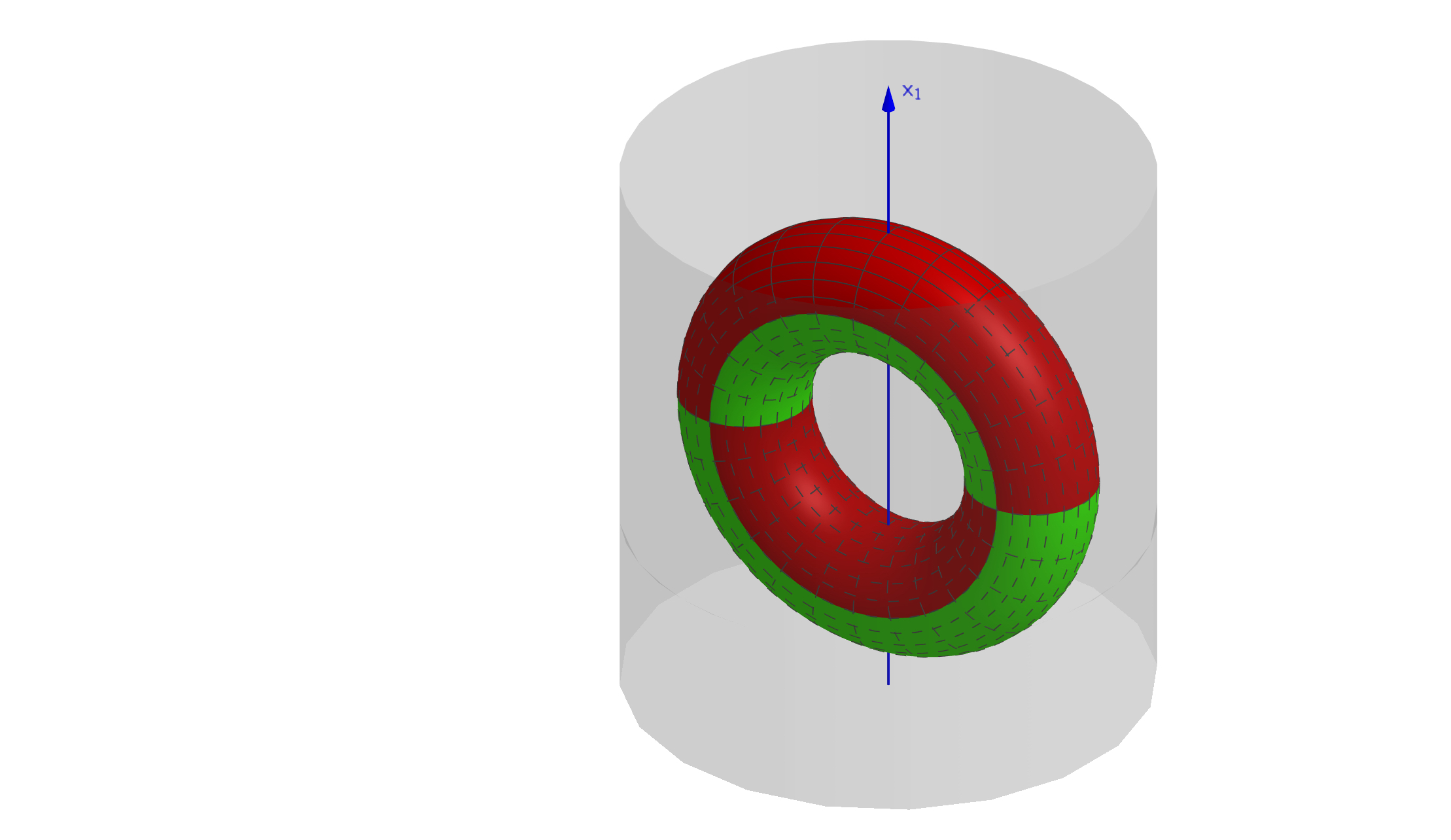}
               \captionsetup{format=hang, labelfont=sc}
    \caption{Solid torus as a CTA manifold, showing front (green) and back (red) faces.}
    \label{frontandback}
\end{figure}

This particular setting is interesting, because of the existence of the ``Euclidean direction" in our manifold, i.e. the direction set out by $\mathbb{R}$; this enables us to define a Carleman weight $\varphi (x) = x_1$, which in turn allows for the CGO solutions to be constructed (see \cite{LCW}; for an alternative construction of the CGOs using the Fourier transform in the $x_1$ variable, see \cite{mikko_calderon}). Our construction is based on the solutions known as Gaussian beams, which have already shown to be fertile in the less complicated case of the operator $\Delta + q$ in \cite{CTA}. We have also adapted the construction to the case of the connection Laplacian, valid for functions with values in a vector bundle; the idea is to show existence of approximate solutions which concentrate in a suitable way around geodesics. This is done locally in charts covering the geodesic by using a WKB-type construction and then glued together to form a global solution. Moreover, it is worth emphasising that our main result Theorem \ref{maintheorem} generalises the one present in \cite{LCW}, in that it \emph{does not ask for $M_0$ to be simple}\footnote{Simple manifolds are those for which the geodesics from every point parametrise the manifold, or more formally the exponential map is a diffeomorphism from its domain of definition; in addition, one also asks that the boundary is strictly convex (second fundamental form positive definite). The following implication holds: $M_0$ simple $\Rightarrow$ injectivity of the ray transform (see \cite{xray} for a short survey).}, which complicates the CGO construction significantly -- more concretely, it allows for the geodesics on $M_0$ to self-intersect and allows for the existence of \emph{conjugate points} (which prevent the exponential map from being a diffeomorphism). 

Furthermore, in Section 6 another approach based on the interplay between the parallel transport and the unique continuation principle (UCP) for elliptic equations is pursued. Theorem \ref{holtheorem} proves the Conjecture \ref{conjecture1} in the setting of partial data, in the case of two flat connections. The latter assumption simplifies the problem significantly, because the parallel transport along homotopic curves is then the same, which enables us to define a suitable gauge. A similar idea was already used in \cite{hol} in the case of line bundles over surfaces. Moreover, there is a natural way of pushing these results further to the case of Yang-Mills connections; this will be considered in a forthcoming paper, in which we also prove boundary determination for connections and potentials for general vector bundles (that is, the restriction of the connection to the boundary is determined from the DN map -- see Section 8 in \cite{LCW} for the case of line bundles).

In addition to the above, we also provide a general framework and base for the future work in the direction of the Calder\'{o}n problem for connections on vector bundles, by constructing the CGOs in general (see Theorem \ref{main} and Remark \ref{mainvector}). For simple transversal manifolds and the trivial vector bundle of any rank, we also get to the fourth step in our previous analysis -- see Section 4. Moreover, in this case, one can reduce the main DN inverse problem to a \textit{new non-abelian X-ray transform} -- see Question \ref{conjecture2}, which we have not found in the literature. The reduction process is fully explained and outlined in Section \ref{sec4.5}. One distinct feature of this transform is that it involves the complex derivative $\mathbb{X} = \frac{\partial}{\partial x_1} + iX$, rather than just the usual geodesic vector field derivative $X$ -- hence, one could expect that methods from complex analysis and geometry might be useful to deduce certain properties of this transform (as in \cite{Eskin}). Another characteristic property of this transform is that it is not abelian in general, making it harder to reduce to an $X$-ray transform on just $M_0$, which is usually done in such situations (see \cite{LCW}). The question is posed here in the form of a transport equation.

\begin{que}[The non-abelian Radon transform]\label{conjecture2}
Let $(M_0, g_0)$ be a compact simple manifold with boundary, with $\dim M_0 \geq 2$ and let $M$ be an isometrically embedded, compact submanifold of $T = (\mathbb{R} \times M_0, e \oplus g_0)$ with non-empty boundary and $\dim M = \dim T$. Let $E = \mathbb{R} \times M_0 \times \mathbb{C}^m$ be a Hermitian vector bundle equipped with two unitary connections $A_1$ and $A_2$, which are compactly supported and satisfy $A_1 = A_2$ on $\mathbb{R} \times M_0 \setminus M$. Let $R' = \{(x_1, x, v) \in \mathbb{R} \times SM_0: (x_1, x) \not \in M\}$. Assume we are given a smooth matrix function $G: \mathbb{R}\times SM_0 \to GL(m, \mathbb{C})$ such that, if $\mathbb{X} = \frac{\partial}{\partial x_1} + iX$, where $X$ is the geodesic vector field:
\begin{align*}
\mathbb{X}G(x_1, x, v) = -A_1(x_1, x)\Big(\frac{\partial}{\partial x_1} + iv\Big)G(x_1, x, v) + G(x_1, x, v)A_2(x_1, x)\Big(\frac{\partial}{\partial x_1} + iv\Big)
\end{align*}
for all $(x_1, x, v)$, with the additional condition $G|_{R'} = Id$. Prove that $G$ is independent of the velocity variable $v$.
\end{que}


In order to support our Theorem \ref{maintheorem}, let us list a number of results that generate a large class of non-trivial examples for which our theorem is new. Firstly, the results of Stefanov, Uhlmann and Vasy \cite{uv1, uv2} give the injectivity of the ray transform if the manifold is foliated by convex hypersurfaces up to a small set; secondly, the result of Guillarmou in \cite{colin_lens} proves the injectivity in the case of manifolds with negative curvature and strictly convex boundary (second fundamental form positive). Finally, the very recent results of Paternain, Salo, Uhlmann and Zhou \cite{paternain2016} show that the geodesic transform is injective in the case of strictly convex manifolds with non-negative sectional curvature. The second one of these results allows existence of trapping (geodesics of infinite length), while the third one allows for the existence of conjugate points. As a concrete example of where our Main Theorem is a new result, we can let the transversal manifold $M_0$ be a catenoid -- a surface with negative curvature and for which the boundary is strictly convex; it has geodesics that are trapped (e.g. the middle circle) and hence is not simple, but the ray transform is injective by the results in \cite{colin_lens}.

Let us briefly indicate some classes of manifolds discussed in the previous paragraph that admit non-trivial line bundles. We will obtain an example where Theorem \ref{maintheorem} applies and $E$ is non-trivial, by letting $M$ to be a product manifold of such $M_0$ and a unit interval (smoothed at the corners) and $E$ to be the pullback bundle under $\pi$, as in Definition \ref{vecadmissible}. It is well-known that the topological line bundles of a space $X$ are classified by the second cohomology $H^2(X; \mathbb{Z})$ (isomorphism given by the first Chern class). Firstly, notice that all compact surfaces with non-empty boundary have a trivial $H^2$, so we have to look for $\dim M_0 \geq 3$. Let us discuss the case $\dim M_0 = 3$. By \cite{Hass} (Section 3.1), the only three manifold $M_0$ with positive sectional curvature and strictly convex boundary is the $3$-ball. On the contrary, there are plenty of examples of $3$-manifolds with negative sectional curvature and strictly convex boundary -- by \cite{Hass} (Section 3.4), such manifolds are completely classified by being irreducible and having no $\pi_1$-injective tori (right to left direction follows from Thurston's hyperbolisation theorem for Haken manifolds). More concretely, if additionally we want an example with $H^2(M_0; \mathbb{Z})$ of non-zero rank, it follows that we may take $M_0 = S_g \times [0, 1]$ for $g \geq 2$, where $S_g$ denotes the closed surface of genus $g$.





The paper is organised as follows: in the next section we provide some elementary background and also prove an integral identity based on integration by parts, while Sections \ref{sec3} and \ref{sec5} are the most technical ones -- in the former one we prove the necessary Carleman estimates for sections of vector bundles using semiclassical calculus. The latter one we divide into two parts: in the first part, we present the lengthy construction of the version of Gaussian beam solutions that is relevant for us, for general vector bundles; in the second one, we apply this construction to deduce the existence of CGO solutions. Moreover, in Section \ref{secrecovery} we prove that we may recover the differential of the connection $dA$ from the DN map in the case of line bundles. However, before that in Section \ref{CGOsimple} we consider the case where the transversal manifold is simple and for which we may construct the ansatz in a much easier way -- we conclude with a reduction to a new ray transform. Finally, the Section \ref{holsec} considers the case of two flat connections and finishes the proof of the main theorem.
\medskip

\noindent \textbf{Acknowledgements.} The author would like to express his gratitude for the support of his supervisor, Gabriel Paternain and also to Trinity College for financial support. Furthermore, he would also like to acknowledge the very useful conversations with Mikko Salo, especially in regard to the contents of the third section of the paper.

\section{Preliminaries and the identity}\label{sec2}
Throughout this section, $(M, g)$ is a compact connected Riemannian manifold of dimension $n$ with boundary, $E$ is a Hermitian vector bundle of rank $m$ over $M$, equipped with a unitary connection $\nabla$. Let $\nu$ be the outward normal to $\partial M$. We also fix a matrix valued potential $Q$, that is a section of the endomorphism bundle of $E$. Moreover, we will denote the sections of $E$ by $C^{\infty}(M; E)$ or by $\Gamma(E)$ (both notations are standard). Recall that the connection gives rise to a covariant derivative $\nabla: \Gamma(E) \to \Gamma(E \otimes T^*M)$; moreover, in a trivial vector bundle $E = M \times \mathbb{C}^m$ with the standard Hermitian inner product in the fibers, a connection is given by a $m \times m$ matrix of one-forms $A$ and the covariant derivative by $d_A = d + A$. We will interchangeably use the following symbols for the covariant derivative: $d_A$, $\nabla_A$ and $\nabla$; subscript $A$ here denotes the connection as a formal object, but can also mean the connection $1$-form, depending on the context. Furthermore, we will assume the summation convention, where repeated indices mean that we sum over the corresponding index. All manifolds in this paper are assumed to be orientable.

The covariant derivative being unitary, means the following compatibility condition: $d\langle{u, v}\rangle_E = \langle{\nabla u, v}\rangle_E + \langle{u, \nabla v}\rangle_E$. We can use the Hermitian inner product to define inner product on sections of $E$: 
\begin{center}
$(u, v)_{L^2(M; E)} = \int_M \langle{u, v}\rangle_E dV$
\end{center}
where $dV$ is the volume form on $M$ (sometimes omitted from the integrals for simplicity) and more generally on $E$-valued one forms (that is, sections of $C^{\infty}(M; E \otimes T^*M)$), where in local coordinates $\alpha = \alpha_i dx^i$ and $\beta = \beta_i dx^i$
\begin{center}
$(\alpha, \beta)_{L^2(M; E \otimes T^*M)} = \int_M g^{ij}\langle{\alpha_i, \beta_j}\rangle_E dV$
\end{center}
By slightly abusing the notation, we will sometimes also use $(\cdot, \cdot)$ (in addition to $\langle{\cdot, \cdot}\rangle$) to denote the fibrewise inner product on differential forms -- this will be clear from the context.

Let $\nabla^*$ (or $d_A^*$) be the formal adjoint of the covariant derivative with respect to these inner products. Now using Stokes' theorem one can prove that the following identity holds (see \cite{pat2015UCP}):
\begin{align}\label{greenidentity}
(\nabla^*u, v)_{L^2} - (u, \nabla v)_{L^2} = -(\iota_\nu u, v)_{L^2(\partial{M}; E|_{\partial M})}
\end{align}
\noindent where $u$ is an $E$-valued one form and $v$ is a section of $E$.

Now we can define the \textit{twisted} or the \textit{connection Laplacian} as
\begin{align*}
\Lapl_\nabla = \nabla^*\nabla
\end{align*}
We also denote by $\Lapl_{\nabla, Q} = \nabla^*\nabla + Q$ the corresponding Schr\"{o}dinger operator and $\Lapl_{g, \nabla, Q}$ when we want to emphasise the dependence on the metric. With the assumption that $0$ is \emph{not a Dirichlet eigenvalue} of $\Lapl_{\nabla, Q}$ in $M$ (so we have unique solvability of the Dirichlet problem), the DN map is defined as:
\begin{definition}
The \textit{Dirichlet-to-Neumann map} or the \textit{DN map} $\Lambda_{\nabla, Q}: H^{\frac{1}{2}}(\partial M; E|_{\partial M}) \to H^{-\frac{1}{2}}(\partial M$\\$; E|_{\partial M})$ is defined as\footnote{In a suitable weak sense.} $\Lambda_{\nabla, Q}f = \iota_\nu \nabla u|_{\partial M}$, where $u$ is the unique solution of the elliptic problem:
\begin{equation*}
\Lapl_{\nabla, Q}u = 0 \quad \text{and} \quad u|_{\partial M} = f
\end{equation*}
where $f \in H^{\frac{1}{2}}(\partial M; E|_{\partial M})$.
\end{definition}

An alternative (not always equivalent) and a more general way of phrasing the equality of the DN maps is through the \textit{Cauchy data spaces}. The full Cauchy data spaces are given by $C_{\nabla, Q} = \Big\{\Big(u|_{\partial M}, \iota_{\nu} \nabla u|_{\partial M}\Big) \Big| \Lapl_{\nabla, Q}u = 0, u \in H^1(M)\Big\}$. It is important to point out that in one of the cases that are important to us, that is when $Q = 0$, we automatically have that zero is not a Dirichlet eigenvalue of the operator $\Lapl_\nabla$ and so the DN map is well defined.

The following identity will be used and is stated in \cite{pat2015UCP}. For a connection $A$ on a trivial bundle $E = M \times \mathbb{C}^m$, we may define $(A, \beta) = g^{ij}A_i\beta_j$ for $\beta$ an $E$-valued one-form. One can easily check that $\nabla^* = d_A^* = d^* - (A, \cdot)$, where $d^*$ is the adjoint of the ordinary differential. One can then get the expression:
\begin{lemma}
If $E = M \times \mathbb{C}^m$, we have the following expansion for the twisted Laplacian, where $A = A_idx^i$:
\begin{align*}
\Lapl_A u= \Delta u - 2(A, du) + (d^*A)u - (A, Au)
\end{align*}
\end{lemma}
For clarity, here we take the Laplacian to be with a negative sign, i.e. $\Delta u= d^* d u = -|g|^{-1/2} \frac{\partial}{\partial x^j}\\ (|g|^{1/2} g_{jk} \frac{\partial u}{\partial x^k})$, so our operator is positive definite. Therefore, we can clearly identify the second, the first and the zero order terms in the connection Laplacian. If we let $(A, Q)$ be a pair of a connection and a potential, we will sometimes use the notation of the pair $(X, q)$ to denote the matrix vector field $X$ and the matrix potential $q$ such that:
\begin{align}\label{notation}
d_A^*d_A + Q = \Delta + X + q
\end{align}
in local coordinates, or globally if the corresponding bundle is trivial. The relationship between $(A, Q)$ and $(X, q)$ is given by $X = -2g^{ij}A_i \frac{\partial}{\partial x^j}$ and $q(u) = d^*A - (A, Au) + Q(u)$. 

The next lemma computes the adjoint of the DN map, where $Q$ is in $\Gamma(\text{End }E)$:

\begin{lemma}
The following identity holds for smooth $f$ and $g$ ($Q^*$ is the Hermitian conjugate):
\begin{center}
$(\Lambda_{\nabla, Q} f, g)_{L^2(\partial M; E|_{\partial M})} = (f, \Lambda_{\nabla, Q^*} g)_{L^2(\partial M; E|_{\partial M})}$
\end{center}
\end{lemma}

\begin{proof}
We drop the full notation of $L^2(M; E)$. By using \eqref{greenidentity} we have:
\begin{equation}
(Q u, v)_M + (\nabla u, \nabla v)_M = (\Lambda_{\nabla, Q} f, g)_{\partial{M}}
\end{equation}
where $\Lapl_{\nabla, Q} u = 0$ and $u|_{\partial M} = f$ and any $v$ such that $v|_{\partial M} = g$. If we swap the order of $f$ and $g$ and use the fact that the inner product is Hermitian, along with $v$ being the solution to $\Lapl_{\nabla, Q^*} v = 0$ and $v|_{\partial M} = g$, we get:
\begin{align*}
(Q^* v, u)_M + (\nabla v, \nabla u)_M= (\Lambda_{\nabla, Q^*} g, f)_{\partial M}
\end{align*}
which after conjugation finishes the proof.
\end{proof}

Now we restrict our attention to the trivial vector bundle $E = M \times \mathbb{C}^m$ with the connection matrix $A$. We will use the notation $|A|^2 = g^{ij} A_i A_j$ -- please note this is \emph{not a norm}, but rather comes from the complex bilinear extension of the metric inner product and that it is endomorphism valued. Also, $(A_j)_{kl}$ will denote the $kl^{\text{th}}$ entry of the matrix $A_j$ given by the expansion $A = A_jdx^j$.

\begin{theorem}[Main identity]\label{identity}
The following identity holds for two pairs of smooth unitary connections and potentials $(A, Q_A)$ and $(B, Q_B)$, and $f$ and $g$ smooth sections of $E|_{\partial M}$:
\begin{multline}
\big((\Lambda_{A, Q_A} - \Lambda_{B, Q_B}) f, g\big)_{\partial M} = 
\Big(\big(Q_A - Q_B +|B|^2 - |A|^2\big) u , v\Big)_M  \\
+ \int_M g^{ij}\big((A - B)_j\big)_{kl} \left(u_l \frac{\partial \bar{v}_k}{\partial x^i} - \frac{\partial u_l}{\partial x^i} \bar{v}_k\right)
\end{multline}
where $u, v \in C^\infty(M; E)$ solve $\Lapl_{A, Q_A} u = 0$ with $u|_{\partial M} = f$ and $\Lapl_{B, Q^*_B} v = 0$ with $v|_{\partial M} = g$. Equivalently, for $m = 1$ one can write this as:
\begin{multline*}
\big((\Lambda_{A, q_A} - \Lambda_{B, Q_B}) f, g\big)_{\partial M} = 
\Big(\big(Q_A - Q_B + |B|^2 - |A|^2\big) u , v\Big)_M \\
+ \int_M \langle{u d\bar{v} - \bar{v} du, B - A}\rangle_g
\end{multline*}
\end{theorem}

\begin{proof}
As above, we have:
\begin{equation*}
(\Lambda_{A, Q_A} f, g)_{\partial{M}} = (Q_A u, v)_M + (d_{A} u, d_{A} v)_M
\end{equation*}
and similarly, where $u$ and $v$ as in the statement:
\begin{align*}
(\Lambda_{B, Q_B} f , g)_{\partial{M}} &= (f ,\Lambda_{A_B, Q^*_B} g)_{\partial{M}} \\
&=  (Q_B u , v)_M + (d_{B} u , d_{B} v)_M
\end{align*}
So we get by subtracting:
\begin{align*}
\big((\Lambda_{A, Q_A} - \Lambda_{B, Q_B}) f , g\big)_{\partial M} = \big((Q_A - Q_B)u , v\big)_M + (d_{A} u , d_{A} v)_M - (d_{B} u , d_{B} v)_M
\end{align*}
We have $(A u , A v)_{M} = - \big(|A|^2 u , v\big)_M$ and $(B u , B v)_{M} = - \big(|B|^2 u , v\big)_M$ and moreover:
\begin{align*}
\big(du , (A - B) v\big)_M + \big((A - B)u , dv\big)_M =  \int_M g^{ij} \big((A - B)_i\big)_{kl} \left( u_l \frac{\partial \bar{v}_k}{\partial x^j} - \frac{\partial u_l}{\partial x^j} \bar{v}_k \right)
\end{align*}
by the skew-Hermitian property of $A$ and $B$, where $u_l$ and $v_k$ denote the components of the vectors $u$ and $v$. By putting the pieces together, this finishes the proof.
\end{proof}


Let us now denote by $E' = M \times \mathbb{C}^{m \times m}$ the endomorphism bundle of $E$, carying the natural trace Hermitian inner product $\langle{X, Y}\rangle = \Tr(XY^*)$. Then we can naturally let the $\Lapl_{A, Q}$ operator act on matrices by matrix multiplication\footnote{Note $\Lapl_{A, Q}$ is \textit{not} the same as the connection Laplacian obtained from the standard induced connection $D_A U= dU + [A, U]$ on the endomorphism bundle.}; furthermore, one easily shows the similarly extended DN maps for $A_1$ and $A_2$ on $E'$ obtained in this way agree if and only if the usual DN maps for $A_1$ and $A_2$ agree on $E$ -- one just notices that the first claim is the same as the second one applied to all of $n$ column vectors. Therefore, we have a version of the previous identity for matrices, where by capital letter we denote a matrix instead of a vector (we will need it in Section \ref{sec4.5}):

\begin{theorem}[The identity for matrices]\label{matidentity} In the notation as in Theorem \ref{identity}, for two smooth sections $F$ and $G$ of $E'|_{\partial M}$, we have:
\begin{multline}
\Big(\big(\Lambda_{A, Q_A} - \Lambda_{B, Q_B}\big) F , G\Big)_{\partial M} = 
\Big(\big(Q_A - Q_B +|B|^2 - |A|^2\big) U , V\Big)_M \\
+ \Big(U (dV^*) - (dU) V^* , B - A\Big)_M
\end{multline}
where $U, V \in C^\infty(M; E')$ solve $\Lapl_{A, Q_A} U = 0$ with $U|_{\partial M} = F$ and $\Lapl_{B, Q^*_B} V = 0$ with $V|_{\partial M} = G$. 
\end{theorem}
\begin{proof}
By re-running the proof of the previous theorem, we easily obtain the result; we use the convenient matrix identities such as $(AU, dV)_{M} = -\big(U (dV^*), A\big)_M$ and $(dU, AV)_M = \big((dU)V^*, A\big)_M$.
\end{proof}

\section{Carleman estimates}\label{sec3}

The purpose of this section is to prove suitable Carleman estimates for vector valued functions. The scalar case was covered in \cite{LCW} and we generalise that approach, as expected since the principal part of $d_A^*d_A$ is diagonal.

Firstly, let us briefly explain what the limiting Carleman weights (LCW) are. These are certain functions on open Riemannian manifolds that guarantee the positivity of the conjugated Laplacian operator $P_{0, \varphi} = e^{\frac{\varphi}{h}} \Delta e^{-\frac{\varphi}{h}}$ and hence existence of solutions to equations as below. They were introduced in \cite{LCW1} for the Euclidean case and generalised to manifolds in \cite{LCW}. They have a nice geometric characterisation: in \cite{LCW} it is proved that the existence of LCW is equivalent to existence of a unit parallel vector field on the manifold (a vector field $V$ is parallel if $\nabla V = 0$, where $\nabla$ is the Levi-Civita connection). This vector field yields a Euclidean direction on the manifold -- hence, for simplicity, we will often assume our manifold to be embedded in $\mathbb{R} \times M_0$, which admits the Carleman weight $\varphi(x) = x_1$.

Moreover, one way to motivate the definition of LCWs is that its reverse engineered so that the estimates below in Theorem \ref{specslucaj} hold for both $\pm \varphi$ (the proof of the converse to this statement, i.e. that the inequality holds for both $\pm \varphi$ implies that $\varphi$ is an LCW is outlined in \cite{DSFlecturenotes}), so that the two solutions constructed in Proposition \ref{mainconstruction} with the corresponding phases equal to $\pm x_1$, cancel out in the integral identity from Theorem \ref{identity}. We state the definition of an LCW here.

\begin{definition}
Let $(U, g)$ be an open Riemannian manifold. We say $(U, g)$ admits an LCW if there exists a smooth $\varphi: U \to \mathbb{R}$, such that $d\varphi \neq 0$ and if we let $p_\varphi$ to be the semiclassical principal symbol of $P_{0, \varphi}$, then:
\begin{align*}
\{\re p_\varphi, \im p_\varphi\}(x, \xi) = 0 \quad \text{when} \quad p_\varphi(x, \xi) = 0
\end{align*}
where $\{\cdot, \cdot\}$ is the Poisson bracket on $T^*U$.
\end{definition}

In the text below, we denote by $H^1_{scl}(M; E)$ the semiclassical Sobolev space associated to the sections of the Hermitian vector bundle $E$ of rank $m$ over $M$, equipped with a connection $\nabla$, with the norm:
\begin{equation*}
\lVert{u}\rVert_{H^1_{scl}(M; E)} = \big(\lVert{u}\rVert^2_{L^2(M; E)} + h^2\lVert{\nabla u}\rVert^2_{L^2(M; T^*M \otimes E)}\big)^{\frac{1}{2}}
\end{equation*}
and by $L^2(M; E)$ the inner product space associated with the Hermitian structure and the Riemannian density. We start by proving a warm-up a priori Carleman estimate which relates the $H^1_{scl}$ and $L^2$ norms of a solution to $P_{0, \varphi} u = f$, by essentially using only elementary methods; later we will see, in order to obtain a $H^1$ solution, we have to shift the indices and prove the inequality for every $H^s_{scl}$, where $s \in \mathbb{R}$. 

Let us introduce the setting in which the theorems will be proved. We will work on $(M, g)$, a compact Riemannian manifold with boundary which is compactly contained in $(U, g)$, an open Riemannian manifold admitting a limiting Carleman weight $\varphi$; moreover, $U$ is again contained in a closed Riemannian manifold $(N, g)$, which is useful since then we do not have to worry about boundary conditions on $N$ (e.g. we may let $N$ to be the double of a neighbourhood of $M$ in $U$). Furthermore, we will work on a Hermitian vector bundle $E$ of rank $m$ over $M$, equipped with a connection $A$ and a section of $Q$ of the endomorphism bundle. We also assume there is an extension of $E$ to a bundle over $N$, denoted by the same letter and that $A$ and $Q$ extend, too.

\begin{theorem}\label{specslucaj}
Let $X$ be a smooth matrix of vector fields on $M$ and $q$ a smooth matrix function on $M$ (matrices are $m$ by $m$). Then there exists a constant $C$, such that the following inequality holds for all $u \in C^\infty_c(M^{int}; \mathbb{C}^m)$ and all sufficiently small $h > 0$:
\begin{equation}
\lVert{e^{\varphi /h} u}\rVert_{H^1_{scl}(M; \mathbb{C}^m)} \leq Ch \lVert{e^{\varphi /h} (\Delta + X + q) u}\rVert_{L^2(M; \mathbb{C}^m)}
\end{equation}
Moreover, the following inequality holds for all $u \in C^\infty_c(M^{int}; E)$, for some constant $C' > 0$:
\begin{align}
\lVert{e^{\varphi /h} u}\rVert_{H^1_{scl}(M; E)} \leq C'h \lVert{e^{\varphi /h} (d_A^*d_A + Q) u}\rVert_{L^2(M; E)}
\end{align}
\end{theorem}
\begin{proof}
We prove the first inequality; the second one follows by a partition of unity argument in $N$ and applying the first inequality in the form \eqref{convexificationineq} to absorb the extra factors, since locally $d_A^*d_A + Q$ is of the form $\Delta + X + q$ (see \eqref{notation}).

Firstly, notice we have invariance under conformal scaling, i.e. observe that we have the identity:
\begin{equation*}
c^{\frac{n+2}{4}}(\Delta_g + X + q)u = (\Delta_{c^{-1}g} + cX + q_c)(c^{\frac{n-2}{4}} u)
\end{equation*}
where $q_c = cq - \frac{n-2}{4}Xc + c^{\frac{n+2}{4}} \Delta(c^{\frac{n-2}{4}})$, by using the conformal properties of the Laplacian. By taking $h$ small enough, we easily deduce that without loss of generality we are free to pick the conformal constant. 

We can now assume that $\nabla \varphi$ has unit norm, as conformal scalings preserve the property of being a LCW. In other words we may assume that the function $\varphi$ is a \textit{distance function}, i.e. we have $|\nabla \varphi| = 1$ and $D^2 \varphi = 0$, where $D$ is the Levi-Civita covariant derivative (see Lemma 2.5 in \cite{LCW}).\footnote{In \cite{LCW} it is also proved that a distance function is also an LCW if and only if $\varphi(\exp_x v) = \varphi(x) + \langle{\nabla \varphi(x), v}\rangle$; in particular, this means that we have a lot of LCWs in the Euclidean spaces, by letting $\varphi (x) = \rho \cdot x$ for a vector $\rho$.}

We can further assume that $q = 0$, since the corresponding factor can be absorbed to the left hand side for $h$ small enough.

In this step we show the inequality under the additional assumption that $X = 0$. Recall the following identity, with the specific expansion we will make use of later:
\begin{gather*}
P_{0, \varphi} = e^{\varphi /h} h^2 \Delta_g e^{-\varphi/h} = \underbrace{h^2 \Delta - |\nabla \varphi|^2}_{A} + \underbrace{2\langle{\nabla \varphi, h \nabla}\rangle - h\Delta \varphi}_{i B}
\end{gather*}
Hence, we can build the following estimates (we leave out the $L^2$ subscript for convenience):
\begin{align*}
(P_{0, \varphi} v, v) = h^2 (\Delta v, v) - (|\nabla \varphi|^2 v, v) + 2h (\langle{\nabla \varphi, \nabla v}\rangle, v) - h(\Delta \varphi v, v)
\end{align*}
By using the fact that $\int{\langle{df, dg}\rangle} = (\Delta f, g)$ for $f$ and $g$ compactly supported, we get:
\begin{gather*}
\lVert{h \nabla v}\rVert^2 = (P_{0, \varphi} v, v) + (|\nabla \varphi|^2 v, v) - 2h(\langle{\nabla \varphi, \nabla v}\rangle, v) + h(\Delta \varphi v, v)
\end{gather*}
Therefore, we finally have, using Cauchy-Schwartz and AM-GM:
\begin{align*}
\lVert{h \nabla v}\rVert^2 &\leq \lVert{P_{0, \varphi} v}\rVert \lVert{v}\rVert + \lVert{v}\rVert^2 + 2 \lVert{v}\rVert \lVert{h \nabla v}\rVert + h |\sup \Delta \varphi| \lVert{v}\rVert^2 \\
&\leq \frac{1}{2} \lVert{P_{0, \varphi} v}\rVert^2 + \frac{1}{2} \lVert{v}\rVert^2 + \frac{1}{\epsilon} \lVert{v}\rVert^2 + \epsilon \lVert{h \nabla v}\rVert^2
\end{align*}
So for some $C_1$ and sufficiently small $\epsilon$:
\begin{align*}
\lVert{h \nabla v}\rVert^2 \leq \lVert{P_{0, \varphi} v}\rVert^2 + C_1 \lVert{v}\rVert^2
\end{align*}
Therefore, it suffices to prove $\lVert{v}\rVert \leq C_2 h^{-1} \lVert{P_{0, \varphi} v}\rVert$ for some $C_2$. 

Now, we claim that in the above expansion of $P_{0, \varphi}$, the parts $A$ and $B$ are formally self-adjoint. The proof is not too hard, but we give one for completeness. The bilinear map $\langle{\cdot, \cdot}\rangle$ we use is complex bilinear; also, formal self-adjointness means $(P \varphi, \psi) = (\varphi, P \psi)$ for all smooth compactly supported functions $\varphi$ and $\psi$. We have, for $m = 1$:
\begin{gather*}
\big((h^2\Delta - |\nabla \varphi|^2)u, v\big) = \big(u, (h^2\Delta - |\nabla \varphi|^2) v\big)
\end{gather*}
for all $u, v \in C_c^{\infty}(M^{int})$ because $\varphi$ is real and $\Delta$ is self-adjoint. Moreover, we have:
\begin{multline*}
\Big(2\langle{\nabla \varphi, h \nabla u}\rangle, v\Big) = 2h \Big(\langle{d \varphi, du}\rangle, v\Big) = 2h \int{\langle{d \varphi, \bar{v} du}\rangle} \\
= 2h \int{\langle{d \varphi, d(u \bar{v}) - u d\bar{v}}\rangle} = 2h \int{\Delta \varphi u \bar{v}} - 2h \Big(u, \langle{\nabla \varphi, \nabla v}\rangle\Big)
\end{multline*}
and $\Big(h \Delta \varphi u, v\Big) = h \Big(u, \Delta \varphi v\Big)$. Therefore, by combining the two results:
\begin{multline*}
\Big(2 \langle{\nabla \varphi, h \nabla u}\rangle - h \Delta \varphi u, v\Big) = 2h \Big( \Delta \varphi u, v\Big) - h \Big(\Delta \varphi u, v\Big) - 2h \Big(u, \langle{\nabla \varphi, \nabla v}\rangle\Big)\\
= - \Big(u, (-h \Delta \varphi v + 2h \langle{\nabla \varphi, \nabla v}\rangle)\Big) = - \Big(u, iB v\Big)
\end{multline*}
which finally implies that $A$ and $B$ are formally self-adjoint in the scalar case. For the $m > 1$ case we just observe that the action of the Laplacian $\Delta$ extends diagonally to vector valued functions and the inner product $\langle{u, v}\rangle = \sum u_i \bar{v}_i$ splits nicely with respect to this action, so we can simply sum over components.

We will now make use of the following identity:
\begin{align*}
\lVert{P_{0, \varphi} v}\rVert^2 = (P_{0, \varphi} v, P_{0, \varphi} v) = \big((A + iB)v, (A + iB)v\big) = \lVert{Av}\rVert^2 + \lVert{Bv}\rVert^2 + (i[A, B] v, v)
\end{align*}
The idea is to use the positivity of the principal symbol to deduce the positivity of the last term in the expression above. We first need to use a \textit{convexification} argument (see \cite{LCW}), where we slightly perturb $\varphi$ by a convex function. Namely, we consider a function $f: \mathbb{R} \to \mathbb{R}$ and the composition $\tilde{f} = f \circ \varphi$. Then we have:
\begin{itemize}
\item $P_{0, \tilde{f}} = \tilde{A} + i\tilde{B}$, according to the above decomposition.
\item $\nabla(f \circ \varphi) = f' \circ \varphi \nabla \varphi$.
\item $D^2(f \circ \varphi) = D(f' \circ \varphi d\varphi) = d(f' \circ \varphi) \otimes d\varphi + f' \circ \varphi D(d \varphi) = f'' \circ \varphi d\varphi \otimes d\varphi + f' \circ \varphi D^2 \varphi$, where we used the fact that $\varphi$ is a distance function.
\end{itemize}
Now we quote Lemma 2.3 from \cite{LCW}, which computes the Poisson bracket of the principal symbols of $A$ and $B$, which are respectively denoted as $a$ and $b$:
\begin{align*}
\{a, b\}(x, \xi) = 4D^2 \varphi(\xi^{\#}, \xi^{\#}) + 4D^2 \varphi(\nabla \varphi, \nabla \varphi)
\end{align*}
where we have the expressions $a = |\xi|^2 - |d \varphi|^2 = |\xi^{\# }|^2 - |\nabla \varphi|^2$ and $b = 2\langle{\nabla \varphi, \xi^{\#}}\rangle$. By $\alpha^{\#}$ we denote the unique element of $T_pM$ such that $\alpha(v) = \langle{\alpha^{\#}, v}\rangle$ for all $v$. With this notation, $a + ib = p_{\varphi}$ is the principal symbol of $P_{0, \varphi}$ in the standard semiclassical quantisation. Using the result of this lemma, we have for $m = 1$:
\begin{align*}
\{\tilde{a}, \tilde{b}\} (x, \xi) &= 4(f'' \circ \varphi) \langle{\nabla \varphi, \xi^{\#}}\rangle^2 + 4(f'' \circ \varphi) (f' \circ \varphi)^2 |\nabla \varphi|^4\\
&= 4(f'' \circ \varphi) (f' \circ \varphi)^2 + \underbrace{(f'' \circ \varphi)(f' \circ \varphi)^{-2}}_{\beta} \tilde{b}^2
\end{align*}
where $\tilde{b} = 2 \langle{d(f \circ \varphi), \xi}\rangle = 2(f' \circ \varphi) \langle{\nabla \varphi, \xi^{\#}}\rangle$. So, we must have
\begin{align*}
i[\tilde{A}, \tilde{B}] = 4h (f'' \circ \varphi)(f' \circ \varphi)^2 + h \tilde{B} \beta \tilde{B} + h^2 R
\end{align*}
where $R$ is first order semiclassical differential operator. Now we pick $f$ such that:
\begin{itemize}
\item $f(s) = s + \frac{h}{2\epsilon} s^2$, $f'(s) = 1 + \frac{h}{\epsilon} s$ and $f'' = \frac{h}{\epsilon}$.
\item Take $1 \geq \epsilon_0 \geq \frac{h}{\epsilon} > 0$ small enough such that $f' > 1/2$ on $\varphi(M)$ and denote $\varphi_{\epsilon} = f \circ \varphi$. One can check that the coefficients of $R$ are uniformly bounded with respect to $h$ and $\epsilon$, and $\beta = \frac{h/\epsilon}{(1 + \frac{h}{\epsilon}s)^2}$ is uniformly bounded.
\end{itemize}

Namely, one has:
\begin{align*}
\Big(i [\tilde{A}, \tilde{B}]v, v\Big) &= \Big(4 \frac{h^2}{\epsilon}(f' \circ \varphi)^2v + h\tilde{B} \big(\frac{h/\epsilon}{(f' \circ \varphi)^2} \tilde{B}v \big), v\Big) + h^2 \Big(Rv, v\Big)\\
&\geq \frac{h^2}{\epsilon} \lVert{v}\rVert^2 - C_0h \lVert{\tilde{B}v}\rVert^2 - C_0 h^2 \lVert{v}\rVert_{H^1_{scl}}\lVert{v}\rVert_{L^2}
\end{align*}
because $\lVert{Rv}\rVert \leq C_0 \lVert{v}\rVert_{H^1_{scl}}$. The previous inequality hold for $m > 1$, as $[\tilde{A}, \tilde{B}]$ acts diagonally, so $\Big(i[\tilde{A}, \tilde{B}] v, v\Big)_{L^2(\mathbb{C}^m)} = \sum \Big(i[\tilde{A}, \tilde{B}]v_j, v_j\Big)_{L^2}$.

Using the inequality $\lVert{h \nabla v}\rVert^2 \leq \lVert{P_{0, \varphi_{\epsilon}} v}\rVert^2 + C_1 \lVert{v}\rVert^2$, we conclude:
\begin{align}\label{njn2}
\Big(i[\tilde{A}, \tilde{B}] v, v\Big) \geq \frac{h^2}{\epsilon} (1 - C_4 \epsilon)\lVert{v}\rVert^2 - C_3 h \lVert{\tilde{B} v}\rVert^2 - C_3 \lVert{P_{0, \varphi_{\epsilon}} v}\rVert^2
\end{align}
by employing $\lVert{v}\rVert_{H^1_{scl}} = \lVert{v}\rVert_{L^2} + \lVert{h \nabla v}\rVert_{L^2} \leq C_1'\cdot \lVert{v}\rVert_{L^2} + \lVert{P_{0, \varphi_{\epsilon}} v}\rVert_{L^2}$ and AM-GM. Hence, we finally get the inequality:
\begin{gather}\label{njn1}
(1 + C_3) \lVert{P_{0, \varphi_{\epsilon}} v}\rVert^2 \geq \lVert{\tilde{A} v}\rVert^2 + (1 - C_3h)\lVert{\tilde{B} v}\rVert^2 + \frac{h^2}{\epsilon} (1 - C_4 \epsilon) \lVert{v}\rVert^2
\end{gather}

Let us now turn to the case $X \neq 0$ -- we want to incorporate it into the inequality \eqref{njn1} and to estimate it in a suitable way. Note that we have $h^2 X_{\varphi_{\epsilon}} = h e^{\varphi_{\epsilon}/h} X e^{-\varphi_{\epsilon}/h} = h^2 X - h f' \circ \varphi X \varphi$. Thus we have:
\begin{itemize}
\item $\lVert{h^2 X v}\rVert_{L^2} = \lVert{h^2 \langle{X, \nabla v}\rangle}\rVert_{L^2} \leq h |X|_{L^{\infty}} \lVert{h \nabla v}\rVert_{L^2} \leq C_2' \cdot h \lVert{v}\rVert_{H^1_{scl}}$
\item $\lVert{h(f' \circ \varphi) X (\varphi) v}\rVert_{L^2} \leq C_3'\lVert{v}\rVert_{L^2}$

\end{itemize}
By combining the two inequalities above, we conclude, by using \eqref{njn2}:
\begin{gather*}
\lVert{h^2 X_{\varphi_{\epsilon}} v}\rVert \lesssim h \lVert{v}\rVert_{H^1_{scl}} \lesssim h(\lVert{v}\rVert_{L^2} +\lVert{P_{0, \varphi} v}\rVert_{L^2})
\end{gather*}
which in turn implies the following chain of inequalities:
\begin{align*}
2(1 + C_3) \lVert{P_{0, \varphi_{\epsilon}}v + h^2X_{\varphi_{\epsilon}}v}\rVert^2 &\geq \frac{4}{3}(1+ C_3)\lVert{P_{0, \varphi_{\epsilon}}v}\rVert^2 - C_5 h^2 \Big( \lVert{v}\rVert^2 + \lVert{P_{0, \varphi_\epsilon} v}\rVert^2\Big)\\
&\geq \lVert{v}\rVert^2 \Big(\frac{h^2}{\epsilon}(1 - C_6\epsilon)\Big)
\end{align*}
where $C_6 = C_4 + C_5$. So for $\epsilon$ small enough, there exists $C_7$ such that:
\begin{gather}\label{convexificationineq}
C_7 \lVert{P_{0, \varphi_{\epsilon}} v + h^2 X_{\varphi_{\epsilon}} v}\rVert^2 \geq \frac{h^2}{\epsilon} \lVert{v}\rVert^2
\end{gather}
Therefore we have for $u = e^{-\varphi_{\epsilon}/h} v$: 
\begin{gather*}
C_{\epsilon} h^2 \lVert{e^{\frac{\varphi^2}{2 \epsilon}} e^{\frac{\varphi}{h}} (\Delta + X) u}\rVert^2 \geq \lVert{e^{\frac{\varphi^2}{2\epsilon}} e^{\frac{\varphi}{h}} u}\rVert^2
\end{gather*}
which together with $1 \leq e^{\frac{\varphi^2}{2 \epsilon}} \leq C_\epsilon'$ implies the result.
\end{proof}
\begin{rem}\rm(Carleman estimates with a boundary term).\label{partialboundary}
We record a corollary of the above inequality for functions not necessarily supported in the interior of our manifold; this extends the inequality (2.13) from \cite{MagU} to the higher rank case. Let $v \in C^{\infty}(M; \mathbb{C}^m) \cap H^1_0(M; \mathbb{C}^m)$ -- then we claim that the following inequality holds:
\begin{align}\label{Carlemanboundary0}
\lVert{v}\rVert^2_{H^1_{scl}} \lesssim h^2 \lVert{e^{\frac{\varphi}{h}} (\Delta + X + q)e^{-\frac{\varphi}{h}} v}\rVert^2 + h(\partial_\nu \varphi \partial_\nu v, \partial_\nu v)_{\partial M}
\end{align}
This is an exercise in partial integration and using the condition that $v|_{\partial M} = 0$ to get rid of the extra factors. Namely, what we get in the above notation is:
\begin{align*}
\lVert{(A + iB)v}\rVert^2 &= \lVert{Av}\rVert^2 + \lVert{Bv}\rVert^2 + i(Bv, Av) - i(Av, Bv)\\
&= \lVert{Av}\rVert^2 + \lVert{Bv}\rVert^2 + i([A, B]v, v) - 2h^3(\partial_\nu \varphi \partial_\nu v, \partial_\nu v)_{\partial M}
\end{align*}
by using $(AB v, v) - (Bv, Av) = -2i h^3 (\partial_\nu \varphi \partial_\nu v, \partial_\nu v)_{\partial M}$ and $(Bu, v) = (u, Bv)$ since $v$ vanishes at the boundary. For the proof of the first equality we use the Green's identity and for the second, we use the formula \eqref{greenidentity}. The proof then proceeds exactly the same way as before, by bounding the extra $X$ factor in the equation and using the positivity of $i([A, B]v, v)$.

Finally, let us recast the inequality \eqref{Carlemanboundary0} in the following form, by letting $u = e^{-\frac{\varphi}{h}} v$ and noticing that on $\partial M$ we have $\partial_\nu u = e^{-\frac{\varphi}{h}} \partial_\nu v$, since $v \in H^1_0(M)$:
\begin{multline}\label{Carlemanboundary1}
\lVert{e^{\frac{\varphi}{h}} u}\rVert_{H^1_{scl}(M; \mathbb{C}^m)} + \sqrt{h}\lVert{\sqrt{-\partial_\nu \varphi} e^{\frac{\varphi}{h}} \partial_\nu u}\rVert_{L^2(\partial M_-; \mathbb{C}^m)}\\
\lesssim h\lVert{e^{\frac{\varphi}{h}} (\Delta + X + q) u}\rVert_{L^2(M; \mathbb{C}^m)} + \sqrt{h}\lVert{\sqrt{\partial_\nu \varphi} e^{\frac{\varphi}{h}} \partial_\nu u}\rVert_{L^2(\partial M_+; \mathbb{C}^m)}
\end{multline}
where we use the notation $\partial M_{\pm} = \{x \in \partial M \mid \pm \partial_\nu \varphi (x) \geq 0\}$. By generalising appropriately, we have a version of this inequality for an arbitrary vector bundle on $M$.
\end{rem}
Now we turn to the proof of inequalities similar to the ones from Theorem \ref{specslucaj}, but with shifted indices of the Sobolev spaces, which is actually necessary to obtain the wanted solvability estimates. This is done using the semiclassical pesudodifferential calculus (see \cite{Zworski}).

Before we start, let us briefly introduce the Sobolev spaces for a real parameter, in a coordinate invariant way. This is described in more detail in \cite{aubin}. It is a known fact that the connection Laplacian on a compact Riemannian manifold (without boundary) is essentially self-adjoint on the dense subspace $C^{\infty}(N; E) \subset L^2(N; E)$ (more generally, this holds for any elliptic differential operator on $E$), meaning that the closure of $\Lapl_A$ is equal to the adjoint $\Lapl_A^*$. 

Then by applying the spectral theorem for unbounded densely defined operators and since $\Lapl_A$ is positive, we can define the semiclassical Bessel potentials $J_A^s = (1 - h^2\Delta_A)^{\frac{s}{2}}$ for $s \in \mathbb{R}$ (here $\Delta_A = -\Lapl_A$). The functional calculus from the spectral theorem also gives us that $J_A^sJ_A^t = J_A^{s + t}$ and $J_A^s$ commutes with any function of the connection Laplacian $\Lapl_A$. Moreover, it is well-known that a function of a semiclassical PDO is again a semiclassical PDO (see Chapter 8 in \cite{DS99}); thus $J_A^s$ is a semiclassical PDO of order $s$. Finally, we define the semiclassical Sobolev spaces $H^s_{scl}$ as the completion of the $C^{\infty}(N; E)$ in the norm given by:
\begin{align*}
\lVert{u}\rVert_{H^s_{scl}(N; E)} = \lVert{J_A^s u}\rVert_{L^2(N; E)}
\end{align*}
One can easily check that the dual of $H^s_{scl}(N; E)$ may be isometrically identified with the $H^{-s}_{scl}(N; E)$. Similarly, we may define the usual semiclassical Sobolev space, by introducing the semiclassical Bessel potentials $J^s = (1 - h^2\Delta)^{\frac{s}{2}}$ which define the spaces $H_{scl}^s(N)$; we extend $J^s$ to act diagonally on $C^\infty(N; \mathbb{C}^m)$.

Next, observe that we have the following commutator estimates for sections of $E$. Let $\psi$, $\chi \in C^{\infty}_c (N)$ with $\chi = 1$ near $\text{supp}(\psi)$ and consider any $s, \alpha, \beta \in \mathbb{R}$, and $K \in \mathbb{N}$ -- then we can find $C_K > 0$ such that:
\begin{align}\label{pseudolocalscl}
\lVert{(1 - \chi)J_A^s(\psi u)}\rVert_{H^{\alpha}_{scl}(N; E)} \leq C_K h^K \lVert{u}\rVert_{H^{\beta}_{scl}(N; E)}
\end{align}
This follows from the pseudolocality of the semiclassical PDOs and the mapping properties of semiclassical PDOs on Sobolev spaces. Moreover, we record another commutator estimate:
\begin{align}\label{commineq}
\lVert{[D, J_A^s]u}\rVert_{L^2(N; E)} \leq C h \lVert{u}\rVert_{H^{s}_{scl}(N; E)}
\end{align}
where $D$ is a first order, diagonal semiclassical differential operator in $E$ over $N$; this follows from the formula for the symbol of the commutator of two semiclassical PDOs (see \cite{Zworski}).


For what follows, assume that the LCW $\varphi$ is a smooth function in a neighbourhood of $\overline{U}$ and extend this function smoothly to $N$. We are now ready to shift the indices of the Sobolev estimates from Theorem \ref{specslucaj}:

\begin{theorem}
Under the assumptions of Theorem \ref{specslucaj} and given $s \in \mathbb{R}$, there exist constants $C_s$ and $h_s > 0$ such that for all $0 < h \leq h_s$ and $u \in C^{\infty}_c (M^{int}; \mathbb{C}^m)$:
\begin{gather*}
\lVert{e^{\frac{\varphi}{h}} u}\rVert_{H^{s+1}_{scl}(N; \mathbb{C}^m)} \leq C_s h \lVert{e^{\frac{\varphi}{h}}(\Delta + X + q)u}\rVert_{H^s_{scl} (N; \mathbb{C}^m)}
\end{gather*}
Moreover, there are corresponding constants such that for every $u \in C^{\infty}_c (M^{int}; E)$:
\begin{align*}
\lVert{e^{\frac{\varphi}{h}} u}\rVert_{H^{s+1}_{scl}(N; E)} \leq C_s' h \lVert{e^{\frac{\varphi}{h}}\Lapl_{A, Q}u}\rVert_{H^s_{scl} (N; E)}
\end{align*}
\end{theorem}
\begin{proof}
We closely follow the proof of Lemma 4.3 from \cite{LCW}. Let us introduce $P_{\varphi_\epsilon} = e^{\frac{\varphi_\epsilon}{h}} h^2(\Delta + X + q) e^{-\frac{\varphi_\epsilon}{h}}$ and let $\chi \in C_0^\infty(U)$ such that $\chi = 1$ near $M$; here $\varphi_\epsilon$ comes from the proof of Theorem \ref{specslucaj}. Then we have by \eqref{convexificationineq} and \eqref{pseudolocalscl}:
\begin{align*}
h\lVert{u}\rVert_{H^{s+1}_{scl}} &\leq h\lVert{\chi J^s u}\rVert_{H^1_{scl}} + h \lVert{(1 - \chi)J^s u}\rVert_{H^1_{scl}}\\
&\lesssim \sqrt{\epsilon}\lVert{P_{\varphi_\epsilon}(\chi J^s u)}\rVert_{L^2} + h^2 \lVert{u}\rVert_{H^{s+1}_{scl}}
\end{align*}
which means that the second term may be absorbed to the left hand side for small $h$. Furthermore, for some $\chi' \in C_0^\infty(U)$ with $\chi' = 0$ near $M$, by \eqref{convexificationineq} again:
\[\lVert{[P_{\varphi_\epsilon}, \chi]J^s u}\rVert_{L^2} = \lVert{[P_{\varphi_\epsilon}, \chi] \chi' J^s u}\rVert_{L^2} \lesssim \lVert{\chi' J^s u}\rVert_{H^1_{scl}} \lesssim h^2 \lVert{u}\rVert_{H^{s+1}_{scl}}\]
so after absorbing the remaining factors, we have:
\[h\lVert{u}\rVert_{H^{s+1}_{scl}} \lesssim \sqrt{\epsilon} \lVert{J^s P_{\varphi_\epsilon} u}\rVert_{L^2} + \sqrt{\epsilon} \lVert{[P_{\varphi_\epsilon}, J^s] u}\rVert_{L^2}\]
The first term gives the right bound; for the second one, by expanding the operator and putting $X_{\varphi_\epsilon} = e^{\frac{\varphi_\epsilon}{h}} X e^{-\frac{\varphi_\epsilon}{h}}$, we have:
\[P_{\varphi_\epsilon} = h^2 \Delta - |d\varphi_\epsilon|^2 + 2 \langle{d \varphi_\epsilon, h d(\cdot)}\rangle - h \Delta \varphi_\epsilon + h^2 X_{\varphi_\epsilon} + h^2 q =: h^2\Delta + P_1\]
Since $[h^2\Delta, J^s] = 0$ and since $J^s$ acts diagonally, by the composition formula we have $[J^s, P_1] = hR_1$ where $R_1$ a semiclassical PDO of order $s$. Thus by taking $\epsilon$ to be small enough (and such that $h \leq \epsilon \epsilon_0$), we may absorb this remainder to the left hand side.

For an arbitrary vector bundle, note that all the steps above work the same with $J^s_A$ instead of $J^s$, until the estimate for $\lVert{[P_{\varphi_\epsilon}, J_A^s] u}\rVert_{L^2}$. In local coordinates, we have the expansion
\[e^{\frac{\varphi_\epsilon}{h}} h^2 \Lapl_{A, Q} e^{-\frac{\varphi_\epsilon}{h}} = h^2 \Lapl_A - \underbrace{(|d\varphi_{\epsilon}|^2 - 2 \langle{d \varphi_\epsilon, h d(\cdot)}\rangle + h \Delta \varphi_\epsilon)}_{D} + \underbrace{2h\langle{A, d\varphi_\epsilon}\rangle}_{P_2}\]
where $D$ is a diagonal first order semiclassical differential operator. Now observe that $[\Lapl_A, J^s_A] = 0$ and also that locally the symbol of $[D, J^s_A]$ is in $hS^s$, and so is the symbol of $[P_2, J^s_A]$. This implies that $[-D + P_2, J^s_A]$ is in $h \Psi^s(M)$, which makes us able to absorb the extra factor for small enough $\epsilon$ and finish the proof.
\end{proof}

Essentially the only case that we will use in the previous theorem is the case $s = - 1$; it appears that it is necessary in the following result, to establish the existence of an $H^1$ solution to our equation with a suitable norm estimate (otherwise, with Theorem \ref{specslucaj} we would only get solutions in $L^2$ with bounds in $H^{-1}$ norm). It is left without a proof, since it is well-known and formally follows from the scalar case in Theorem 4.4 in \cite{LCW}.
\begin{theorem}\label{carleman}
Given a connection $A$ and an endomorphism $Q$ of $E$, there exists a positive constant $h_0$ such that for any $0 < h \leq h_0$ and any section $f \in L^2(M; E)$, there exists a solution $u \in H^1(M; E)$ to the equation $e^{\frac{\varphi}{h}}\Lapl_{g, A, Q} e^{-\frac{\varphi}{h}} u = f$ satisfying:
\begin{align*}
\lVert{u}\rVert_{H^1_{scl}(M; E)} \leq Ch \lVert{f}\rVert_{L^2(M; E)}
\end{align*}
\end{theorem}

\section{The CGO construction for the case of simple manifolds}\label{CGOsimple}

In this section, we construct the special CGO solutions of the form $u = e^{-\frac{\Psi}{h}}(a + r)$ (for suitable $\Psi$, $a$ and $r$) to the connection Laplacian equation $\Lapl_A(u) = 0$, in the particular case when the transversal manifold is simple. In this case, we have an easy ansatz to the transport and the eikonal equation, so we get away without using the construction of Gaussian Beams in Section \ref{sec5}. The purpose of this is to reduce Conjecture \ref{conjecture1} in this case to a new non-abelian ray transform -- see Question \ref{conjecture2}.

Throughout the section, we will be working in the following setting: $M$ is an $n$-dimensional compact manifold with boundary, $E = M \times \mathbb{C}^m$ is the trivial vector bundle of rank $m$ with the standard fibrewise Hermitian inner product, $A$ a unitary connection on it and $Q$ an $m$ by $m$ matrix potential (section of $\text{End}(E)$). Furthermore, our assumption will be that $M_0$ is simple and that $M$ is isometrically embedded inside the manifold of the same dimension $\mathbb{R} \times M_0$, with the product metric $g = e \oplus g_0$.

Recall that the manifold $M_0$ is \textit{simple} if the exponential map $\text{exp}_p: \text{exp}_p^{-1}(M_0) \to M_0$ is a diffeomorphism for every point $p \in M_0$ and the boundary $\partial M$ is strictly convex. Simplicity of $M_0$ is a natural assumption and many questions about the X-ray transform are posed in this setting. 

We start with stating an identity which will be useful for identifying different parts of the CGO solution. The proof is left as an exercise.



\begin{lemma}\label{conjugate}
The following identity holds, for $s \in \mathbb{C}$, $\rho$ a smooth function on $M$, $u$ a section of $E$, $X$ a smooth $m \times m$ matrix with entries as vector fields and $q$ a smooth $m \times m$ matrix potential:
\begin{align*}
e^{-s\rho}\Big(\Delta + X + q\Big)e^{s\rho}u = (\Delta + X + q)u + s \Big((\Delta \rho)u + X(\rho)u - 2\langle{\nabla \rho, \nabla u}\rangle\Big) - s^2|d\rho|^2u
 \end{align*}
\end{lemma}
Now plugging in the specific form of the solution as above $u = e^{-\frac{\Psi}{h}}(a + r)$ to the equation $h^2 \Lapl_{A, Q} u = 0$ ($a$ and $r$ are $\mathbb{C}^m$-valued, $\Psi$ a complex function) and using Lemma \ref{conjugate}, we get three equations:
\begin{align}
|d\Psi|^2 &= 0 \label{eikonal}\\
-2 \langle{d\Psi, da}\rangle + (\Delta \Psi)a + X(\Psi)a &= 0\label{transport}\\
e^{\frac{\Psi}{h}} \Lapl_{A, Q} e^{-\frac{\Psi}{h}} r & = - \Lapl_{A, Q} a \label{inhomog}
\end{align}

\noindent where the first two of the them correspond to the dominating factors (the coefficients next to $h^0$ and $h^1$, respectively) when $h \to 0$ and the last one makes sure we get an exact solution and solves for the residue. The notation $\langle{d\Psi, da}\rangle$ means that we consider the vector formed by taking the inner product of each component of $da$ with $d \Psi$. Recall that $X = -2g^{ij}A_i \frac{\partial}{\partial x^j}$ is derived in \eqref{notation} from the pair $(A, Q)$.

\subsection{Eikonal equation}\label{4.1}

This is the equation \eqref{eikonal} above. Recall that in this case the operation $|\cdot|$ is just a complex bilinear form obtained by extending the Riemannian real inner product. Thus, if we write $\Psi = \varphi + i \psi$, the equation can be rewritten as:
\begin{equation}
|\nabla \psi |^2 = |\nabla \varphi |^2, \quad \langle{\nabla \psi, \nabla \varphi}\rangle = 0
\end{equation}
Here we let $\varphi$ to be the LCW given by $\varphi (x) = x_1$.  With this special choice for $\varphi$, our equations become simple: 
\begin{equation}\label{eikonal1}
|\nabla \psi | = 1, \quad \frac{\partial \psi}{\partial x_1} = 0
\end{equation}
because of the splitting of the metric in $\mathbb{R} \times M_0$. Here we will fix a polar coordinate system: we pick a point $\omega \in M_0$ such that $(x_1, \omega)$ is not in $M$ for any $x_1$. We can always do this if we enlarge $M_0$ slightly at the beginning, keeping the metric simple (this is always possible -- see \cite{LCW}), to some manifold $D$ such that:
\begin{align*}
(M, g) \Subset (\mathbb{R} \times M_0, g) \Subset (\mathbb{R} \times D, g)
\end{align*}
We then use the geodesic polar coordinate system to get a coordinate chart $(x_1, r, \theta)$ for $\theta \in S^{n-2}$, to cover $\mathbb{R} \times M_0$, in which the metric has a nice form.

One can then check that $\psi = r$ solves \eqref{eikonal1} and in this case $\Psi = x_1 + ir$ (note that the solution depends on $\omega$). Observe that we could have chosen $\varphi = -x_1$, in which case $\Psi = -x_1 + ir$ works equally well. This will be useful when we plug the solutions into our identity in Theorem \ref{identity}, so that the exponential parts cancel in the product.

\subsection{Transport equation}
This is the equation \eqref{transport}. We now proceed to the calculation of the three terms in this equation, taking $\Psi = x_1 + ir$ for the solution of the eikonal equation. We get the expressions:
\begin{gather*}
\langle{d\Psi, da}\rangle  = \frac{\partial \Psi}{\partial x_1} \frac{\partial a}{\partial x_1} + \sum_{j, k\geq 2}g^{jk}\frac{\partial \Psi}{\partial x^j} \frac{\partial a}{\partial x^k} = \Big(\frac{\partial}{\partial x_1} + i \frac{\partial}{\partial r}\Big) a\\
\Delta \Psi = -|g|^{-1/2}\Big(\sum_{j, k\geq 1} \frac{\partial}{\partial x^j} \big(|g|^{1/2} g^{jk} \frac{\partial \Psi}{\partial x^k}\big)\Big) = - |g|^{-1/2} \Big(\frac{\partial}{\partial x_1} + i \frac{\partial}{\partial r}\Big) (|g|^{1/2})\\
X(\Psi) = -2\Big(\sum_{j, k\geq 1}g^{jk}A_j \frac{\partial}{\partial x^k}(x_1 + ir)\Big) = -2(A_1 + iA_r)
\end{gather*}
Here $A_1$ and $A_r$ are the $dx_1$ and $dr$ components of $A$, respectively and we are taking the $(x^2, \dotso, x^n)$ coordinates on $M_0$, where $x^2 = r$. We set $z = x_1 + ir$ and so we define the complex derivatives as $\frac{\partial}{\partial \bar{z}} = \frac{1}{2}\Big(\frac{\partial}{\partial x_1} + i \frac{\partial}{\partial r}\Big)$ and $\frac{\partial}{\partial z} = \frac{1}{2}\Big(\frac{\partial}{\partial x_1} - i \frac{\partial}{\partial r}\Big)$. Then the equation \eqref{transport} takes the form:
\begin{equation}
4\frac{\partial a}{\partial \bar{z}} + 2|g|^{-1/2} \frac{\partial}{\partial \bar{z}}\big(|g|^{1/2}\big)a + 2(A_1 + iA_r)a = 0
\end{equation}
By introducing an integrating factor and using the substitution $b = a |g|^{1/4}$, we get the following nicer form:
\begin{equation}\label{transport1}
\frac{\partial b}{\partial \bar{z}} = -\frac{1}{2}(A_1 + iA_r)b
\end{equation}
Analogously, using the other solution $\Psi = -x_1 + ir$ of the eikonal equation, we get: 
\begin{equation}\label{transportxxx}
\frac{\partial b}{\partial z} = \frac{1}{2}(-A_1 + iA_r)b
\end{equation}
Since \eqref{transportxxx} can be obtained from \eqref{transport1} by conjugation, we will focus only on the latter. Actually we consider a slightly more general equation:
\begin{equation}\label{transport2}
\frac{\partial C}{\partial \bar{z}}  = BC
\end{equation}
where $C(\theta, x_1, r)$ is a smooth $m$ by $m$ matrix function and we denoted $B = -\frac{1}{2}(A_1 + iA_r)$. We impose one additional condition that $C$ should be invertible. Such a matrix $C$ will play an important role and we will need the solution on an open bounded subset of the plane, depending smoothly on $\theta$.


If one is interested in solving this equation on the whole domain of $\mathbb{C}$, a natural boundary condition would be to have $C$ approaching the identity at $\infty$; however this might be impossible -- see \cite{Eskin} for the proof of existence of a $C$ which has polynomial growth.

For $m = 1$, we may solve \eqref{transport2} by substituting the exponential function $C = e^\Phi$ and then using the Cauchy operator to solve $\bar{\partial} \Phi = B$.\footnote{Another way to solve $\bar{\partial} \Phi = B$ is to recall the fundamental solution $\frac{1}{\pi z}$ of the Cauchy-Riemann operator $\frac{\partial}{\partial \bar{z}}$ that satisfies $\frac{\partial}{\partial \bar{z}} \frac{1}{\pi z} = \delta$, where $\delta$ is the Dirac delta; then the convolution $\Phi = \frac{1}{\pi z} * B$ is a solution of $\frac{\partial}{\partial \bar{z}} \Phi = B$ (here $B$ has compact support). This is just a restatement of the generalised Cauchy integral formula that is being referred to in the text, which gives: $\Phi(\omega) = \frac{1}{2\pi i} \int_\mathbb{C} \frac{B(z)}{z - \omega} dz d\bar{z}$.} However, for $m > 1$ the situation complicates, so we give one proof of existence in the next subsection and a brief overview of other approaches. Given a matrix $C$ solution of \eqref{transport2}, one solution of the transport equation \eqref{transport1} is given by $a = C h$, where $h$ is holomorphic in each coordinate.


%

\subsection{Complex geometric approach to the construction of the solution to transport equation}\label{paramsoln}
Using some standard theory of holomorphic vector bundles one can describe a solution to the transport equation \eqref{transport2} in a geometric way. References are books by Kobayashi \cite{kobayashi_cvb} (Propositon 3.7) and Foster \cite{forster} (Theorem 30.1).

\begin{theorem}
Let $E$ be a $C^{\infty}$ complex vector bundle over a complex manifold $M$. Then if $D$ is a connection on $E$ such that $D'' \circ D'' = 0$, then there exists a unique holomorphic vector bundle structure on $E$ such that $D'' = d''$.
\end{theorem}

\begin{theorem}\label{Forster}
Let $X$ be an open Riemann surface and $E$ a holomorphic vector bundle over $X$, of rank $m$. Then $E$ is trivial, i.e. there exists a set of holomorphic sections $s_i$, $i = 1, \dots, m$ such that they span $E_p$ for each point $p$ in $X$.
\end{theorem}

In the former theorem, by $D''$ we mean the $(0, 1)$ component of the connection derivative and by $d'' = \bar{\partial}$ the $(0, 1)$ component of the exterior derivative. 

\begin{theorem}\label{geointerpret}
Let $\Omega \subset \mathbb{C}$ be an open subset of the complex plane and let $E = \Omega \times \mathbb{C}^m$, equipped with a connection $D$. Then there exists a smoothly varying invertible matrix $F$ such that $\bar{\partial} F = -F A_{0, 1} $, where $A_{0, 1}$ is the $(0, 1)$ part of the connection matrix of $D$. In particular, for any matrix $B$, there exists an invertible, smoothly varying matrix $C$ such that $\frac{\partial C}{\partial \bar{z}} = BC$.
\end{theorem}
\begin{proof}
The proof relies on the previous two theorems; namely, we automatically have $D'' \circ D'' = 0$ by dimension. Thus, there exists a holomorphic structure on $E$ such that $D'' = d''$. Although our vector bundle is smoothly trivial, we do not know if it is holomorphically trivial -- this is given by Theorem \ref{Forster}. Thus, there exists a set of holomorphic trivialisations $s_i$, $i = 1, \dots, m$ such that they are linearly independent at each point of $\Omega$; in these new coordinates, we also have $D'' = d''$. In other words, there exists a smoothly (not necessarily holomorphically) varying matrix $F: \Omega \to GL(m, \mathbb{C})$ such that, $s_i = F e_i$, where $e_i$ is our standard global frame of $E$. Then we have the change of basis law for connections: 
\begin{equation}
0 = \bar{\partial} s_i = D'' s_i = D''(F e_i) = \bar{\partial} F e_i + F D'' e_i = \bar{\partial} F e_i + F A_{0, 1} e_i
\end{equation}
for all $i = 1, \dots, m$. Thus we get, in matrix form:
\begin{equation}
\bar{\partial} F = - F A_{0, 1}
\end{equation}
By picking the $(0, 1)$ part of the connection matrix to be $Bd\bar{z}$, and letting $C = F^{-1}$, we get $\frac{\partial C}{\partial \bar{z}} = BC$.
\end{proof}

\begin{rem}\rm
We digress slightly to note that there are examples of smoothly trivial holomorphic line bundles, but \emph{not holomorphically} trivial. The long exact sequence associated to the short exact sequence $0 \to 2\pi i \mathbb{Z} \to \mathcal{O} \to \mathcal{O^*} \to 0$ (here $\mathcal{O}$ and $\mathcal{O^*}$ are the sheaves of holomorphic and nowhere vanishing holomorphic functions, respectively) that the map $c_1: \text{Pic}(M) \to \mathbb{Z}$ given by the first Chern class has a non-trivial kernel over a surface of positive genus $M$ (Pic($M$) is the holomorphic Picard group).
\end{rem}

Theorem \ref{geointerpret} provides us with a geometric interpretation of \eqref{transport2} for a fixed $\theta$. In order to solve this equation smoothly in $\theta$, we need to go through the proof of trivialising a family of holomorphic vector bundles parametrically. We will not do this here, since there are already a few proofs of existence of such parametric solutions present in other sources.

Let us give a brief overview of proofs of existence of (invertible) solutions to the above equation we found in literature. As mentioned, Eskin \cite{Eskin} gives us $C$ depending smoothly on a parameter, with polynomial growth as $|z| \to \infty$. A more concise proof is given by the same author and Ralston in Theorem 4, \cite{CRsystems} ($Y = S^{n-2}$ in our case) -- it relies on solving the equation locally in $z$ using the Cauchy operator to transform it to an integral equation and then gluing these local solutions together using the Cartan's lemma. Finally, Nakamura and Uhlmann \cite{nakamura_uhlmann} also provide us with another method.


\subsection{The inhomogeneous part}

Here we deal with the third equation set out above, the equation \eqref{inhomog}. With the Carleman estimates established so far, we can easily construct the residue with the wanted estimates -- we just use Theorem \ref{carleman} to solve for the $h$-dependent residue $r_h$ (note the distinction between the radial variable $r$ and the function $r_h$), such that $\lVert{r_h}\rVert_{L^2(M; E)} = O(h)$ and $\lVert{r_h}\rVert_{H^1(M; E)} = O(1)$; equivalently $\lVert{r_h}\rVert_{H^1_{scl}(M; E)} = O(h)$.


\subsection{Consequences of the CGO construction and recovering the connection}\label{sec4.5}

In this section, we use the previously obtained CGO solutions to deduce some new information from the equality of the DN maps. Reducing to an X-ray transform or asking for injectivity of some other transform is often the way to make the final step in solving inverse problems: see \cite{Pat, LCW, CTA, chung} for examples of such results for the X-ray transform or \cite{MagU} for an example of the Radon transform on planes; this is the viewpoint we will take. 

We equip $E = M \times \mathbb{C}^m$ with two potentials $Q_{1, 2}$ and unitary connections $A_{1, 2}$; we assume that $\Lambda_{A_1, Q_1} = \Lambda_{A_2, Q_2}$. It is technically easier to consider the endomorphism bundle $E' = M \times \mathbb{C}^{m \times m}$ and extend the action of $\Lapl_{A_1, Q_1}$ and $\Lapl_{A_2, Q_2}$ in the trivial way to sections of $E'$ (by matrix multiplication). So we consider matrix solutions $U_1$ and $U_2$ to $\Lapl_{A_1, Q_1} U_1 = 0$ and $\Lapl_{A_2, Q_2^*} U_2 = 0$, constructed by our work in previous subsections, which are of the form:
\begin{align*}
U_1 & = e^{-\frac{x_1 + ir}{h}} \big(|g|^{-1/4} C_1 H(x_1, r) b(\theta) + R_1\big)\\
U_2 &= e^{-\frac{-x_1 + ir}{h}} \big(|g|^{-1/4} C_2 + R_2\big)
\end{align*}
where $H$ a holomorphic matrix, $b$ is a smooth function and we have the estimates $\lVert{R_1}\rVert_{H^1_{scl}(M; E')} = O(h)$ and $\lVert{R_2}\rVert_{H^1_{scl}(M; E')} = O(h)$. The invertible matrices $C_i$ are given by solving the transport equations \eqref{transport1} and \eqref{transportxxx} in the matrix form:
\begin{equation}\label{trans}
\frac{\partial C_1}{\partial \bar{z}} = -\frac{1}{2}\big((A_1)_1 + i(A_1)_r\big)C_1 \quad \text{and} \quad \frac{\partial C_2}{\partial z} = \frac{1}{2}\big(-(A_2)_1 + i(A_2)_r\big)C_2
\end{equation}
We wish to plug these in the identity obtained in Theorem \ref{matidentity}. Note that we have:
\begin{align*}
dU_1 &= e^{-\frac{x_1 + ir}{h}} \Big(-\frac{dx_1 + i dr}{h}\big(|g|^{-\frac{1}{4}} C_1 H b + R_1\big) + d(|g|^{-\frac{1}{4}}C_1 H b) + d(R_1)\Big)\\
dU_2^* &= e^{\frac{x_1 + ir}{h}} \Big(\frac{dx_1 + i dr}{h}\big(|g|^{-\frac{1}{4}} C_2^* + R_2^*\big) + d(|g|^{-\frac{1}{4}}C_2^*) + d(R_2^*)\Big)
\end{align*} 
Therefore, in the limit $h \to 0$, for $\tilde{A} = A_2 - A_1$:
\begin{equation*}
\lim_{h \to 0} h\big(U_1 (dU_2^*) - (dU_1) U_2^*,  \tilde{A}\big)_M = -2 \int_M{\Tr\big(|g|^{-\frac{1}{2}} b(\theta) C_1 H C_2^* (\tilde{A}_1 + i \tilde{A}_r)\big)} dV_g
\end{equation*}
by using Cauchy-Schwartz and the bounds we have on the $\lVert{R_1}\rVert_{H^1_{scl}}$ and  $\lVert{R_2}\rVert_{H^1_{scl}}$, along with the fact that everything else is uniformly bounded. Moreover, since the $A_i$ and $Q_i$ are bounded for $i = 1, 2$ and the exponential parts of $U_1$ and $U_2^*$ cancel, the first integral in the identity is equal to $O(1)$. Thus we get, by taking the limit $h \to 0$:
\begin{equation}
\int_M{|g|^{-\frac{1}{2}} b(\theta) \Tr\big(C_1 H C_2^* (\tilde{A}_1 + i \tilde{A}_r)\big)} dV_g = 0
\end{equation}
where $dV_g$ is the volume form. Since $dV_g = |g|^{1/2} dx_1 dr d\theta$ and since we can vary $b$ so that it approximates the delta function $\delta_\eta$ for some fixed angle $\eta$, by rearranging the terms in the trace bracket we obtain:
\begin{equation}\label{gen1}
\int_{M_0^{\eta}} \Tr \Big( H C_2^* (\tilde{A}_1+ i\tilde{A}_r) C_1\Big) dz \wedge d\bar{z} = 0
\end{equation}
where $z = x_1 + ir$ and $M_0^{\eta} := [-N, N] \times M_0 \cap \{\theta = \eta \}$, for some large $N$ (we also have $M^{\eta} = M \cap \{\theta = \eta\}$ is a 2-dimensional smooth manifold for almost all $\eta$ by Sard's theorem; the previous integral can be made over such $M^\eta$, too), such that $[-N, N] \times D$ contains a neighbourhood of $M$.

Here, we extended the connections $A_1$ and $A_2$ to the outside of $M$ (whole of $\mathbb{R} \times M_0$), such that they are unitary, compactly supported and such that $A_1 = A_2$ outside $M$. This is allowed by boundary determination, which gives us that the full jets of $A_1$ and $A_2$ are the same in suitable gauges. 




Now, by using the equations \eqref{trans}, we get:
\begin{equation}
\frac{\partial}{\partial \bar{z}} \big(C_2^* C_1\big) = \frac{1}{2} C_2^* (\tilde{A}_1 + i \tilde{A}_r)C_1
\end{equation}
where we also used that $A_i$s are skew-Hermitian. By substituting $C_0H$ in place of $H$ in the identity \eqref{gen1}, where $C_0$ is a constant matrix and $H$ holomorphic, and by varying the entries of $C_0$ we obtain:
\begin{equation}
\int_{M_0^{\eta}} H \frac{\partial}{\partial \bar{z}}\big(C_2^* C_1\big) dz \wedge d\bar{z} = 0
\end{equation}
and therefore by Stokes' theorem, we get:
\begin{equation}\label{integraleq}
\int_{\partial M_0^{\eta}} H C_2^* C_1dz = 0
\end{equation}
Note that $H$ is an arbitrary holomorphic matrix, i.e. $\frac{\partial H}{\partial \bar{z}} = 0$ and that the order in which we take matrix multiplication inside the integral is important.
%

We would now like to deduce a suitable transport equation on $\mathbb{R} \times SM_0$ and try to solve the problem from there.

\begin{figure}[h]
   \centering
    \includegraphics[width=1\textwidth]{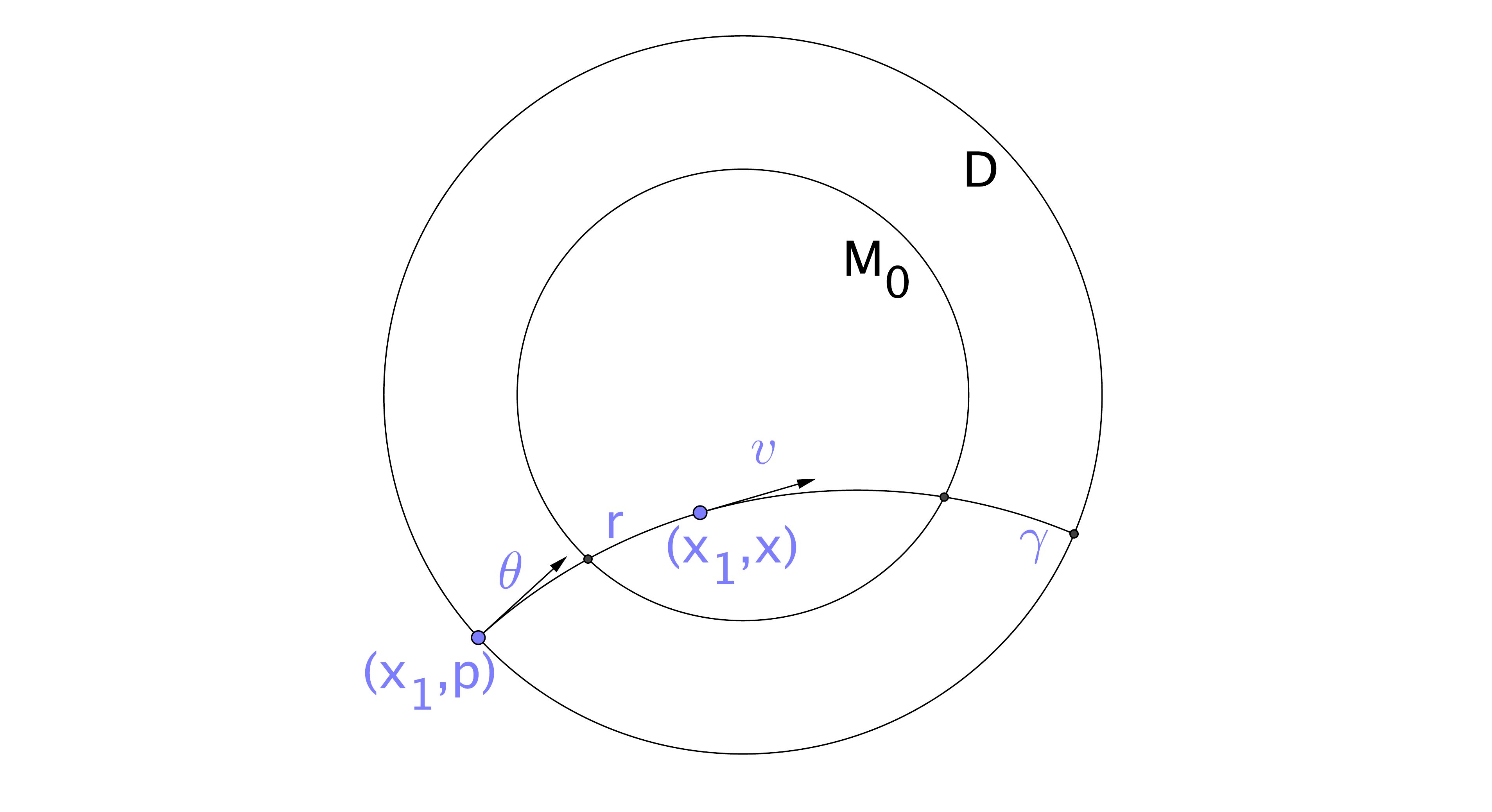}
            \captionsetup{format=hang, labelfont=sc}
    \caption{The construction of the matrix function $G(x_1, x, v)$: we are following the geodesic $\gamma$ at $(x_1, x)$ in the direction $-v$ up to the point $(x_1, p)$ on $\partial D$. The angle $\theta$ denotes the polar coordinate of the point $x$ with centre at $p$; distance between $x$ and $p$ is $r$.}
    \label{simplepic}
\end{figure}

Recall from Section \ref{4.1} the enlarged simple manifold $D$, which contains $M_0$. As we go along $\partial D$ and follow the tangent vectors, we obtain families of geodesics on $M_0$. Let us denote by $C_1(p, \theta, x_1, r)$ and $C_2(p, \theta, x_1, r)$ the solutions to equations \eqref{trans}, where $p$ denotes the point of the origin of the polar coordinate system. As explained previously in Subsection \ref{paramsoln}, we may construct solutions to \eqref{trans} depending smoothly on a parameter, giving $C_1$ and $C_2$ smooth as we vary $(p, \theta, x_1, r)$. 

Now given any $(x_1, x) \in \mathbb{R} \times M_0$ and $v \in S_x M_0$ a unit tangent vector, we may trace backwards the geodesic $\gamma$ starting at $(x_1, x)$ with speed $v$ (or go forwards in time with the geodesic with speed $-v$), until we hit $\partial D$; call this point $(x_1, p)$ -- see Figure \ref{simplepic}. Since $D$ is simple, we have the smooth dependence $p = p(x, v)$. Define 
\[G (x_1, x, v) = C_1 (p, \theta, x_1, r) C_2^*(p, \theta, x_1, r)\] 
where $r$ is the length along $\gamma$ from $(x_1, p)$ to $(x_1, x)$, $\theta$ is the coordinate of $(x_1, x)$ in the polar coordinate system (i.e. $\dot{\gamma}$ at the point $(x_1, p)$). Again since $D$ is simple we have the smooth dependence $\theta = \theta(x, v)$, which implies that $G$ is smooth. Therefore, we obtain a smooth matrix function $G$ (section of $E'$) on $\mathbb{R} \times SM_0$, where $SM_0$ denotes the unit sphere bundle. By the previous analysis, we have an equation for $G$: 
\begin{equation*}
2\frac{\partial G}{\partial \bar{z}} = -\big((A_1)_1 + i(A_1)_r\big)G + G\big((A_2)_1 + i(A_2)_r\big)
\end{equation*}
on the planes which are generated by the $x_1$ direction and a geodesic, i.e. by setting $\theta$ to be constant for a given $p \in \partial D$. From the previous equation we easily deduce that we have globally:
\begin{equation}\label{ray}
\big(\frac{\partial}{\partial x_1} + i X(x, v)\big)G = -A_1\big(\frac{\partial}{\partial x_1} + iv\big)G + GA_2\big(\frac{\partial}{\partial x_1} + iv\big)
\end{equation}
for all $x \in M_0$, $v$ unit tangent vectors in $S_xM_0$ and $x_1 \in \mathbb{R}$; $X(x, v)$ is the geodesic vector field on $SM_0$. Let us make a shorthand notation for the complex vector field $\mathbb{X}(x_1, x, v) = \frac{\partial}{\partial x_1} + iX(x, v)$.

First of all, let us see what information our integral equation \eqref{integraleq} gives us. We will need the following standard result:

\begin{lemma}\label{restrictionlemma}
Let $\Omega \subset \mathbb{C}$ be a domain with smooth boundary and let $f$ be a smooth function on $\partial \Omega$. Then $f$ is a restriction of a holomorphic function $h$ on $\Omega$, i.e. $f = h|_{\Omega}$ if and only if
\begin{align*}
\int_{\partial \Omega} g(z) f(z) dz = 0
\end{align*}
for all holomorphic functions $g$ on $\Omega$, which have a continuous extension to $\bar{\Omega}$.
\end{lemma}

The proof of this Lemma uses the  Plemelj-Sokhotski-Privalov formula and it follows from the proof of Lemma 5.1 in \cite{MagU}. As an application of this result, we have:
\begin{lemma}\label{holo_restrict}
There exists a holomorphic, invertible matrix function $F$ on $M_0^{\eta}$, such that $F^{-1}|_{\partial M_0^{\eta}} = C_2^* C_1|_{\partial M_0^{\eta}}$.
\end{lemma}
\begin{proof}
By applying Lemma \ref{restrictionlemma} to the equation \eqref{integraleq}, we deduce there exists a holomorphic matrix function $F'$, such that $F'|_{\partial M_0^{\eta}} = C_2^* C_1|_{\partial M_0^{\eta}}$. We need to prove $F'$ is invertible on $M_0^{\eta}$. 

Firstly, $M_0^{\eta} = [-N, N] \times [0, L]$ in local coordinates ($L$ is the length of the segment of the unit speed geodesic starting at a point $p \in \partial D$, which lies in $M_0$), which is simply-connected. Therefore, since $\det(C_1) \neq 0$ on $M_0^{\eta}$, it is a standard fact that $\det(C_1)$ admits a logarithm: we have a smooth function $\Phi_1$ on $M_0^{\eta}$ such that $\det(C_1) = e^{\Phi_1}$ and similarly we have $\Phi_2$ such that $\det(C_2) = e^{\Phi_2}$. From this, we infer that the variation of the argument of $\det(F')|_{\partial M_0^{\eta}} = e^{\Phi_1 + \overline{\Phi}_2}|_{\partial M_0^{\eta}}$ is zero, since $\Phi_1$ and $\Phi_2$ are honest functions. Therefore, by the argument principle applied to the holomorphic function $\det(F')$, we conclude $F'$ is invertible on the whole of $M_0^{\eta}$. By setting $F = (F')^{-1}$, we are done.\footnote{Moreover, one can show that $\frac{\partial \Phi_1}{\partial \bar{z}} = -\frac{1}{2}\Tr\big((A_1)_1 + i(A_1)_r\big)$ and $\frac{\partial \overline{\Phi}_2}{\partial \bar{z}} = \frac{1}{2}\Tr\big((A_2)_1 + i(A_2)_r\big)$, but we will not need this here.}
\end{proof}

More generally, we have such $F$ depending smoothly on the parameters in the influx boundary $(p, \theta) \in \partial_+ SD$ so we obtain a smooth matrix function $F$ on $[-N, N] \times SM_0$\footnote{The Plemelj-Sokhotski-Privalov formula actually gives $F^{-1}(z) = \frac{1}{2\pi i} \int_{\partial M_0^{\eta}}{\frac{C_2^*(\zeta)C_1(\zeta)}{\zeta - z} d\zeta}$.} such that $F|_{[-N, N] \times  \partial SM_0} = (C_2^*C_1)|_{[-N, N] \times \partial SM_0}$ and $\mathbb{X}(F) = 0$. Then we can redefine the solution $C_2$ to equations \eqref{trans} (parametrised by $(p, \theta) \in \partial_+ SD$), by setting $C_2' = C_2 F^*$. The transport equations will be satisfied again, but more importantly, we must have \eqref{ray} fulfilled with the new $G' (x_1, x, v) = C_1 (p, \theta, x_1, r) (C_2')^*(p, \theta, x_1, r)$ defined analogously as before and:
\begin{equation*}
G'|_{[-N, N] \times \partial SM_0} = (C_1C_2'^{*})|_{[-N, N] \times \partial SM_0} = (C_1F C_2^*)|_{[-N, N] \times \partial SM_0} = Id|_{[-N, N] \times \partial SM_0}
\end{equation*}
by the definition of $F$. Let us relabel the $G'$ back to $G$.

Let us now consider a reduction of the problem to a convex region, i.e. a larger manifold with certain properties. We take $M'$ to be a slightly smaller manifold than $[-N, N] \times M_0$ with corners smoothed out -- for example, we may take a compact simple manifold with boundary $M_0' \subset M_0^\circ$, such that the interior of $[-N, N] \times M_0'$ contains $M$, and take $M'$ to be a smoothed out version of this. Hence $M'$ is homeomorphic to a ball, and the exterior of $M'$ in $[-N, N] \times M_0$ is homeomorphic to an $n$-dimensional annulus. Now we can make the following reduction:

\begin{prop}[Reduction to the convex case]\label{reductionconvex}
Let $U$ and $V$ be two sections of $C^\infty(M';$\\$ \mathbb{C}^{m \times m})$ which solve solve $\Lapl_{A_1, Q_1}U = 0$ and $\Lapl_{A_2, Q_2^*}V = 0$. Then we have 
\begin{align*}
    \int_{M'} \big\langle{(Q_1 - Q_2 + |A_2|^2 - |A_1|^2)U, V}\big\rangle + \int_{M'} \big\langle{(U (dV^*) - (dU)V^*, A_2 - A_1)}\big\rangle = 0
\end{align*}
In particular, if $Q_1 = Q_2 = 0$, then we also have that $\Lambda_{A_1} = \Lambda_{A_2}$ in $M'$.
\end{prop}
\begin{proof}
Recall that we extended $A_1$ and $A_2$ to the whole of $\mathbb{R} \times M_0$ such that $A_1 = A_2$ outside $M$; similarly, we extend $Q_1$ and $Q_2$ to have compact support and such that $Q_1 = Q_2$ outside $M$ (allowed by boundary determination).

Then the proof follows immediately after applying Theorem \ref{matidentity} to the restrictions $U|_{M}$ and $V|_{M}$, which solve the appropriate equations in $M$, and the fact that $A_1 = A_2$ and $Q_1 = Q_2$ outside $M$. The final conclusion follows since $U$ and $V$ were arbitrary.
\end{proof}

Let us denote by $L$ the connected component of $[-N, N] \times \partial M_0$ in $\mathbb{R} \times M_0 \setminus M^\circ$. Furthermore, in this setting, we have the following:

\begin{lemma}\label{Gattheboundary}
We have $G(x_1, x, v)$ equal to the identity for $(x_1, x) \in L$ and $v \in S_xM_0$. In particular, $G$ is equal to identity on the complement of $M'$ in $\mathbb{R} \times M_0$.
\end{lemma}
\begin{proof}

Let us fix a point $p \in \partial D$ and the polar coordinate $\eta \in S^{n-2}$ with $(p, \eta) \in \partial_+SD$. We have that $\tilde{A} = 0$ outside $M^\eta$ and it would suffice to show $G = Id$ on the connected component of $\partial M_0^\eta$ in $M_0^\eta \setminus M^\eta$, that we denote by $L^\eta$. In $M_0^{\eta} \setminus M^{\eta}$, the equation \eqref{trans} becomes:
\begin{align*}
2\frac{\partial G}{\partial \bar{z}} = [G, (A_1)_1 + i(A_1)_r]
\end{align*}
and thus we also have $2\frac{\partial G''}{\partial \bar{z}} = [G'', (A_1)_1 + i(A_1)_r]$, where $G'' = G - Id$, with $G''|_{\partial M_0^{\eta}} = 0$. Since $\frac{\partial}{\partial \bar{z}}$ is an elliptic operator and the previous equation is a linear one, we may apply the unique continuity theorem for linear elliptic first order systems (see \cite{bar} for a precise statement) and conclude that $G''(p, \eta, x_1, r) = 0$ for $z \in L^\eta$, since $G'' = 0$ on a codimension one set, thus proving the claim.

More precisely, note that $G''|_{\partial M_0^\eta} = 0$ implies that $dG''|_{\partial M_0^\eta} = 0$ and so we may extend $G''$ by zero slightly outside $M_0^\eta$ to a $C^\infty$ function by elliptic regularity. Then by the mentioned UCP we get $G'' = 0$ on $N^\eta$.
\end{proof}

In particular, we also have $G(x_1, x, v) = Id$ for $(x_1, x)$ in the connected component of $\partial M$ in $L$ (this is non-empty and open in $\partial M$) and $v \in S_xM_0$, by the previous lemma. Call this component $\Gamma$.

Often, the crux of the matter in the X-ray injectivity problems is to prove the independence of the gauge of the velocity variable; the only difference here from the usual problem is that we have a complex derivative $\mathbb{X}$, instead of the usual geodesic vector field $X$. Indeed, we have:

\begin{lemma}\label{Greduction}
If the solution of \eqref{ray} is independent of the velocity variable, then $G$ is a gauge equivalence between $A_1$ and $A_2$ on $E$, with $G|_{\Gamma} = Id$.
\end{lemma}
\begin{proof}
It is easy to show the following fact about the geodesic vector field: $X(x, v) f = df(v)$, when $f$ is independent of the velocity variable. Therefore, we can write down two equations out of \eqref{ray}, one for $v$ and the other for $-v$, respectively:
\begin{align*}
\frac{\partial G}{\partial x_1} + idG(v) = -A_1\Big(\frac{\partial}{\partial x_1}\Big)G - iA_1(v)G + GA_2\Big(\frac{\partial}{\partial x_1}\Big) + iGA_2(v)\\
\frac{\partial G}{\partial x_1} - idG(v) = -A_1\Big(\frac{\partial}{\partial x_1}\Big)G + iA_1(v)G + GA_2\Big(\frac{\partial}{\partial x_1}\Big) - iGA_2(v)
\end{align*}
by adding and subtracting the above equations, we easily get that $dG = -A_1G + GA_2$ or equivalently $G^*(A_1) = G^{-1}dG + G^{-1}A_1G = A_2$, which together with $G|_{\Gamma} = Id$ finishes the proof.
\end{proof}

Ideally we would like to reduce this to an ordinary X-ray injectivity problem on $M_0$ (technically, \eqref{ray} would become an injectivity problem for $G - Id$, with the inhomogeneous term equal to $\tilde{A}^1 + i\tilde{A}^r$) in some process of excluding the $x_1$ variable. This is indeed possible for the line bundle case (similar to what we will see in the next chapter) -- it involves the procedure of taking the logarithm of $G$ and applying the Fourier transform. Moreover, let us emphasise that all the information we obtained from the DN map through CGO solutions, we managed to pack into a single boundary condition: $G(x_1, x, v) = Id$ for $(x_1, x) \in \Gamma$ and $v \in S_xM_0$. The main problem is then reduced to an injectivity problem of an $X$-ray transform and is stated separately as Question \ref{conjecture2}.

It turns out that, under additional assumptions, we have $G$ equal to the identity on the whole of $\partial M$: 
\begin{prop}\label{Gextension}
\emph{If} the answer to Question \ref{conjecture2} is positive, and $G^*(Q_1) = Q_2$ \footnote{In particular, note that the condition on potentials holds if $Q_1 = Q_2 = 0$.}, then $G|_{\partial M} = Id$.
\end{prop}
\begin{proof}
By Lemma \ref{Gattheboundary}, we have that $G$ is equal to the identity on the outside of $M'$; thus, by the hypothesis and Lemma \ref{Greduction} we have $G^*(A_1) = A_2$. Moreover, we have $G|_{\Gamma} = Id$ and we want to prove that $G|_{\partial M} = Id$.

Let $F \in C^\infty(\partial M; \mathbb{C}^{m \times m})$ and assume smooth $U$ and $V$ solve $\Lapl_{A_1, Q_1} U = 0$ and $\Lapl_{A_2, Q_2}V = 0$ with the boundary condition $U|_{\partial M} = V|_{\partial M} = F$. By the DN map equality and the assumption on the gauges of $A_1$ and $A_2$ (normal components equal to zero near $\partial M$), we have $\partial_\nu U|_{\partial M} = \partial_\nu V|_{\partial M}$. The hypothesis on $G$ implies that $U' := GV$ satisfies $\Lapl_{A_1, Q_1} U' = 0$ and $U'|_{\Gamma} = F|_{\Gamma}$. Moreover, we have on $\Gamma$:
\begin{align*}
    \partial_\nu (U') = \iota_\nu \big((dG) V + G (dV)\big) = \iota_\nu \big(GA_2 V - A_1 GV + G(dV)\big) = \partial_\nu(V) = \partial_\nu (U)
\end{align*}
So by the UCP for elliptic systems (see Remark \ref{UCP}), we have $U \equiv U'$ and so $G|_{\partial M} \equiv Id$, as $F$ was arbitrary.
\end{proof}

\begin{rem}\rm
Notice that if $Q_1 = Q_2 = 0$, Proposition \ref{reductionconvex} implies that $\Lambda_{A_1} = \Lambda_{A_2}$ on $M'$, so the problem is reduced to proving uniqueness (up to gauges) on $M'$. More precisely, a gauge $G$ between $A_1$ and $A_2$ on $M'$, equal to the identity on $\partial M'$, would by uniqueness of first order equations and $G^*(A_1) = A_2$ imply $G = Id$ on $\Gamma$, so we may apply Proposition \ref{Gextension} to get $G = Id$ on $\partial M$.
\end{rem}

\begin{rem}\rm
There is a way of formulating Question \ref{conjecture2} in a more compact way. Namely, one could define the unitary connection $\hat{A}(R) = A_1R - RA_2$ on the endomorphism bundle of $E$ to get the form of the equation to $\mathbb{X}G + \hat{A}\big(\frac{\partial}{\partial x_1} + iv\big)G = 0$. Then we may formulate the problem in terms of just a single connection. 
\end{rem}

\begin{rem}\rm
If $A_1$ and $A_2$ are independent of the $x_1$ variable (on $M$) in the setting of Question \ref{conjecture2}, then we would have $A_1 \equiv A_2$ by the boundary condition and therefore $G \equiv Id$. 
\end{rem}

Therefore, the problem is reduced to a new kind of a non-abelian X-ray transform, Question \ref{conjecture2}. We leave it as one of the future projects to either further reduce the problem to an attenuated X-ray transform on $M_0$ or apply some other method to prove independence of the velocity variable directly. However, one thing is expected: methods from complex analysis and geometry could be useful to prove Question \ref{conjecture2}. This is supported by the work of Eskin (see Section 5 in \cite{Eskin}), where he proves Conjecture \ref{conjecture1} in the Euclidean metric case, by ``moving around" the $x_1$ direction, which can be interpreted as having the equations \eqref{trans} for essentially all planes going through points in $M$. In short, by generating a holomorphic family of such planes, Eskin obtains that $G$ is holomorphic with respect to this variable and hence constant by Liouville's theorem; such families are dense enough to guarantee $G$ is constant in the vertical directions and hence independent of $v$. Unlike the Euclidean metric case, in our situation we have a fixed $x_1$ direction, so we may also expect a different approach to be used.

\section{Gaussian Beams}\label{sec5}

In this section, we will construct the Gaussian Beam quasimodes (or generalised approximate eigenfunctions) that concentrate near geodesics, for the purposes of constructing the CGO solutions in the case where the transversal manifold is not necessarily simple. Moreover, we will use the method described in \cite{CTA}, where it was used in the case of a scalar potential and no first order term -- here we also consider the vector case and a first order term. More precisely, we consider CGO solutions of the form $e^{-sx_1} (v_s + r_s)$ for the general operator $\Delta + X + q$, where $s = \tau + i \lambda$, with $\tau$ and $\lambda$ real; we want to guarantee certain behaviour of the solutions in the limit as $\tau \to \infty$. In Section \ref{sec5.1} we construct the Gaussian Beams and in Section \ref{sec5.2} we use them to construct the CGOs. We start by motivating our definition:

\begin{itemize}
\item Since $v_s$ is the main part of the solutions we would like to have $e^{sx_1}(\Delta + X + Q) e^{-sx_1} v_s$ small in $L^2$ norm.
\item The solutions should concentrate along geodesics in a certain way.
\item Simple manifold case: this is covered in Proposition \ref{simplegauss} below and motivates the general transversal manifold case.
\end{itemize}

Throughout the section, we are working in the setting of $M \Subset (\mathbb{R}\times M_0, e\oplus g)$ with $\dim M_0 + 1 = \dim M = n \geq 3$.
\begin{definition}[Generalised quasimodes]
Given a family of functions $v_s$ on $M$ depending on a parameter $s = \tau + i\lambda$ ($\tau, \lambda \in \mathbb{R}$), we say that $v_s$ is a \textit{generalised approximate eigenfunction} or \textit{generalised quasimode} if $\lVert{v_s}\rVert_{L^2(M)} = O(1)$ as $\tau \to \infty$ and:
\begin{align*}
\Big \lVert{\Big((\Delta_{g} + X + q) + s(2\frac{\partial}{\partial x_1} - X_1) - s^2\Big)v_s}\Big \rVert_{L^2(M)} &= \lVert{e^{sx_1}(\Delta + X + q)e^{-sx_1}v_s}\rVert_{L^2(M)}\\
 &= o(|\tau|)
\end{align*}
\end{definition}

\begin{rem}\rm
The main difference between this and the definition of a quasimode found in \cite{CTA} is that the definition of a quasimode is independent of the $x_1$ direction, i.e. $v_s$ there was a function defined on $M_0$ only and it was only asked that $\lVert{(\Delta - s^2)v_s}\rVert_{L^2(M_0)} = o(|\tau|)$. This produces certain problems for us in the sense that the twisted Laplacian $d_A^*d_A$ now splits in a non-trivial way in an $x_1$ component, $x'$ component and a mixed component, unlike the ordinary Laplacian, $\Delta_{e\oplus g} = -|g|^{-1/2}\frac{\partial}{\partial x_1}\big(|g|^{1/2} \frac{\partial}{\partial x_1}\big) + \Delta_{g}$. As we will shortly see, this amounts to solving a certain $\bar{\partial}$-equation, which complicates things. 
\end{rem}


\subsection{Main construction of Gaussian Beams}\label{sec5.1}

We will focus on constructing generalised quasimodes. A complex vector field $X$ on $M$ is a skew-Hermitian vector field if $X^* = -X$ in the complexified tangent bundle $T_\mathbb{C}M$; moreover, we have the notion of a skew-Hermitian matrix of vector fields, which is a clear generalisation of the previously defined term. As a warm up for the general construction, we will first deal with the easy case of line bundles and $M_0$ simple, which comes out of our work in Section \ref{CGOsimple} -- in this case we have an ansatz for the eikonal equation.

Recall also that a unit speed geodesic $\gamma:[0, L] \to M$ is called \textit{non-tangential} if $\gamma(0), \gamma(L) \in \partial M$ and $\dot{\gamma}(0), \dot{\gamma}(L)$ are not parallel to $\partial M$, with $\gamma(t)$ in the interior of $M$ for $0 < t < L$.

\begin{prop}\label{simplegauss}
Let $(M_0, g)$ be a simple manifold and $\gamma: [0, L] \to M_0$ a non-tangential geodesic and let $\lambda$ be a real parameter. Let $X$ and $Y$ be two smooth skew-Hermitian vector fields on $M$. Then there exists a family of generalised quasimodes satisfying the above conditions, i.e. if $s = \tau + i \lambda$, then there exists $v_s, \omega_s \in C^{\infty}(\mathbb{R} \times M_0)$ such that:
\begin{align*}
\Big \lVert{\Big((\Delta_{g} + X + q) + s(2\frac{\partial}{\partial x_1} - X_1) - s^2\Big)v_s}\Big \rVert_{L^2(M)} = o(|\tau|) \quad \text{and} \quad \lVert{v_s}\rVert_{L^2(M)} = O(1)\\
\Big \lVert{\Big((\Delta_{g} + Y + q) - s(2\frac{\partial}{\partial x_1} - Y_1) - s^2\Big)\omega_s}\Big \rVert_{L^2(M)} = o(|\tau|) \quad \text{and} \quad \lVert{\omega_s}\rVert_{L^2(M)} = O(1)
\end{align*}
as $\tau \to \infty$ and for each $\phi \in C(M_0)$ and $x'_1 \in \mathbb{R}$ we have:
\begin{align*}
\lim_{\tau \to \infty} \int_{\{x'_1\} \times M_0} v_s \bar{\omega}_s \phi dV_{g} =  \int_0^L e^{2\lambda t} e^{\Phi_1 + \bar{\Phi}_2}\phi(\gamma(t)) dt
\end{align*}
where $\Phi_1$ and $\Phi_2$ are smooth on $\mathbb{R} \times [0, L]$ and satisfy the following equations:\footnote{In these equations, we extend the domain of definition of $X$ and $Y$ from $M$ to $\mathbb{R} \times M_0$ smoothly to compactly supported vector fields and with a slight abuse of notation still denote them the same.}
\begin{align*}
\left(\frac{\partial}{\partial x_1} + i \frac{\partial}{\partial r}\right) (\Phi_1) = \frac{1}{2}(X_1 + iX_r) \quad \text{and} \quad \left(-\frac{\partial}{\partial x_1} + i \frac{\partial}{\partial r}\right) (\Phi_2) = \frac{1}{2}(-Y_1 + iY_r)
\end{align*}
\end{prop}
\begin{proof}
As in Section \ref{4.1}, consider a simple manifold $D$ which contains $M_0$ and a point $p \in D$ such that $\mathbb{R} \times \{p\}$ is disjoint from $M$ and consider the global polar coordinate system at this point. Furthermore, we proceed by picking a different conjugating exponent -- we let $\rho = x_1 + ir$. By Lemma \ref{conjugate}:
\begin{align*}
e^{s\rho}\Big(\Delta + X + q\Big)e^{-s\rho} u_s = (\Delta + X + q)u_s - s \Big(\Delta \rho + X(\rho) - 2\langle{d \rho, d \cdot}\rangle\Big)u_s - s^2|d\rho|^2 u_s
\end{align*}
One wants to have a handle on the size of the right hand side, so one equates the linear and the quadratic terms in $s$ to zero; this is done in Section \ref{CGOsimple}. The same construction gives us $u_s = |g|^{-1/4} \cdot a \cdot b_{\tau}(\theta)$, where $a$ and $b_\tau \in C^{\infty}(S^{n-2})$ are chosen such that:
\begin{gather*}
\left(\frac{\partial}{\partial x_1} + i \frac{\partial}{\partial r}\right) (a) = \frac{1}{2}(X_1 + iX_r) a \\
\lVert{b_{\tau}}\rVert^2_{L^2(S^{n-2})} = 1, \quad \lVert{b_{\tau}}\rVert^2_{W^{2, \infty}(S^{n-2})} = O(\tau^{\alpha}) \quad \text{and} \quad |b_\tau|^2 dS \to \delta_{\theta_0}
\end{gather*}
i.e. $b_\tau$ is a $C^{\infty}$ approximation to the delta function, with $\alpha < 1$; here $dS$ is the volume element of $S^{n-2}$. We pick $a$ of the form $e^{\Phi_1}$, so that $\Phi_1$ satisfies the equation: 
\begin{align*}
\left(\frac{\partial}{\partial x_1} + i \frac{\partial}{\partial r}\right) (\Phi_1) = \frac{1}{2}(X_1 + iX_r)
\end{align*}
Now, given $u_s$ as above, we set $v_s = e^{-isr}u_s$:
\begin{align*}
e^{sx_1}(\Delta + X + q)e^{-sx_1} v_s &= e^{-isr}e^{s \rho}(\Delta + X + q)e^{-s\rho} u_s\\
 &= e^{-isr} (\Delta + X + q) \Big(|g|^{-1/4} \cdot a \cdot b_{\tau}(\theta)\Big) =: f
\end{align*}
By using the properties of $b_{\tau}$ and the boundedness of other factors, we see that $f$ is clearly equal to $O(\tau^{\alpha})$ in $L^2(M)$ with $\alpha<1$. But this exactly means that $v_s$ is a generalised approximate eigenfunction. Analogously we construct the $\omega_s$ function with respect to $Y$, but with one difference in mind -- we take $-x_1$ to be the Carleman weight (this will be important in the integral identity). Moreover, we have:
\begin{align*}
\int_{\{x'_1\} \times M_0} v_s \bar{\omega}_s \phi dV_{g} \to \int_0^L e^{2\lambda r} e^{\Phi_1 + \bar{\Phi}_2} \phi(\gamma(r)) dr
\end{align*}
when $\tau \to \infty$, for each $x'_1$, by using that the volume element on $M_0$ is $dV_{g_0} = |g|^{\frac{1}{2}}dx_2 \wedge \dotso \wedge dx_n$ and the concentration properties of $b_\tau$.
\end{proof}
Now we are ready to make the passage to the case of the transversal manifold being \textit{non-simple}, with the previous proposition giving us some intuition. Most of the proof we are about to see is analogous to the proof of Proposition 3.1 in \cite{CTA}. The main difference is that, when constructing the amplitude $a$ in $v_s = e^{is \Theta} a$, we do not get an ordinary differential equation -- we get that $a$ satisfies a certain $\bar{\partial}$ equation. This complicates the construction of $a$ slightly and uses the properties of $\bar{\partial}$ equations we already discussed in Section \ref{paramsoln}. Moreover, the derivation of the limit integral is also more involved. We will prove the following theorem for line bundles first and then generalise to all vector bundles in a series of results after it:

\begin{theorem}[Main construction of the Gaussian Beams]\label{main}
Let $\gamma: [0, L] \to M_0$ be a non-tangential geodesic and let $\lambda$ be a real parameter, with $M_0$ \textit{any compact manifold with boundary}. Let $X$ and $Y$ be two smooth skew-Hermitian vector fields on $M$, which we extend to compactly supported vector fields on $\mathbb{R}\times M_0$ (still denoted $X$ and $Y$). Then there exists a family of generalised quasimodes satisfying the above conditions, i.e. if $s = \tau + i \lambda$, then there exists $v_s, \omega_s \in C^{\infty}(J_0 \times M_0)$, where $J_0 = [-N_0, N_0]$ for some large positive integer $N_0$, such that:
\begin{align*}
\Big \lVert{\Big((\Delta_{g} + X + q) + s(2\frac{\partial}{\partial x_1} - X_1) - s^2\Big)v_s}\Big \rVert_{L^2(J_0 \times M_0)} = o(|\tau|) \quad \text{and} \quad \lVert{v_s}\rVert_{L^2(J_0 \times M_0)} = O(1)\\
\Big \lVert{\Big((\Delta_{g} + Y + q) - s(2\frac{\partial}{\partial x_1} - Y_1) - s^2\Big)\omega_s}\Big \rVert_{L^2(J_0 \times M_0)} = o(|\tau|) \quad \text{and} \quad \lVert{\omega_s}\rVert_{L^2(J_0 \times M_0)} = O(1)
\end{align*}
as $\tau \to \infty$ and for each $\phi \in C(M_0)$ and $x'_1 \in [-N_0, N_0]$ we have:
\begin{align*}
\lim_{\tau \to \infty} \int_{\{x'_1\} \times M_0} v_s \bar{\omega}_s \phi dV_{g} =  \int_0^L e^{-2\lambda t} e^{\Phi_1 + \bar{\Phi}_2}\phi(\gamma(t)) dt
\end{align*}
where $\Phi_1$ and $\Phi_2$ are smooth on $\mathbb{R} \times [0, L]$ and satisfy the following equations:
\begin{align}\label{fijne}
\left(\frac{\partial}{\partial x_1} - i \frac{\partial}{\partial r}\right) (\Phi_1) = \frac{1}{2}(X_1 - iX_r) \quad \text{and} \quad \left(\frac{\partial}{\partial x_1} - i \frac{\partial}{\partial r}\right) (\bar{\Phi}_2) = \frac{1}{2}(-Y_1 + iY_r)
\end{align}
Moreover, the following limit holds for $v_s$ and $\omega_s$ and any one form $\alpha$ on $M_0$:
\begin{align*}
\lim_{\tau \to \infty} \frac{1}{\tau} \int_{\{x'_1\} \times M_0} \langle{\alpha, dv_s}\rangle\bar{\omega}_s \phi dV_g &= \int_0^L i \alpha(\dot{\gamma}(t)) e^{\Phi_1 + \bar{\Phi}_2}e^{-2 \lambda t} \phi(\gamma(t)) dt\\
\lim_{\tau \to \infty} \frac{1}{\tau} \int_{\{x'_1\} \times M_0} \langle{\alpha, d\bar{\omega}_s}\rangle v_s \phi dV_g &= -\int_0^L i \alpha(\dot{\gamma}(t)) e^{\Phi_1 + \bar{\Phi}_2}e^{-2 \lambda t} \phi(\gamma(t)) dt
\end{align*}
\end{theorem}
\begin{proof}
Firstly, let us isometrically embed our manifold $(M_0, g)$ into a larger closed manifold $(\widehat{M}, g)$ of the same dimension. This is possible since we can form the manifold $\widehat{M} = M_0 \sqcup_{\partial M_0} M_0$, which is the disjoint union of two copies of $M_0$, glued along the boundary; $g$, $X$ and $Y$ are smoothly extended to $\mathbb{R} \times \widehat{M}$. We will extend the geodesic such that for $\epsilon > 0$ we have $\gamma(t) \in \hat{M} \setminus M_0$ for $t \in (-2\epsilon, 0) \cup (L, L + 2\epsilon)$; this is possible since $\gamma$ is non-tangential. Let $N_0$ be a large positive integer such that $(-N_0, N_0) \times M_0$ contains $M$ and the support of $X$ and $Y$; let us introduce the notation for the interval $J_1:= [-N_0-1, N_0+1]$.

Let us first introduce a set of local coordinates along the geodesic; a detailed account of this can be found in \cite{CTA}. Since our manifold is compact and $\gamma$ has no loops, we can assume $\gamma$ self-intersects only finitely many times, at $0 < t_1 < \dotso < t_{N'} < L$ and that there is an open cover $\{(U^{(j)}, \varphi_j)\}_{j = 0}^{N' + 1}$ of $\gamma([-\epsilon, L + \epsilon])$ such that $\varphi_j(U^{(j)}) = I^{(j)} \times B$, where $I^{(j)}$s are open intervals and $B$ a small $n-2$-dimensional ball. Also, $\varphi_j(\gamma(t)) = (t, 0)$ and $t_j$s belong only to $I^{(j)}$s and $\bar{I}^{(j)} \cap \bar{I}^{(k)} = \emptyset$ unless $|j - k| \leq 1$; $\varphi_i$s agree on overlaps. These are called the \emph{Fermi coordinates} and they have the following two properties along the geodesic: the metric is diagonal and $\partial_i g^{jk} = 0$ (and so the Christoffel symbols vanish). Also, let us denote by $F$ the map from $U = [-2\epsilon, L + 2\epsilon] \times B$ to $\widehat{M}$, which restricts to the inverse charts on $I^{(i)} \times B$s; this is well defined since the charts agree on overlaps. The map $F$ is locally a diffeomorphism, but is not globally because of self-intersections of the geodesic (see Figure \ref{fig:myfigure}).

\begin{figure}
\centering
  \resizebox{\textwidth}{!}{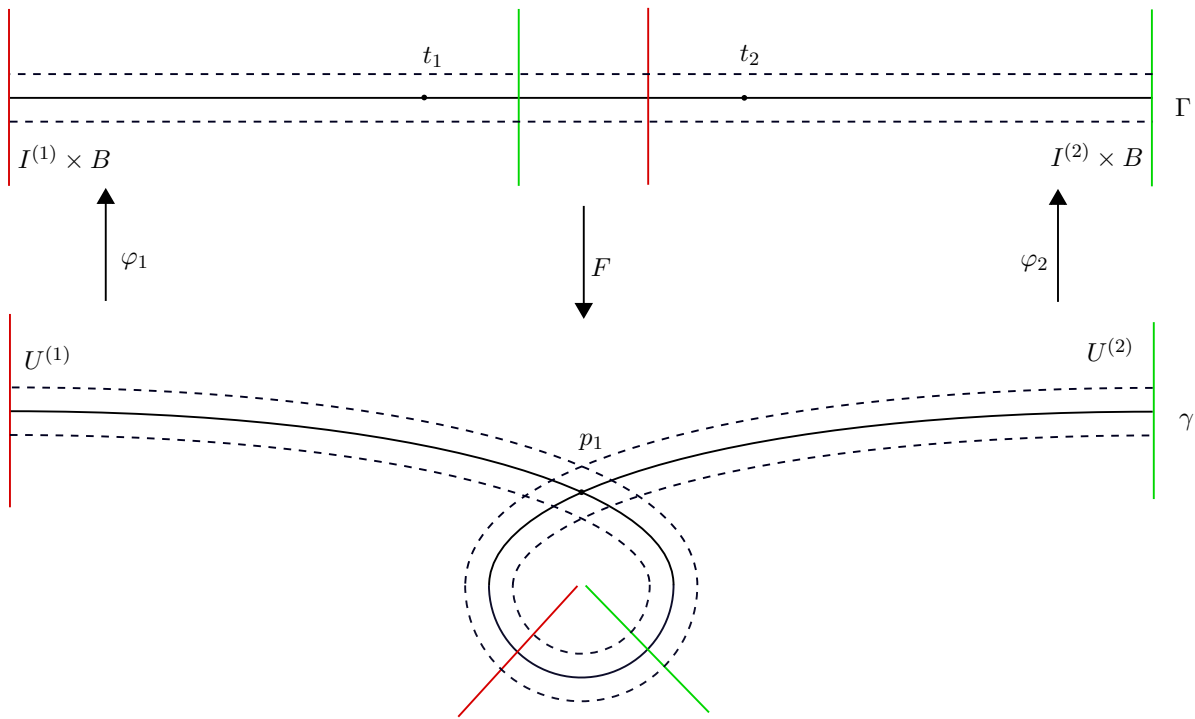}
    \caption{An illustration of the local diffeomorphism $F$ obtained from Fermi coordinates: the cover of the geodesic $\gamma$ is given by charts $(U^{(1)}, \varphi_1)$ and $(U^{(2)}, \varphi_2)$, with $\varphi_i(U^{(i)}) = I^{(i)} \times B$ for $i = 1, 2$. Red colour delimits the $U^{(1)}$ piece, the green one delimits $U^{(2)}$ and $\gamma(t_1) = \gamma(t_2) = p_1$.} \label{fig:myfigure}

 \end{figure}

Rather than constructing the quasimode locally, near a point $p_0 = \gamma(t_0)$ on $\gamma([-\epsilon, L +  \epsilon])$, observe that we may use the map $F$ as a local diffeomorphism and pull back all the data ($g$, $X$ and $Y$) to $\mathbb{R} \times U$ -- let us still denote the pullbacks with the same letters. Let us also use the notation $D_i: = J_i \times U$ for $i = 0$ and $1$. We will use the coordinate $y$ on $B$ and denote the geodesic in these local coordinates as $\Gamma = \{(t, 0)\}$ in $U$. Furthermore, we will construct the quasimode on $D$ and then provide a method to pushforward this quasimode to $J_0 \times M_0$.

Let us seek for solutions of the form $v_s = e^{is\Theta} a$, where $a$ and $\Theta$ will be complex functions supported in $|y| < \delta'/2$. Then we have:
\begin{multline*}
e^{sx_1}(\Delta + X + q)e^{-sx_1} v_s = e^{is\Theta} e^{-s(-x_1 + i\Theta)} (\Delta + X + q)e^{s(-x_1 + i\Theta)}a  \\
= e^{is \Theta} \Big\{(\Delta + X + q)a + s\Big(2\frac{\partial a}{\partial x_1} - 2i\langle{d\Theta, da}\rangle + \big(-X_1 + iX(\Theta)\big)a + i\Delta (\Theta) a\Big) - s^2 (1 - |d\Theta|^2)a \Big\}
\end{multline*}
by putting $\rho = -x_1 + i\Theta$ and using Lemma \ref{conjugate}. Firstly, let us solve $|d\Theta|^2 = 1$ up to order $|y|^N$ on $\Gamma$. We look for $\Theta$ in the form $\Theta = \sum_{j = 0}^{N} \Theta_j$, where:
\begin{align*} 
\Theta_j = \sum_{|\alpha| = j} \frac{\Theta_{j, \alpha}(t)}{\alpha !} y^{\alpha}
\end{align*}
are the homogeneous components and we write $g^{jk} = \sum_{l = 0}^{N} g^{jk}_l + r_{N+1}^{jk}$, where 
\begin{align*}
g_l^{jk} (t, y) = \sum_{|\beta| = l} \frac{g_{l, \beta}^{jk} (t)}{\beta !} y^{\beta} \quad \text{and} \quad r_{N+1}^{jk} = O(|y|^{N + 1})
\end{align*} 
is the remainder in Taylor's theorem. By the properties of the coordinates, we have $g^{jk}_0 = \delta^{jk}$ and $g_{1}^{jk}=0$. Let us accordingly choose $\Theta_0(t, y) = t$ and $\Theta_1(t, y) = 0$. Most of the next step follows from the lines of \cite{CTA}, but we give it here for completeness:
\begin{multline*}
\langle{d\Theta, d\Theta}\rangle - 1 = g^{jk}\partial_j \Theta \partial_k \Theta - 1 = (1 + g_2^{11} + \dotso)(1 + \partial_t \Theta_2 + \dotso)(1 + \partial_t \Theta_2 + \dotso)\\
+ 2(g_2^{1\alpha} + \dotso)(1 + \partial_t \Theta_2 + \dotso)(\partial_{y_{\alpha}} \Theta_2 + \dotso)\\
+ (\delta^{\alpha \beta} + g_2^{\alpha \beta} + \dotso)(\partial_{y_{\alpha}} \Theta_2 + \partial_{y_{\alpha}}\Theta_3 + \dotso)(\partial_{y_{\beta}} \Theta_2 + \partial_{y_{\beta}}\Theta_3 + \dotso) - 1\\
= [2\partial_t \Theta_2 + \nabla_y \Theta_2 \cdot \nabla_y \Theta_2 + g_2^{11}] + \sum_{p = 2}^N (\dotso) + O(|y|^{N+1})
\end{multline*}
We want to choose $\Theta_i$ such that the first bracket and the sum above vanish. We pick $\Theta_2(t, y) = \frac{1}{2} H(t)y \cdot y$ where $H(t)$ is a smooth complex symmetric matrix. For the first bracket to vanish, we need to have:
\begin{align*}
\dot{H}(t) + H(t)^2 = F(t)
\end{align*}
where $F(t)$ is the symmetric matrix determined by $g_2^{11}(t, y) = - F(t) y\cdot y$. Choosing $H_0 = H(t_0)$ for $t_0 := -2\epsilon$ to be any complex symmetric matrix with $\im(H)$ positive definite; following \cite{CTA} this Riccati equation has a unique smooth complex symmetric solution $H(t)$ with $\im(H(t))$ positive definite for all $t \in [-2\epsilon, L+2\epsilon]$. Now we find $\Theta_3, \dotso, \Theta_N$ by inductively solving the first order ODEs along $\Gamma$ with an initial condition at $t_0$, obtained by collecting the homogeneous terms in $y$ of higher order in the previous expansion. We obtain a smooth $\Theta$ such that $|d\Theta|^2 = 1$ up to order $|y|^N$.

Now we turn to the more interesting step, how to solve:
\begin{align*}
s\Big[2\frac{\partial a}{\partial x_1} - 2i \langle{d\Theta, da}\rangle + (-X_1 + iX(\Theta))a + i\Delta(\Theta) a\Big] + \Big(\Delta + X\Big)a = 0 
\end{align*}
up to order $|y|^{N}$. We look for $a$ in the form
\begin{align*}
a = \tau^{\frac{n - 2}{4}}(a_0 + s^{-1} a_{-1} + \dotsb + s^{-N}a_{-N})\chi(\frac{y}{\delta'})
\end{align*}
where $\chi$ is a bump function defined such that $\chi = 1$ on $|y| \leq 1/4$, $\chi = 0$ for $|y| \geq 1/2$. We now equate each degree of $s$ in the above expression to zero and obtain $N + 1$ equations for each degree $1, 0, \dotsc, -(N - 1)$:
\begin{gather*}
2\frac{\partial a_0}{\partial x_1} - 2i \langle{d\Theta, da_0}\rangle + \Big(-X_1 + iX(\Theta) + i\Delta(\Theta)\Big)a_0 = 0\\
2\frac{\partial a_j}{\partial x_1} - 2i \langle{d\Theta, da_j}\rangle + \Big(-X_1 + iX(\Theta) + i\Delta(\Theta)\Big)a_j + (\Delta + X)a_{j+1}= 0
\end{gather*}
for each $j = -1, \dotsc, -N$. Let us introduce $\eta = i\Delta\Theta - X_1 + iX(\Theta)$ and write $\eta = \eta_0 + \dotso + \eta_N + O(|y|^{N+1})$ for the Taylor expansion of $\eta$. We look for $a_0 = a_{00} + a_{01} + \dotso + a_{0N}$ where each $a_{0i}$ is a homogeneous polynomial of degree $i$. Then the degree one equation becomes:
\begin{align*}
2\frac{\partial}{\partial x_1}(a_{00} + \dotso + a_{0N}) - 2i g^{jk}\partial_j\Theta \partial_k a + (\eta_0 + \eta_1 + \dotso)(a_{00} + \dotso + a_{0N}) = 0
\end{align*}
to order $|y|^{N}$. After rewriting, this becomes:
\begin{multline*}
= 2\frac{\partial}{\partial x_1}(a_{00} + \dotso a_{0N}) - 2i (1 + g_2^{11} + \dotso )(1 + \partial_t \Theta_2 + \dotso)(\partial_t a_{00} + \partial_t a_{01} + \dotso) \\
- 2i(g_2^{1\beta} + \dotso)(1 + \partial_t \Theta_2 + \dotso)(\partial_{y_{\beta}} a_{01} + \dotso)\\
- 2i(g_2^{\alpha 1} + \dotso)(\partial_{y_{\alpha}} \Theta_2 + \dotso)(\partial_t a_{00} + \partial_t a_{01} + \dotso)\\
- 2i(\delta^{\alpha \beta} + g_2^{\alpha \beta} + \dotso)(\partial_{y_{\alpha}} \Theta_2 + \dotso)(\partial_{y_{\beta}} a_{01} + \dotso)\\
+ (\eta_{0} + \eta_1 + \dotso + \eta_{N} + O(|y|^{N + 1}))(a_{00} + a_{01} + \dotso + a_{0N})\\
= \Big[2\frac{\partial a_{00}}{\partial x_1} - 2i \partial_t a_{00} + \eta_0 a_{00}\Big] + \Big[2\frac{\partial a_{01}}{\partial x_1} - 2i \partial_t a_{01} - 2i \partial_{y_{\alpha}} \Theta_2 \partial_{y_{\alpha}} a_{01} + \eta_1 a_{00} + \eta_0 a_{01}\Big] + \dotso
\end{multline*}
where we have written down the first two terms in the $y$ expansion. For us, the equation for $a_{00}$ is particularly important (it will give us the value of $a_0$ along $\Gamma$). We have that $\eta_0 = (i\Delta \Theta - X_1 + iX(\Theta))(t, 0)$, where we know that $\Theta = t + 1/2 H(t) y\cdot y + O(|y|^3)$. Therefore, we compute:
\begin{align*}
\Delta \Theta (t, 0)= - |g|^{-\frac{1}{2}}\frac{\partial}{\partial x^j}(|g|^{\frac{1}{2}} g^{jk} \frac{\partial \Theta }{\partial x^k}) = -|g|^{-\frac{1}{2}} \frac{\partial |g|^{\frac{1}{2}}}{\partial t} - \delta^{jk}H_{jk} = -\mathrm{tr}H(t)
\end{align*}
So, our equation for $a_{00}$ becomes:
\begin{align}\label{a00}
\Big(\frac{\partial}{\partial x_1} - i\frac{\partial}{\partial t}\Big)a_{00} = \frac{1}{2}\Big(X_1 - iX_t + i\mathrm{tr}H(t)\Big)a_{00}
\end{align}
which we have seen in a more general, matrix case. Here, we want a solution of the form $a_{00} = e^{\Phi_1 + f_1}$, so that we obtain, for $\partial = 1/2(\partial / \partial x_1 - i\partial / \partial t)$
\begin{align}\label{fijna}
\partial \Phi_1 = \frac{1}{4}\Big(X_1 - iX_t\Big) \quad \text{and} \quad \frac{\partial f_1}{\partial t} = -\frac{1}{2}\Tr H(t)
\end{align}
where $\Phi_1$ is a function in both $x_1$ and $t$, $f_1$ is a function of just $t$. Now for the rest of the $a_{0i}$ for $i > 0$, we obtain a similar vector valued equation of the form:
\begin{align*}
\partial v + A v + f = 0
\end{align*}
where $v$ and $f$ are vectors and $A$ is a matrix. The reason for this is that for $i > 0$, we get more components in the Taylor expansion, so we get a coefficient for each (think of $a_{0i}$s as vectors). This is solvable by our previous work on fundamental solutions of such equations, so that we produce an invertible matrix $C$ such that
\begin{equation}
\partial C = -AC
\end{equation}
in $\mathbb{R} \times (-2\epsilon, L + 2\epsilon)$ (see Section \ref{paramsoln}). Then we try $v = Cu$ for some vector function $u$ and we get the equation: $\partial u = - C^{-1}f$, which we know how to solve in the bounded domain $J_0 \times [-\frac{3}{2}\epsilon, L + \frac{3}{2}\epsilon]$, by e.g. multiplying  $f$ with a cut-off function, equal to one on $J_0 \times [-\frac{3}{2}\epsilon, L + \frac{3}{2}\epsilon]$ and supported in $J_1 \times (-2\epsilon, L + 2\epsilon)$ in order to extend it to the whole $(x_1, t)$-plane and use the generalised Cauchy integral formula. Hence we determine $a_{0}$ and proceed to determine $a_i$s for $i > 0$ inductively. Notice also that $X$ is compactly supported, so we may indeed take the zero extension of it to the $(x_1, t)$-plane and solve the first equation in \eqref{fijna}.

At this point we make a remark about constructing the $\omega_s$ solution, which is the solution where we use $e^{sx_1}$ exponent in the CGO solution (and hence the $-s$ in the formulation of the theorem). The point is that everything just gets a minus sign at each spot where we use $x_1$. Checking through the details, we obtain a version of the equation \eqref{fijna} (we use the fact that $Y$ is skew-Hermitian):
\begin{align*}
\partial \bar{\Phi}_2 = \frac{1}{4}\Big(-Y_1 + iY_t\Big) \quad \text{and} \quad \frac{\partial f_2}{\partial t} = -\frac{1}{2}\Tr H(t)
\end{align*}
We are left with the terms of the form:
\begin{multline*}
e^{sx_1}(\Delta + X)e^{-sx_1} e^{is\Theta} a = e^{is\Theta} \tau^{\frac{n-2}{4}}\Big[s^2h_2a + sh_1 + \dotsb + s^{-(N - 1)}h^{-(N - 1)} \\
+ s^{-N}(\Delta + X)a_{-N}\Big]\chi(\frac{y}{\delta'}) + e^{is\Theta}\tau^{\frac{n-2}{4}}sb \tilde{\chi}(\frac{y}{\delta'})
\end{multline*}
where we have $h_j$s to be equal to zero to order $|y|^N$ on $\Gamma$; we also introduce $b$ and $\tilde{\chi}$ to describe the leftover terms which appear upon differentiating the function $\chi$ in a sum, but which therefore are zero near and far away of $\Gamma$. Concretely, we have $b = 0$ for $|y| \leq \delta'/4$ and $\tilde{\chi} = 0$ for $|y| \geq 1/2$ and the most important fact about this term is that it is linear in $s$.

In order to determine some bounds on $v_s$, let us introduce a positive constant $c$, for which it holds that $\im{H(t)} y \cdot y \geq c|y|^2$. Then we have:
\begin{align}
|e^{i s \Theta}| &= e^{-\lambda \re{\Theta}} e^{-\tau \im{\Theta}} = e^{-\lambda t} e^{-\lambda O(|y|^2)} e^{-\frac{\tau}{2} \im{H(t)} y \cdot y} e^{-\tau O(|y|^3)}\\
|v_s(x_1, t, y)| &\lesssim \tau^{\frac{n-2}{4}} e^{-\frac{1}{4} c\tau |y|^2} \chi(\frac{y}{\delta'})\label{esimatev_s}
\end{align}
after decreasing $\delta'$ if necessary and using the $1/4$ factor in the exponential to dominate the remaining $O(|y|^3)$ factor, for $x_1 \in J_0$. Thus we have:
\begin{equation*}
\lVert{v_s}\rVert_{L^2(J_0 \times U)} \lesssim \lVert{\tau^{\frac{n-2}{4}} e^{-\frac{1}{4}c\tau |y|^2} }\rVert_{L^2(J_0 \times U)} = O(1)
\end{equation*}
\begin{multline}\label{tauKproof}
\Big\lVert{e^{sx_1}(\Delta + X)e^{-sx_1} v_s}\Big\rVert_{L^2(J_0 \times U)} \lesssim \Big\lVert{\tau^{\frac{n-2}{4}} e^{-\frac{1}{4}c\tau |y|^2}\Big(\tau^2|y|^{N + 1} + \tau^{-N} + \tau b \tilde{\chi}\Big)}\Big\rVert_{L^2(J_0 \times U)}\\ = O(|\tau|^{\frac{3 - N}{2}})
\end{multline}
where the second line is equal to $O(|\tau|^{-K})$ upon setting $N = 2K + 3$, for any fixed $K$, a positive integer. 

Let us now record a boundary estimate for future purposes. Namely, since the geodesic intersects the boundary $\partial M_0$ transversely at $t = 0$ and $t = L$, we can introduce the implicit coordinates $\{(t(y), y) : |y| < \epsilon'\}$ for some smooth function $t(y)$ and small $\epsilon' > 0$. Then for $\delta'$ small enough:
\begin{gather*}
\lVert{v_s(x_1, \cdot)}\rVert^2_{L^2(\partial M_0 \cap U)} = \int_{|y| < \epsilon'} |v_s(x_1, t(y), y)|^2 dS(y) \lesssim \int_{\mathbb{R}^{n - 2}} \tau^{\frac{n-2}{2}} e^{-\frac{1}{2}c\tau |y|^2} dy = O(1)
\end{gather*}
for $x_1$ in $J_0$ and as $|\tau| \to \infty$.

Now we are done with the local construction and bounds on $J_0 \times [-\epsilon, L + \epsilon] \times B$ and want to glue the solutions together with desired concentration properties. Let us denote by $v_s^{(j)}$ the pushforward by the coordinate map $Id \times \varphi_j^{-1}$ of the so obtained solution on $J_0 \times U^{(j)}$ (where $Id: \mathbb{R} \to \mathbb{R}$ is the identity map). We thus obtain $v_s^{(0)}, v_s^{(1)}, \dotsc, v_s^{(r)}$. To glue these, let $\chi_j(t)$ be a partition of unity subordinate to $I^{(j)}$; the we extend these to $U^{(j)}$ by saying $\tilde{\chi}_j(x_1, t, y) = \chi_j(t)$ and finally let:
\begin{align}\label{methodpou}
v_s := \sum_{j = 0}^{r} \tilde{\chi}_j v_s^{(j)}
\end{align}
The previous remark allows us to have $v_s^{(j)} = v_s^{(j+1)}$ in the overlaps $J_0 \times \big(U^{(j)} \cap U^{(j+1)}\big)$. Now, pick small neighbourhoods of the geodesic self-intersection points $p_1, \dotsc, p_R$ and call them $V_1, \dotsc, V_R$; for $\delta'$ sufficiently small, we get that $F$ is injective on the complement of the inverse image by $F$ of the $V_i$s (see Figure \ref{fig:myfigure}). Therefore, we can pick a finite cover by $W_1, \dotsc, W_S$ of the remaining points on the geodesic such that $W_i \subset U^{(l_i)}$ for some $l_i$ and $\text{supp}(v_s) \subset \big(\cup V_i \big) \cup \big(\cup W_j \big)$ and moreover, the restrictions to these satisfy:
\begin{align}\label{pou_argument}
v_s|_{V_i} = \sum_{\gamma(t_l) = p_i} v_s^{(l)} \quad \text{and} \quad v_s|_{W_i} = v_s^{(l_i)}
\end{align}
It is now clear that the wanted $L^2$ bounds on $v_s$ follow from our previous local considerations on each of $v_s^{(i)}$. We are left with the concentration results to prove -- by considering the partitions of unity subordinate to $V_i$s and $W_j$s, we can assume that $\phi$ has compact support in one of these sets. Let us first consider the easier case where $\text{supp}( \phi) \subset W_k$ for some $k$. By a completely analogous  construction, we may assume that we have $\omega_s = e^{is\Theta} b$ on $J_0 \times [-\epsilon, L+ \epsilon] \times B$, constructed with respect to $Y$ -- notice that $\Theta$ is solved for independently of the vector fields $X$ and $Y$ (recall that we only want $|d\Theta|^2 = 1$ up to order $|y|^N$).

In $W_k$, we have $v_s = e^{is\Theta} a$ and $\omega_s = e^{is \Theta} b$, where we dropped the indices to simplify notation. Then we have:
\begin{multline*}
\int_{\{x'_1\} \times M_0} v_s \bar{\omega}_s \phi dV_g =  \int e^{i s \Theta} e^{-i \bar{s} \bar{\Theta}} a \bar{b} \phi dV_g \\
= \int_0^L \int_{\mathbb{R}^{n-2}} e^{-2 \lambda \re \Theta} e^{-2 \tau \im \Theta} \tau^{\frac{n-2}{2}} (a_0 + O(\tau^{-1}))(\bar{b}_0 + O(\tau^{-1})) \chi(y/\delta')^2 \phi |g|^{1/2} dy dt \\
= \int_0^L \int_{\mathbb{R}^{n-2}} e^{-\im H(t) y \cdot y} e^{-2 \tau^{-1/2}O(|x|^3)} e^{-2\tau^{-1} O(|x|^2)} e^{-2\lambda t}\tau^{\frac{n-2}{2}} \\
\cdot \big(a_0(t, \tau^{-1/2} x) + O(\tau^{-1})\big)\big(\bar{b}_0 (t, \tau^{-1/2} x)+ O(\tau^{-1})\big) \chi\big(x/(\tau^{1/2}\delta')\big)^2 |g|^{1/2}\big(t, \tau^{-1/2} x\big) \phi dy dt 
\end{multline*}
by performing the substitution $y = \tau^{-1/2} x$; we can see what the limit is -- namely, by bounding
\[e^{-c|x|^2} e^{2A|x|^3/(\tau^{1/2})} e^{2B|x|^2/\tau} \leq e^{|x|^2(-c + 2A\delta' + 2B/\tau)}\] where $c$ is as before the positive constant such that $\im H(t) y \cdot y \geq c |y|^2$ and using the fact that we integrate over $|y| \leq \tau^{1/2} \delta'$, by taking sufficiently small $\delta'$ we get exponent negative and hence we get an integrable function; thus we may use the Dominated convergence theorem to get this tends to, as $\tau \to \infty$:
\begin{multline}\label{det}
\int_0^L e^{-2\lambda t} e^{\Phi_1+ f_1 + \bar{\Phi}_2 + \bar{f}_2} \phi(\gamma(t)) \int_{\mathbb{R}^{n-2}} e^{-\im H(t) x \cdot x} dx dt \\
= \int_{\mathbb{R}^{n-2}} e^{-|y|^2} dy \int_0^L \frac{e^{-2\lambda t} e^{\Phi_1 + f_1 + \bar{\Phi}_2 + \bar{f}_2} \phi(\gamma(t))}{\sqrt{\text{det} \im H(t)}} dt
\end{multline}
by using the linear change of variable by the matrix $\im H(t)$. However, from before we know that:
\begin{align*}
\text{det}(\im H(t)) = \text{det}(\im H(t_0)) e^{-2 \int_{t_0}^t \Tr \re H(s) ds} \quad \text{and} \quad \frac{\partial(f_1 + \bar{f}_2)}{\partial t} = -  \Tr \re H(t)
\end{align*}
Hence we obtain cancellation in the above integral and by picking the initial condition for $H(t_0)$ such that $\det(\im H(t_0)) = \pi^{n-2}$, we finally get the desired limit:
\begin{align*}
\int_0^L e^{-2\lambda t} e^{\Phi_1 + \bar{\Phi}_2} \phi(\gamma(t)) dt
\end{align*}
Moreover, in the case where we have $\text{supp}(\phi) \subset V_j$ for some $j$, we have $v_s = \sum_{\gamma(t_l) = p_j} v_s^{(l)}$ and $\omega_s = \sum_{\gamma(t_l) = p_j} \omega_s^{(l)}$, which means that we have the following expression:
\begin{gather*}
v_s \bar{\omega}_s = \sum_{\gamma(t_l) = p_j} v_s^{(l)} \bar{\omega}_s^{(l)} + \sum_{l \neq l', \gamma(t_l) = \gamma(t_{l'}) = p_j} v_s^{(l)}\bar{\omega}_s^{(l')}
\end{gather*}
We want to show that the mixed terms vanish; i.e. want to show $\int_{V_j \cap M_0} v_s^{(l)} \bar{\omega}_s^{(l')} \phi dV_g \to 0$ as $\tau \to \infty$ for $l \neq l'$, so that we are left with the expression from the statement -- this would prove our claim.

Let us use the fact that $\frac{\partial \Theta}{\partial t} (t, 0) = 1$; write $v_s^{(l)} = e^{is\Theta^{(l)}}a^{(l)}$ and $\omega_s^{(l)} = e^{is\Theta^{(l)}}b^{(l)}$. This gives us that for $\varphi = \re (\Theta^{(l)} - \Theta^{(l')})$ we have $d\varphi \neq 0$ at the point $p_j$, as the geodesic intersects itself transversally. Therefore, by further reducing $\delta'$ if necessary, we may assume that $d\varphi$ is non-vanishing in $V_j$. From now on, we drop the subscript $s$ to relax the notation.

 Let $p^{(l)} = e^{-s \im \Theta^{(l)}} e^{-\lambda \re \Theta^{(l)}} a^{(l)}$ and analogously $q^{(l)} = e^{-s \im \Theta^{(l)}} e^{-\lambda \re \Theta^{(l)}} b^{(l)}$. Then we can write $v^{(l)} = e^{i \tau \re (\Theta^{(l)})} p^{(l)}$ and similarly $\omega^{(l)} = e^{i \tau \re (\Theta^{(l)})} q^{(l)}$. Then one can easily check that:
\begin{align*}
\int_{V_j \cap M_0} v^{(l)}\bar{\omega}^{(l')} \phi dV_g = \int_{V_j \cap M_0} e^{i \tau \varphi} p^{(l)} \bar{q}^{(l')} \phi dV_g
\end{align*}
Fix $\epsilon'' > 0$. In order to be able to do calculus with $\phi$, we split it into a smooth and a sufficiently small part: let $\phi = \phi_1 + \phi_2$, where $\phi_1 \in C_c^{\infty}(V_j \cap M_0)$ smooth and $\lVert{\phi_2}\rVert_{L^{\infty}(V_j \cap M_0)} \leq \epsilon''$. For the $\phi_2$ part, we have the bound $\big|\int_{V_j \cap M_0} e^{i \tau \varphi} p^{(l)} \bar{q}^{(l')} \phi_2 dV_g \big| \lesssim \lVert{p^{(l)}}\rVert_{L^2}\lVert{\bar{q}^{(l')}}\rVert_{L^2}\lVert{\phi_2}\rVert_{L^\infty} \lesssim \epsilon''$, since $\lVert{p^{(l)}}\rVert_{L^2} \lesssim \lVert{v^{(l)}}\rVert_{L^2} = O(1)$ and similarly for $q^{(l')}$.

For the main, smooth part we perform integration by parts with the operator $Lf = \langle|d\varphi|^{-2} d\varphi, $\\$df\rangle$, by noting that $\frac{1}{i\tau} L(e^{i\tau \varphi}) = e^{i \tau \varphi}$:

\begin{multline*}
\int_{V_j \cap M_0} e^{i \tau \varphi} p^{(l)} \bar{q}^{(l')} \phi_1 dV_g = \int_{V_j \cap \partial M_0 } \frac{\partial_\nu \varphi}{i \tau |d \varphi |^2} e^{i \tau \varphi} p^{(l)} \bar{q}^{(l')} \phi_1 dS  \\
+ \frac{1}{i \tau} \int_{V_j \cap M_0} e^{i \tau \varphi} L^t (p^{(l)} \bar{q}^{(l')} \phi_1) dV_g
\end{multline*}

\noindent where $L^t$ is the transpose of the operator $L$. Now we have the job to estimate the two integrals on the right hand side; the proof of this is identical to the proof in \cite{CTA}. By using the fact that $\int \tau^{\frac{n-2}{2}} e^{-c\tau |y|^2} |y|^2 dy = O(\tau^{-1})$ and that in the local chart determined by $l$, $|d \im \Theta^{(l)}| \lesssim |y|$, we have:
\begin{gather*}
\lVert{|d \im \Theta^{(l)}| v^{(l)}}\rVert_{L^2} \lVert{\bar{\omega}^{(l')}}\rVert \lVert{\phi_1}\rVert_{L^{\infty}} \lesssim \tau^{-1/2}
\end{gather*}
But this is exactly the form of summand that contributes the most to the second integral; it is the one that is obtained upon acting of $L^t$ on $e^{-s \im \Theta^{(l)}}$, because after differentiation we get an extra $\tau$ term which happily cancels with $\frac{1}{i \tau}$; everything else is bounded.

The boundary integral is bounded by previous local bounds; hence the $\frac{1}{i \tau}$ factor takes care of it. Therefore, finally, by using the previous case on each of the factors $v^{(l)}_s \bar{\omega}^{(l)}_s$, we have that:
\begin{gather*}
\lim_{\tau \to \infty} \int_{\{x'_1\} \times M_0} v_s^{(l)} \bar{\omega}^{(l)} \phi dV_g = \int_{I^{(l)}} e^{-2 \lambda t} e^{\Phi_1 + \bar{\Phi}_2} \phi dt
\end{gather*}
So by adding these for time intervals $I^{(l)}$ for $\gamma(t_l) = p_j$, we get the desired result.

We are left with the final piece of the proof, which is concerned about the concentration properties of the solutions when coupled with a $1$-form. As before, by using a partition of unity, we may assume $\phi$ has compact support in some of the $W_k$ or $V_i$ (the part of $\phi$ which is zero near the geodesic, can be made to have disjoint support with $v_s$). 

Let us first consider the case $\text{supp}(\phi) \subset W_k$. Here we have $v_s|_{W_k} = v^{(l)} = e^{i s \Theta^{(l)}} a^{(l)}$ and $\omega_s|_{W_k} = \omega^{(l)} = e^{i s \Theta^{(l)}}b^{(l)}$ for some $l$. We want to compute the following limit, where we use the $x = (t, y)$ coordinates (we drop some of the indices):
\begin{multline*}
\frac{1}{\tau} \int_{\{x'_1\} \times M_0} g^{ij} \alpha_i \frac{\partial v_s}{\partial x^j} \bar{\omega}_s \phi dV_g = \frac{is}{\tau} \int_{\{x'_1\} \times M_0} g^{ij} \alpha_i \frac{\partial \Theta}{\partial x^j} v_s \bar{\omega}_s \phi dV_g  \\
+\frac{1}{\tau} \int_{\{x'_1\} \times M_0} g^{ij} \alpha_i e^{i s\Theta} \frac{\partial a}{\partial x^j} \bar{\omega}_s \phi dV_g \to \int_0^L i \alpha_t e^{\Phi_1 + \bar{\Phi}_2} e^{-2 \lambda t} \phi dt
\end{multline*}
as $\tau \to \infty$, where $\alpha_t = \alpha(\dot{\gamma}(t))$; this is because the first integral can be computed by our previous considerations and using the fact that $\Theta = t + 1/2 \im H(t) y \cdot y + O(|y|^3)$ to compute the derivatives along the geodesic. Furthermore, the second term goes to zero by this simple estimate:
\begin{gather}\label{assss}
\lVert{\omega_s}\rVert_{L^2} \int_0^L \int_{\mathbb{R}^{n-2}} |\alpha|^2 |e^{is\Theta}|^2 |da|^2 dy dt \lesssim \int_0^L \int_{\mathbb{R}^{n-2}} \tau^{\frac{n-2}{2}} e^{-\frac{1}{2} c\tau |y|^2} dy dt = O(1)
\end{gather}
which finishes the proof in this case.

For the more complicated case $\text{supp}(\phi) \subset V_k$, we have that $v_s = \sum_{\gamma(t_l) = p_k} v_s^{(l)}$ and $\omega_s = \sum_{\gamma(t_l) = p_k} \omega_s^{(l)}$. In the coordinates $x = (t, y)$ corresponding to $I^{(l)}$, for each $l$ and $l'$ with $\gamma(t_l) = \gamma(t_{l'}) = p_k$:
\begin{multline*}
\int_{\{x_1'\} \times M_0} g^{ij}\alpha_i \frac{\partial v^{(l)}}{\partial x^j} \bar{\omega}^{(l')} \phi dV_g = \frac{i s}{\tau}\int_{\{x_1'\} \times M_0} g^{ij} \alpha_i \frac{\partial \Theta}{\partial x^j}v^{(l)} \bar{\omega}^{(l')} \phi dV_g\\ + \frac{1}{\tau} \int_{\{x_1'\} \times M_0} g^{ij}\alpha_i e^{is \Theta} \frac{\partial a}{\partial x^j} \bar{\omega}^{(l')} \phi dV_g
\end{multline*}
where we write $v^{(l)} = e^{i s \Theta} a$. Now by the previous steps, we easily see that, if $l \neq l'$, the first term is zero in the limit and the second term goes to zero by the bound \eqref{assss} above. However, if we have $l = l'$, by the previous step we again have the right limit, which is $\int_{I^{(l)}} i\alpha_t e^{\Phi_1 + \bar{\Phi}_2} e^{-2 \lambda t} dt$. Combining the results, we obtain:
\begin{gather*}
\lim_{\tau \to \infty}\int_{\{x'_1\} \times M_0} \langle{dv_s, \alpha}\rangle \bar{\omega}_s \phi dV_g = \sum_{\gamma(t_l) = p_k} \int_{I^{(l)}} i \alpha_t e^{\Phi_1 + \bar{\Phi}_2}e^{-2 \lambda t} \phi dt = \int_0^L i \alpha_t e^{\Phi_1 + \bar{\Phi}_2} e^{-2 \lambda t} \phi dt
\end{gather*}
which finally finishes the proof. Similarly to this last part of the proof we can determine the limit where the integrand is $\langle{\alpha, d\bar{\omega}_s}\rangle v_s \phi$ -- we get the same limit with just a minus sign in front.
\end{proof}

\begin{rem}\rm \label{tauK} The equation \eqref{fijna} defining $\Phi_1$ is invariant under summing with an anti-holomorphic function. Therefore, in the previous theorem, we could have inserted an extra anti-holomorphic part $h$ in the integrand of the limit. Moreover, we can see from the proof (see \eqref{tauKproof} and the lines nearby) that we could have changed the estimate $\lVert{e^{sx_1}(\Delta + X + q)e^{-sx_1}v_s}\rVert_{L^2(M)} = o(\tau)$ as $|\tau| \to \infty$ with the stronger, $O(|\tau|^{-K})$ estimate, for any $K > 0$ -- this will get used in the partial boundary data setting.
\end{rem}

\begin{rem}\rm \label{dv_s}
Note that we also have $\lVert{dv_s}\rVert_{L^2(M; T^*M)} = O(|\tau|)$ (or equivalently $\lVert{v_s}\rVert_{H^1_{scl}(M)} = O(1)$ for $h = \frac{1}{\tau}$). This simply follows from the local estimate \eqref{esimatev_s} and the fact that $dv_s = is(d\Theta) e^{is\Theta} a + e^{is \Theta} da$ (locally), so in the end we just get an extra factor of $\tau$ in the $L^2(M)$ norm.
\end{rem}

\begin{rem}\rm
It is also of interest to mention that the above construction works for metrics on $\mathbb{R} \times M_0$ that are \textit{conformal} to the product metric (this is also considered in \cite{CTA}). However, for simplicity we have omitted this conformal factor from the statement of this theorem, but more importantly we can prove Theorem \ref{mainconstruction} without this fact. It is not essential at this point (it will be important later, when when we use the integral identity) that $X$ and $Y$ are skew-Hermitian, but the equation \eqref{fijne} is simpler with this assumption.
\end{rem}

We are also interested in a vector valued version of the previous theorem. The statement of this theorem is completely analogous for vectors (matrices), as well as the proof; however, we give a sketch of the proof at some points of difference ($E' = \mathbb{R} \times M_0 \times \mathbb{C}^{m\times m}$ is the vector bundle of matrices with the fibrewise Hermitian inner product $\langle{A, B}\rangle = \Tr(AB^*)$).

\begin{theorem}[Construction of the vector valued Gaussian Beams]
Let $\gamma: [0, L] \to M_0$ be a non-tangential geodesic and let $\lambda$ be a real parameter, with $M_0$ \textit{any compact manifold with boundary}. Let $X$ and $Y$ be two skew-Hermitian matrices of vector fields on $M$ and $q$ a matrix potential; we extend $X$, $Y$ and $q$ to have compact support in $\mathbb{R} \times M_0$. Let $N_0$ be a large positive integer and denote $J_0 = [-N_0, N_0]$. Then there exists a family of generalised quasimodes satisfying the above conditions, i.e. if $s = \tau + i \lambda$, then there exists $v_s, \omega_s \in C^{\infty}(J_0 \times M_0, E')$ such that:
\begin{align*}
\Big \lVert{\Big((\Delta_{g} + X + q) + s(2\frac{\partial}{\partial x_1} - X_1) - s^2\Big)v_s}\Big \rVert_{L^2(J_0 \times M_0; E')} = o(|\tau|) \quad \text{and} \quad \lVert{v_s}\rVert_{L^2(J_0 \times M_0; E')} = O(1)\\
\Big \lVert{\Big((\Delta_{g} + Y + q) - s(2\frac{\partial}{\partial x_1} - Y_1) - s^2\Big)\omega_s}\Big \rVert_{L^2(J_0 \times M_0; E')} = o(|\tau|) \quad \text{and} \quad \lVert{\omega_s}\rVert_{L^2(J_0 \times M_0; E')} = O(1)
\end{align*}
as $\tau \to \infty$ and for each $\phi \in C(M_0)$ and $x'_1 \in \mathbb{R}$ we have:
\begin{align*}
\lim_{\tau \to \infty} \int_{\{x'_1\} \times M_0} \Tr{(v_s \omega_s^*)} \phi dV_{g} =  \int_0^L e^{-2\lambda t} \Tr{(C_X C_Y^*)}\phi(\gamma(t)) dt
\end{align*}
where $C_X$ and $C_Y$ are smooth $m \times m$ matrices on $\mathbb{R} \times [0, L]$ which satisfy the following equations:
\begin{align}\label{transportt}
\left(\frac{\partial}{\partial x_1} - i \frac{\partial}{\partial r}\right) (C_X) = \frac{1}{2}(X_1 - iX_r)C_X \quad \text{and} \quad \left(\frac{\partial}{\partial x_1} - i \frac{\partial}{\partial r}\right) (C^*_Y) = \frac{1}{2}C^*_Y(-Y_1 + iY_r)
\end{align}
Moreover, the following limits holds for $v_s$ and $\omega_s$ and any one form $\alpha$ on $\mathbb{R} \times M_0$:
\begin{align*}
\lim_{\tau \to \infty} \frac{1}{\tau} \int_{\{x'_1\} \times M_0} \Tr{(\langle{\alpha, dv_s}\rangle\omega_s^*)} \phi dV_g &= \int_0^L i \alpha(\dot{\gamma}(t)) \Tr{(C_X C_Y^*)}e^{-2 \lambda t} \phi(\gamma(t)) dt\\
\lim_{\tau \to \infty} \frac{1}{\tau} \int_{\{x'_1\} \times M_0} \Tr{(\langle{\alpha, d\omega^*_s}\rangle v_s)} \phi dV_g &= -\int_0^L i \alpha(\dot{\gamma}(t)) \Tr{(C_X C_Y^*)}e^{-2 \lambda t} \phi(\gamma(t)) dt
\end{align*}
\end{theorem}
\begin{proof}
Same as the proof of Theorem \ref{main}, with a few remarks. Firstly, every appearance of $v_s \bar{\omega}_s$ is replaced by the inner product $\Tr{(v_s \omega_s^*)}$ and we are looking for $v_s = e^{is \Theta} a$, where this time $a$ is a \textit{matrix}; so the action of $X$ and $Y$ is matrix multiplication from the left. However, formally, the computations stay the same until the appearance of $\Phi_{1,2}$; the $C_{X, Y}$ take their role, this time as matrices. Namely, when we arrive to the equation for $a_{00}$, which is \eqref{a00}:
\begin{align*}
\Big(\frac{\partial}{\partial x_1} - i\frac{\partial}{\partial t}\Big)a_{00} = \frac{1}{2}\Big(X_1 - iX_t + i\mathrm{tr}H(t)\Big)a_{00}
\end{align*}
we ask for matrices $C_X$ and $C_1$ such that $a_{00} = C_XC_1$, where:
\begin{align*}
\Big(\frac{\partial}{\partial x_1} - i\frac{\partial}{\partial t}\Big)C_X = \frac{1}{2}\Big(X_1 - iX_t\Big)C_X \quad \text{ and } \quad \frac{\partial C_1}{\partial t} = -\frac{1}{2} \Tr{(H(t))} C_1(t)
\end{align*}
so that $C_{1,2}$ play the role of $f_{1,2}$. One checks that such $a_{00}$ satisfies the conditions and for $C_1(t)$ we just take the diagonal matrix obtained by integration. This is later used to get the cancellation of $\sqrt{\det{\im H(t)}}$ with $C_1C_2^*$, which jumps out of the trace as before in the integral \eqref{det}. 

Later, when proving the mixed products vanish, the $p$'s and $q$'s introduced translate to matrices naturally and the estimates which follow stay the same. Finally, let us note that $C_X$ is invariant under multiplication on the right by an anti-holomorphic (conjugate holomorphic) matrix in the sense we could replace $C_X$ by $C_X H$ for such a matrix $H$.
\end{proof}

\begin{rem}[Everything works for admissible vector bundles] \rm\label{mainvector}
We can now easily extend the construction from the case of trivial vector bundles to the case of possibly topologically non-trivial admissible ones (see Definition \ref{vecadmissible}), equipped by a unitary connection. We restrict our attention just to operators $d_A^*d_A + Q$ induced by connections and potentials; to this end, assume the vector bundle $E = \pi^*E_0$ over $\mathbb{R} \times M_0$ is equipped with two unitary connections $A_1$ and $A_2$, where $E_0$ is a vector bundle over $M_0$.

Basically, what we need to do is to imitate the above vector proof with small alterations: to start with, let us recall the Fermi coordinates given by a map $F: J_0 \times U \to \mathbb{R} \times M_0$, where $U = [-\epsilon, \epsilon + L] \times B_\delta$ and $B_\delta$ is a small ball in dimension $(n-2)$ -- $F$ is a local diffeomorphism, giving us the tubular neighbourhood of the geodesic (see Figure \ref{fig:myfigure}). Therefore, we can pull-back the bundle $E$ to the trivial bundle $F^* E = U \times \mathbb{C}^m$ with the standard metric; we pull back the connections and the metric, as well. Furthermore, in this case we cannot work on $\text{End }E$ as we previously did in Section \ref{sec4.5}. This means we have to restrict to vector solutions and in particular our solutions to the transport equation that go into the Gaussian beams will be vectors. Then we may run the proof again; the only thing we need to replace are the resulting concentration properties:
\begin{align*}
\lim_{\tau \to \infty} \int_{\{x'_1\} \times M_0} \langle{v_s, \omega_s}\rangle_E \phi dV_{g} =  \int_0^L e^{-2\lambda t} \langle{C_1a_1, C_2a_2}\rangle_{\mathbb{C}^m}\phi(\gamma(t)) dt
\end{align*}
where $C_1$ and $C_2$ are constructed on $J_0 \times U$ for connections $A_1$ and $A_2$ as fundamental solutions to the $\bar{\partial}$-equation \eqref{transportt}, respectively; the $a_1$ is anti-holomorphic so that $C_1a_1$ solves the vector $\partial$-equation and $a_2$ is analogously holomorphic, so that $C_2a_2$ solves the $\bar{\partial}$-equation. Then we may in particular set $a_i$ to be constant and vary these constants to deduce various properties.

For the other identity we have to be slightly more careful; namely $dv_s$ is not well defined as for the trivial bundle. However, we may define it as $dv_s$ in our construction in $U$ and then push it forward by the same method of partition of unity and the map $F$ to the neighborhood of the geodesic (as in \eqref{methodpou}) and hence to the whole manifold as a $1$-form with values in $E$ (and with support in a neighbourhood of the geodesic). Then the identities become:
\begin{align*}
\lim_{\tau \to \infty} \frac{1}{\tau} \int_{\{x'_1\} \times M_0} \Big\langle{\langle{\alpha, dv_s}\rangle_{T^*M}, \omega_s\Big\rangle}_{E} \phi dV_g &= \int_0^L i \alpha(\dot{\gamma}(t)) \langle{C_1a_1, C_2a_2}\rangle_{\mathbb{C}^m} e^{-2 \lambda t} \phi(\gamma(t)) dt\\
\lim_{\tau \to \infty} \frac{1}{\tau} \int_{\{x'_1\} \times M_0} \Big\langle{\langle{\alpha, d\omega_s}\rangle_{T^*M}, v_s\Big\rangle}_{E} \phi dV_g &= \int_0^L i \alpha(\dot{\gamma}(t)) \langle{C_1a_1, C_2a_2}\rangle_{\mathbb{C}^m} e^{-2 \lambda t} \phi(\gamma(t)) dt
\end{align*}
\end{rem}

\subsection{Application of Gaussian Beams}\label{sec5.2}

We now give a concrete application of the construction of generalised quasimodes -- the construction of the CGO solutions. By using the Carleman estimates from Section \ref{sec3}, we can just put the ingredients together in a simple way. For this section, assume we are working in the setting of the CTA manifolds, that is $\tilde{g} = e \oplus g_0$ with $g = c \tilde{g}$ for a positive function $c$, where as usual we have $(M, g) \Subset (\mathbb{R} \times M_0, g)$ of the same dimension $n$.



\begin{prop}[CGO construction]\label{mainconstruction}
Let $E$ be an admissible Hermitian vector bundle, $A$ a unitary connection and $Q$ be a smooth section of the endomorphism bundle $\textnormal{End }E$. Let $s = \tau + i\lambda$, where $\tau$ and $\lambda$ are real numbers. Then there exists $\tau_0$, such that for $|\tau| \geq \tau_0$ large enough, there exists a smooth solution $u = e^{-sx_1}c^{-\frac{n-2}{4}}(v_s + r_s)$ to the equation $\Lapl_{g, A, Q} u = 0$, with the following conditions fulfilled:
\begin{align*}
\lVert{r_s}\rVert_{L^2(M; E)} = o(1), \quad \lVert{r_s}\rVert_{H^1(M; E)} = o(|\tau|) \quad \text{and} \quad \lVert{v_s}\rVert_{L^2(M; E)} = O(1)
\end{align*}
as $|\tau| \to \infty$ and the concentration properties for $v_s$ as in Theorem \ref{main}.
\end{prop}
\begin{proof}
Let us firstly notice the identity:
\begin{align*}
c^{\frac{n+2}{4}} \Lapl_{g, A, Q} (u) = \Lapl_{\tilde{g}, A, c(Q + Q_c)}(e^{-sx_1}(v_s + r_s))
\end{align*}
where $Q_c = c^{\frac{n-2}{4}} \Delta_g (c^{-\frac{n-2}{4}})$. Therefore, if we let $v_s$ be the function constructed in the proof of Theorem \ref{main}, with all its concentration properties, we will have 
\begin{align*}
\lVert{e^{sx_1}\Lapl_{\tilde{g}, A, c(Q + Q_c)}e^{-sx_1} v_s}\rVert_{L^2(M; E)} = o(|\tau|)
\end{align*}
Hence, to have the required form of the solution, $r_s$ must satisfy 
\begin{align}\label{eqnconstruction}
e^{\tau x_1}\Lapl_{g, A, Q}e^{-\tau x_1} (c^{-\frac{n-2}{4}}e^{-i\lambda x_1}r_s) = - c^{-\frac{n+2}{4}}e^{-i\lambda x_1}e^{s x_1}\Lapl_{\tilde{g}, A, c(Q + Q_c)}e^{-sx_1} v_s
\end{align}
But fortunately, now the right hand side is $o(|\tau|)$ by construction and $c$ is bounded, hence we may apply the existence theorem -- Theorem \ref{carleman}.
\end{proof}

\begin{rem}\rm \label{tauKconstruction}
Note that in Theorem \ref{mainconstruction} we can do better with the estimate on the $H^1(M; E)$ norm of $r_s$, by invoking Remark \ref{tauK} and the improved estimate on the asymptotics of $\lVert e^{sx_1} \Lapl_{A, Q} e^{-sx_1}$\\$v_s \rVert_{L^2(M; E)} = O(|\tau|^{-K})$ for any $K \geq 0$. Moreover, this implies that with the improved estimate on $v_s$ we have the $L^2$ norm of the right hand side of \eqref{eqnconstruction} equal to $O(|\tau|^{-K})$ and consequently, by Theorem \ref{carleman}, we have:
\begin{align*}
\lVert{r_s}\rVert_{L^2(M; E)} = O\big(|\tau|^{-(K + 1)}\big) \quad \text{and} \quad \lVert{r_s}\rVert_{H^1(M; E)} = O\big(|\tau|^{-K}\big)
\end{align*}
or equivalently, $H^1_{scl}(M; E) = O\big(|\tau|^{-(K+1)}\big)$.
\end{rem}

\begin{rem}\rm
Having been through the lengthy proof of existence of Gaussian Beams in case of the connection Laplacian, we now give an alternative idea on how to generalise the notion of quasimodes. Namely, it is natural to attempt to construct the analogous quantity to the approximate eigenfunction that satisfies $\lVert{(\Delta - s^2)v_s}\rVert_{L^2(M_0)} = o(|\tau|)$ by asking that $\lVert{(-d^*_Ad_A - s^2)v_s}\rVert_{L^2(M_0)} = o(|\tau|)$. However, by generalising in this way, we lose the purpose of it: we cannot build the CGO solutions using such construction. Thus, even though the construction of such solutions \textit{should} be possible and completely analogous to our main construction, we cannot find any application for it.
\end{rem}



\section{Main recovery}\label{secrecovery}

In this section we perform the last step of the procedure described in the introduction and insert the previously constructed solutions into the integral identity. By using the density of such solutions, we reduce the problem for line bundles to an $X$-ray transform on $M_0$. More precisely, in Theorem \ref{mainrecovery} we prove $dA_1 = dA_2$ if $\Lambda_{A_1} = \Lambda_{A_2}$, in the full data case. For the case of partial data, one should take extra care to deal with the leftover terms -- this is done in Theorem \ref{therem}. We use notation from Section \ref{sec5.2}.

\begin{theorem}[Main recovery for full data]\label{mainrecovery}
Suppose $A_1$ and $A_2$ are two unitary connections on $E = M \times \mathbb{C}$ and that the DN maps $\Lambda_{A_1} = \Lambda_{A_2}$ are the same. If the geodesic ray transform on $M_0$ is injective on $1$-forms and functions, then we must have $dA_1 = dA_2$.
\end{theorem}
\begin{proof}
Let $\tilde{A} = A_2 - A_1$. By Theorem \ref{mainconstruction}, we have the solutions
\begin{align}\label{CGOeqn}
u = e^{-(\tau + i\lambda)x_1}c^{-\frac{n-2}{4}}(v_1 + r_1) \quad \text{and} \quad v = e^{(\tau + i\lambda)x_1} c^{-\frac{n-2}{4}}(v_2 + r_2)
\end{align}
to the equations $\Lapl_{A_1} u = 0$ and $\Lapl_{A_2} v = 0$, with the desired concentration and decay properties. It is worth noting that $v_i$s are defined on the whole $J_0 \times M_0$, where $J_0 = [-N, N]$ for some large $N$ and $r_i$s on $M$. By applying Theorem \ref{identity}, we obtain the following equality ($dV_g$ is the volume form):
\begin{align}\label{integral}
\int_M (|A_2|_g^2 - |A_1|_g^2)u \bar{v} dV_g + \int_M \Big\langle{ud\bar{v} - \bar{v}du, A_2 - A_1}\Big\rangle_g dV_g = 0
\end{align}
Observe that in the first factor we have $q_i$s and $A_i$s bounded, which together with $L^2$ bounds on $v_i$s and $r_i$s from the construction theorem gives us that the first term is equal to $O(1)$. Now, we will divide by $\tau$ and take the $\tau \to \infty$ limit. First note that:
\begin{multline*}
\bar{v} du = e^{- 2i\lambda x_1} c^{-\frac{n-2}{2}} \Big( \big( c^{\frac{n-2}{4}} d(c^{-\frac{n-2}{4}}) - (\tau + i\lambda)\big) dx_1 (\bar{v}_2 + \bar{r}_2)(v_1 + r_1)\\
 + (\bar{v}_2 + \bar{r}_2)(dv_1 + dr_1)\Big)
\end{multline*}
and a similar formula holds for $u d\bar{v}$. The factor containing the derivative of $c$ will be zero in the limit, when divided by $\tau$.  Therefore, when plugging in these expressions in \eqref{integral}, we can neglect the $r_i$ factors and hence obtain the limit:
\begin{multline}\label{pluggedidentity}
\lim_{\tau \to \infty} \frac{1}{\tau} \int_M \langle{\tilde{A}, \bar{v} du}\rangle_g dV_g = \lim_{\tau \to \infty} \frac{1}{\tau} \int_M e^{-2i\lambda x_1} c^{-\frac{n-2}{2}}\frac{1}{c}\Big\langle \tilde{A}, (-\tau + i\lambda)v_1 \bar{v}_2 dx_1 \\
+ \bar{v}_2 dv_1\Big\rangle_{\tilde{g}}c^{\frac{n}{2}}dV_{\tilde{g}} = \int_0^L \int_{-\infty}^{\infty} e^{-2i \lambda x_1} (-\tilde{A}_1 + i\tilde{A}_t) e^{\Phi_1 + \bar{\Phi}_2} e^{-2\lambda t} dx_1 dt
\end{multline}
where, in the second line we have gone from the integral over $M$ to an integral over $\mathbb{R} \times M_0$; this is allowed since, by a \emph{boundary determination} result, we can assume that $\tilde{A}|_{\partial M} = 0$ to infinite order. Moreover, we may pick $N$ such that the interior of $J_0 \times M_0$ contains the supports of extensions of $A_1$ and $A_2$. 

Also, we used that the inner product on forms is given by the inverse of the metric $g$; hence the $\frac{1}{c}$ factor cancels with the other $c$ factors. The $\Phi_i$ functions satisfy the equations \eqref{fijne}, where $X = -2g^{ij}(A_1)_i \frac{\partial}{\partial x^j}$ and $Y  = -2g^{ij}(A_2)_i \frac{\partial}{\partial x^j}$ are the first order terms of the connection Laplacian:
\begin{align}\label{eqnfi}
\frac{\partial \Phi_1}{\partial z} = \frac{1}{2}(-(A_1)_1 + i (A_1)_t) \quad \text{and} \quad \frac{\partial \overline{\Phi}_2}{\partial z} = \frac{1}{2}((A_2)_1 - i (A_2)_t)
\end{align}
where $z = x_1 + it$ is the complex variable and $\bar{z} = x_1 - it$ is its conjugate. By summing the two equations, we get:
\begin{align}\label{integralbyparts}
\frac{\partial (\Phi_1 + \overline{\Phi}_2)}{\partial z} = \frac{1}{2}(\tilde{A}_1 - i \tilde{A}_t)
\end{align}
Now we obtain a similar expression for the $u d\bar{v}$ part, namely:
\begin{align*}
\lim_{\tau \to \infty} \frac{1}{\tau} \int_M \langle{\tilde{A}, ud\bar{v}}\rangle = \int_0^L \int_{-\infty}^{\infty} e^{-2i \lambda x_1} (\tilde{A}_1 - i\tilde{A}_t) e^{\Phi_1 + \overline{\Phi}_2} e^{-2\lambda t} dx_1 dt
\end{align*}
and finally obtain the limit for \eqref{integral}:
\begin{align}\label{desired_limit}
0 = \int_0^L \int_{-\infty}^{\infty} e^{-2i \lambda x_1} (\tilde{A}_1 - i\tilde{A}_t) e^{\Phi_1 + \overline{\Phi}_2} e^{-2\lambda t} dx_1 dt
\end{align}
By using Stokes' theorem and noting that $dz \wedge d\bar{z} = 2i dx_1 \wedge dt$, together with $\eqref{integralbyparts}$, on a smooth subdomain $\Omega \subset \mathbb{R} \times [0, L ]$ which contains the support of $\tilde{A}$:
\begin{align*}
0 = \int_\Omega d \Big(e^{-2i \lambda x_1} e^{- 2\lambda t} e^{\Phi_1 + \overline{\Phi}_2} d\bar{z} \Big) = \int_{\partial \Omega} e^{-2i \lambda x_1} e^{- 2\lambda t} e^{\Phi_1 + \overline{\Phi}_2} d\bar{z}
\end{align*}
Now by exploiting the fact that we could put an arbitrary anti-holomorphic $h$ as a multiplier of $e^{\Phi_1}$, we obtain the integral identity:
\begin{align}\label{logarithmstep0}
0 = \int_{\partial \Omega} e^{-2i \lambda x_1} e^{- 2\lambda t} h e^{\Phi_1 + \overline{\Phi}_2} d\bar{z}
\end{align}
for all such $h$. Let us take $\Omega$ simply-connected, e.g. $\Omega = J_0 \times [0, L]$ (smoothed out at the corners). This means that upon conjugating, by the proof of Lemma \ref{holo_restrict} (which proves a more general, matrix version of what we need here), the restriction of the function $e^{\overline{\Phi}_1 + \Phi_2}$ at the boundary is a restriction of a non-vanishing holomorphic function $F$, defined on $\Omega$, i.e. $F|_{\partial \Omega} = e^{\overline{\Phi}_1 + \Phi_2}|_{\partial \Omega}$. Moreover, since $\Omega$ is simply-connected, we can find a logarithm, so that $F = e^G$, where $G$ is holomorphic and we may assume $G|_{\partial \Omega} = \overline{\Phi}_1 + \Phi_2$. After using Stokes' theorem again with $\bar{h} = Ge^{-G}$, we obtain:
\begin{align}\label{logarithmstep}
0 = \int_{\Omega} e^{2\lambda i(x_1 + it)} (\tilde{A}_1 + i \tilde{A}_t) dz \wedge d\bar{z}
\end{align}
and so finally:
\begin{align*}
0 = \int_0^L\int_{-\infty}^{\infty} e^{2\lambda i(x_1 + it)} (\tilde{A}_1 + i \tilde{A}_t) dx_1 dt
\end{align*}
Let us define:
\begin{align}\label{notation1}
f(\lambda, x') &= \int_{-\infty}^{\infty} e^{i\lambda x_1} \tilde{A}_1(x_1, x') dx_1 = \mathcal{F}(\tilde{A}_1)(\lambda, x')\\
\alpha(\lambda, x') &= \sum_{j = 2}^n  \Big(\int_{-\infty}^{\infty} e^{i\lambda x_1} \tilde{A}_j(x_1, x') dx_1\Big) dx^j = \sum_{j = 2}^n \mathcal{F}(\tilde{A}_j)(\lambda, x') dx^j \label{notation2}
\end{align}
where $\mathcal{F}$ denotes the Fourier transform; we will write $\mathcal{F}(\alpha)$ for the Fourier transform of a compactly supported $1$-form $\alpha$ on $\mathbb{R} \times M_0$. With this notation, the identity above becomes (replace $2\lambda$ with $\lambda$ without loss of generality and relabel $t$ by $r$):
\begin{align*}
0 = \int_0^L e^{-\lambda r} (f + i\alpha(\dot{\gamma}(r)))dr
\end{align*}
along any unit speed, non-tangential geodesic in $M_0$. We would like to use the fact that the geodesic transform is injective as much as we can, even though we obtained an attenuated transform. Thus we set $\lambda = 0$ and use the injectivity of the ray transforms to get $\alpha(0, x') = -idp_0$ and $f(0, x') = 0$ for some smooth $p_0$ such that $p_0|_{\partial M_0} = 0$. Furthermore, we can take the $\frac{\partial}{\partial \lambda}$ derivative of the integral to get:
\begin{align*}
\int_0^L e^{-\lambda r} \Big(-r(f + i\alpha) + \frac{\partial}{\partial \lambda}(f + i\alpha)\Big)dr = 0
\end{align*}
Again we plug in $\lambda = 0$ and use injectivity, together with the following calculation:
\begin{align*}
\int_0^L ri\alpha dr = \int_0^L r\frac{\partial p_0}{\partial r} dr = -\int_0^L p_0 dr
\end{align*}
where we used the fact that $p_0$ vanishes at the boundary. Now using that $f = 0$, we obtain at $(0, x')$ for all $x' \in M_0$:
\begin{align*}
p_0 + \frac{\partial f}{\partial \lambda} = 0 \quad \text{and} \quad \frac{\partial \alpha}{\partial \lambda} = -idp_1
\end{align*}
for some smooth $p_1$ which vanishes at the boundary. It is now clear how we are going to proceed with this inductively, but let us go one step further for clarity. Taking another derivative with respect to $\lambda$, we have:
\begin{align*}
\int_0^L e^{-\lambda r}\Big(r^2(f + i\alpha) - 2r \frac{\partial (f + i\alpha)}{\partial \lambda} + \frac{\partial^2 (f + i\alpha)}{\partial \lambda^2}\Big) = 0
\end{align*}
Now by partial integration and using the properties of $p_0$, $p_1$, we have:
\begin{align*}
\int_0^L r^2 i\alpha dr = -\int_0^L 2r p_0 dr \quad \text{and} \quad \int_0^L ri\frac{\partial \alpha}{\partial \lambda} = -\int_0^L p_1 dr
\end{align*}
Therefore, by plugging in $\lambda = 0$ and substituting:
\begin{align*}
\int_0^L  \Big(\big(2p_1 + \frac{\partial^2 f}{\partial \lambda^2}\big) + i \frac{\partial^2 \alpha}{\partial \lambda^2}\Big) dr = 0
\end{align*}
Again, we get some smooth $p_2$ vanishing at the boundary such that $\frac{\partial^2 \alpha}{\partial \lambda^2} = -i dp_2$ and $2p_1 + \frac{\partial^2 f}{\partial \lambda^2} = 0$.

Now, let us assume inductively that $\frac{\partial^j \alpha}{\partial \lambda^j} = -i dp_j$ and $j p_{j-1} + \frac{\partial^j f}{\partial \lambda^j} = 0$, for $j = 0, 1, \dotsc, n - 1$ and $p_j$ are smooth functions on $M_0$ vanishing at the boundary (with $p_{-1} = 0$ predefined). We will prove the existence of $p_n$ by induction. Let us define:
\begin{align*}
S = \frac{\partial^n}{\partial \lambda^n} \int_0^L e^{-\lambda r}(f + i\alpha) dr = \int_0^L e^{-\lambda r} \sum_{j = 0}^n \Bigg(\binom{n}{j}(-1)^j r^j \frac{\partial^{n - j}(f + i\alpha)}{\partial \lambda^{n - j}} \Bigg)dr = 0
\end{align*}
Now, using the following formulas for $\lambda = 0$:
\begin{align*}
\int_0^L r^i \frac{\partial^{n - i}f}{\partial \lambda^{n - i}} dr= \int_0^L r^i \Big(-(n - i)p_{n - i - 1}\Big) dr = -(n - i)\int_0^L r^i p_{n - i - 1}dr
\end{align*} 
valid for $i > 0$ and:
\begin{align*}
\int_0^L r^k \frac{\partial^{n - k} (i\alpha)}{\partial \lambda^{n - k}} dr = \int_0^L r^k \frac{\partial p_{n - k}}{\partial r} dr = -k \int_0^L r^{k - 1}p_{n - k} dr
\end{align*}
for $k > 0$, and inserting them in the expression for $S$, we get:
\begin{multline*}
S = \int_0^L \frac{\partial^n (f + i\alpha)}{\partial \lambda^n} dr+ \sum_{j = 1}^n \binom{n}{j} (-1)^j \int_0^L \Big( (-(n - j)r^j p_{n - j - 1} - jr^{j - 1} p_{n - j})\Big) dr \\
= \int_0^L r^0 \Big(\frac{\partial^n (f + i\alpha)}{\partial \lambda^n} + np_{n - 1}\Big) dr + \int_0^L r^1 \Big( n(n - 1)p_{n - 2} - 2p_{n - 2} \cdot \frac{n(n - 1)}{2}\Big) dr +\\
+ \dotsb + \int_0^L r^j\Big(\binom{n}{j}(-1)^{j+1}(n - j) p_{n - j - 1} - (j + 1)p_{n - j - 1}\binom{n}{j + 1}(-1)^{j+1} \Big) dr + \dotsb\\
 = \int_0^L \Big( \frac{\partial^n (f + i \alpha)}{\partial \lambda^n} + np_{n-1}\Big)dr = 0
\end{multline*}
where the last line is true by cancelling the expressions in the brackets for $r^j$, where $j > 0$. Therefore, by the injectivity of the X-ray transform we have $\frac{\partial^n f}{\partial \lambda^n} + np_{n - 1} = 0$ and $\frac{\partial^n (\alpha)}{\partial \lambda^n} = -i dp_n$, for some smooth $p_n$ vanishing at the boundary. This finishes the proof by induction.

From \eqref{notation1} it follows that $\frac{\partial^k f}{\partial \lambda^k} \big|_{\lambda = 0} \leq C^k$ for some positive $C$ and all $k$, so we see that
\begin{align*}
\beta(\lambda, x') := - \sum_{k = 0}^\infty p_k(x') \frac{\lambda^k}{k!}
\end{align*} 
converges and since the Fourier transform of a compactly supported function is analytic:
\begin{align}\label{inj1}
f = \sum_{k= 0}^\infty \frac{\partial^k f}{\partial \lambda^k}\Big|_{\lambda = 0} \frac{\lambda^k}{k!} = -\sum_{k = 0}^\infty kp_{k-1} \frac{\lambda^k}{k!} = \lambda \beta 
\end{align}
and similarly, by using the relation $\frac{\partial^j \alpha}{\partial \lambda^j} = -idp_j$ (for all $j \geq 0$) we get that
\begin{align}\label{inj2}
\alpha = \sum_{k = 0}^\infty \frac{\partial^k \alpha}{\partial \lambda^k} \Big|_{\lambda = 0} \frac{\lambda^k}{k!} = -i \sum_{k = 0}^\infty dp_k \frac{\lambda^k}{k!} = i d'\beta
\end{align}
where $d'$ denotes exterior differentiation in $M_0$. Coming back to the main proof, we see that:
\begin{align*}
\sum_{2 \leq j < k} \mathcal{F}\Big(\partial_j \tilde{A}_k - \partial_k \tilde{A}_j\Big)dx^j \wedge dx^k = d'\alpha = 0
\end{align*}
Again, the Fourier transforms of the quantities on the left hand side are analytic and thus $\partial_j \tilde{A}_k \equiv \partial_k \tilde{A}_j$ for all $j, k \geq 2$. Furthermore, since we have

\begin{align*}
\mathcal{F}\Big(\partial_j \tilde{A}_1 - \partial_1 \tilde{A}_j\Big) = \partial_j f + i\lambda \alpha_j = 0
\end{align*}
for all $j \geq 2$ by \eqref{inj1} and \eqref{inj2}, in the same manner as before we have that $\partial_j \tilde{A}_1 \equiv \partial_1 \tilde{A}_j$; gluing this information together, we finally conclude that $d\tilde{A} = 0$ or equivalently that $dA_1 = dA_2$. This finishes the proof.
\end{proof}

Now we depart to partial data, which is more technical. More precisely, we have to worry about the leftover terms in the partial integration and how we extend the connections outside $M$, since now boundary determination works only on a part of the boundary, so $A_1 - A_2$ is only $L^\infty$ when extended by zero.

\begin{theorem}[Partial boundary data case]\label{therem}
In the same notation as in Theorem \ref{mainrecovery}, we prove $dA_1 = dA_2$ given $\Lambda_{A_1}|_{\Gamma} = \Lambda_{A_2}|_{\Gamma}$, where $\Gamma$ is a neighbourhood of the front side $\partial M_-$.
\end{theorem}
\begin{proof}
We are still able to prove $dA_1 = dA_2$ as follows. We think of the point $x_0$ in Theorem 1.1 from \cite{MagU} as the point at ``infinity" so that the rays are straight lines along the $x_1$ axis. 


Let us use the notation
\[F_{\epsilon} = \{x \in \partial M \mid \big\langle{\frac{\partial}{\partial x_1}, \nu(x)}\big\rangle = \nu_1(x) < \epsilon\}\] 
for any positive $\epsilon > 0$; we also denote $B_\epsilon = \partial M \setminus F_\epsilon$. We pick $\epsilon$ small enough such that $F_{\epsilon} \subset \Gamma$. Consider the CGO solutions $u$ and $v$ to $\Lapl_{A_1} u = \Lapl_{A_2} v = 0$ such that $u|_{\partial M} = f$ and $v|_{\partial M} = g$, of the form in \eqref{CGOeqn}. Then the assumption on the DN map gives us a smooth $w$, such that $\Lapl_{A_1} w = 0$, $w|_{\partial M} = g$ and $\partial_\nu w|_{\Gamma} = \partial_\nu v|_{\Gamma}$. Theorem \ref{identity} gives us (we plug in $A_1$ for $B$ and $A_2$ for $A$, so some terms swap places):
\begin{align}\label{partialidentitystart}
\int_{\partial M \setminus F_\epsilon} \langle{\partial_\nu (v - w), f}\rangle = \int_M{\big(|A_1|^2 - |A_2|^2\big)} v \bar{u} + \int_M \langle{vd\bar{u} - \bar{u}dv, A_1 - A_2}\rangle
\end{align}
Observe (recall) we have the following relations: $F(-\infty) = F = \partial M_{-}$, $B_{\epsilon} \subset \partial M_{+}$ and also $\langle{\frac{\partial}{\partial x_1}, \nu}\rangle = \nu_1(x) \geq \epsilon$ on $B_\epsilon$.

We claim that the term on the left hand side of \eqref{partialidentitystart} is equal to $O(|\tau|^{\frac{1}{2}})$ as $|\tau| \to \infty$ -- it is bounded by (using Cauchy-Schwarz)
\begin{align*}
\frac{1}{\sqrt{\epsilon}} \lVert{\sqrt{\partial_\nu x_1} e^{-\tau x_1} \partial_\nu(v - w)}\rVert_{L^2(B_\epsilon)} \times \lVert{c^{-\frac{n-2}{4}}(v_1 + r_1)}\rVert_{L^2(B_\epsilon)}
\end{align*}
which is in turn bounded (up to constant) by the following expression, by applying the Carleman estimate with the boundary part \eqref{Carlemanboundary1}, since $(v - w)|_{\partial M} = 0$:
\begin{multline}\label{bound}
\frac{1}{\sqrt{\epsilon}} \Big(\sqrt{h} \lVert{e^{-\tau x_1} \Lapl_{A_1} (v - w)}\rVert_{L^2(M)} + \lVert{\sqrt{-\partial_\nu x_1} e^{-\tau x_1} \partial_\nu (v - w)}\rVert_{L^2(\partial M_{-})}\Big)\\
\times \Big(\lVert{v_1}\rVert_{L^2(B_\epsilon)} + \lVert{r_1}\rVert_{L^2(\partial M)}\Big)
\end{multline}
The second summand in the first line of \eqref{bound} is zero by the assumption; the first one is bounded by considering the following formula:
\begin{multline*}
\Lapl_{A_1}(v - w) = \Lapl_{A_1} v = \big(\Lapl_{A_1} - \Lapl_{A_2}\big) v\\
= -2(A_1 - A_2, dv) + d^*(A_1 - A_2)v - \big(|A_1|^2 - |A_2|^2\big)v = O(|\tau|)
\end{multline*}
as $\lVert{e^{-\tau x_1} dv}\rVert = O(|\tau|)$ -- by Remark \ref{dv_s} we have $\lVert{dv_2}\rVert_{L^2(M)} = O(|\tau|)$ and by the construction in Theorem \ref{mainconstruction} we have $\lVert{r_s}\rVert_{H^1(M)} = o(|\tau|)$. Therefore, the first line is equal to $O(|\tau|^{\frac{1}{2}})$. We are left to prove the second line of \eqref{bound} is equal to $O(1)$.

Firstly, observe that by a trace inequality, we have $\lVert{r_s}\rVert_{L^2(\partial M)} \lesssim \lVert{r_s}\rVert_{H^1(M)}$; note that in the previous paragraph we had $\lVert{r_s}\rVert_{H^1(M)} = o(|\tau|)$ -- however, we can do better than that. By recalling Remark \ref{tauKconstruction} (with $K = 0$), we may assume that the $H^1$ norm of $r_s$ is bounded uniformly as $\tau \to \infty$ and hence, so is $\lVert{r_s}\rVert_{L^2(\partial M)}$.

Secondly, we want to prove that $\lVert{v_1}\rVert_{L^2(B_\epsilon)} = O(1)$ as $\tau \to \infty$ -- this will be a bit more subtle, since we will crucially use the fact that we are taking the $L^2$ norm over $B_\epsilon$ (and not over $\partial M_+$). Without loss of generality, we assume that $\partial M_\epsilon = \partial M \cap \pi^{-1}(\epsilon)$ is a manifold, where $\pi: \partial M \to \mathbb{R}$ is the projection (follows from Sard's theorem). Thus $B_\epsilon$ is compact manifold with boundary, of dimension $(n-1)$.

Notice that the second projection $\pi_2: \partial M \to M_0$ is a local diffeomorphism on $B_\eta$ for any $\eta > 0$.  So if we pick an arbitrary point $p \in B_\epsilon$ and an open neighbourhood $U$ of $p$ such that $\pi_2|_{U}$ a diffeomorphism, we see that ${\pi_2}_*(dV_{\partial M}) = J_{\pi_2} dV_{g_0}$ by the change of variables formula, where $J_{\pi_2} = \big|\det d \pi_2^{-1}\big|$ is the Jacobian. So by the properties of the integral we see that
\[\int_{U \cap B_\epsilon} |v_1|^2 dV_{\partial M} = \int_{\pi_2(U \cap B_\epsilon)} \big|v_1 \circ \pi_2^{-1}\big|^2 J_{\pi_2} dV_{g_0}\]
Note that $\pi_2^{-1}(x) = (x_1(x), x)$ on $\pi_2(U)$, where $x_1(x)$ is a smooth function, which means that by taking small enough $U$ we have $J_{\pi_2}$ bounded locally. Therefore, by the estimate \eqref{esimatev_s} in the construction of Gaussian Beams and the lines nearby, we locally have
\[\int_{\pi_2(U \cap B_\epsilon)} \big|v_1 \circ \pi_2^{-1}\big|^2 J_{\pi_2} dV_{g_0} = O(1)\]
as $\tau \to \infty$. Now since $B_\epsilon$ compact, we immediately obtain that $\lVert{v_1}\rVert_{L^2(B_\epsilon)} = O(1)$ as $\tau \to \infty$, which proves the claim.

Finally, if we quotient out by $\tau$ and take the limit $\tau \to \infty$ as before, we now have the left hand side going to zero by the estimate, which takes us back to the second step of the proof of Theorem \ref{mainrecovery} -- what follows addresses the issue that $\tilde{A}$ does not have a smooth zero extension.

Firstly, consider smooth extensions $A^\epsilon_1$ and $A^\epsilon_2$ of $A_1$ and $A_2$ respectively, with supports in $M^\epsilon$, which we define as the manifold obtained by taking the union of $M$ and its exterior $\epsilon$-collar in $\mathbb{R} \times M_0$, for some small $\epsilon > 0$. Let us also write $N^\epsilon = M^\epsilon \setminus M$ and $\tilde{A}^\epsilon = A_2^\epsilon - A_1^\epsilon$. We also denote the corresponding CGO solutions
\begin{align*}
u^\epsilon = e^{-(\tau + i\lambda)x_1}c^{\frac{n-2}{4}}(v_1^\epsilon + r_1^\epsilon) \quad \text{and} \quad v^\epsilon = e^{(\tau + i \lambda) x_1}c^{\frac{n-2}{4}}(v_2^\epsilon + r_2^\epsilon)
\end{align*}
to $\Lapl_{A_1^\epsilon} u^\epsilon = 0$ and $\Lapl_{A_2^\epsilon} v^\epsilon = 0$ in $M^\epsilon$. Corresponding to these solutions, we have $\Phi_1^\epsilon$ and $\Phi_2^\epsilon$ that satisfy the following equations:
\begin{align}\label{fiepsilonjna}
\frac{\partial \Phi^\epsilon_1}{\partial z} = \frac{1}{2}(-(A^\epsilon_1)_1 + i (A^\epsilon_1)_t) =: Z_1^\epsilon \quad \text{ and } \quad \frac{\partial \overline{\Phi^\epsilon_2}}{\partial z} = \frac{1}{2}((A^\epsilon_2)_1 - i (A^\epsilon_2)_t) =: Z_2^\epsilon
\end{align}
on $\mathbb{R} \times [0, L]$. More precisely, we have the following expressions given by the Cauchy operator:
\begin{align}\label{fiepsilonsoln}
\Phi_1^\epsilon(\omega) = \frac{1}{2\pi i} \int_\mathbb{C} \frac{Z_1^\epsilon(z)}{\bar{z} - \bar{\omega}} dz \wedge d\bar{z} \quad \text{ and } \quad \Phi_2^\epsilon(\omega) = \frac{1}{2\pi i} \int_\mathbb{C} \frac{\overline{Z_2^\epsilon}(z)}{z - \omega} dz \wedge d\bar{z}
\end{align}
Moreover, we can still solve the equation \eqref{eqnfi}, where we extend $A_1$ and $A_2$ by zero outside $M$ in the distributional sense (we denote them by the same letter) and obtain $\Phi_1, \Phi_2 \in H^1_{loc}(\mathbb{R} \times [0, L])$, satisfying the equations:
\begin{align}\label{fijna2}
\frac{\partial \Phi_1}{\partial z} = \frac{1}{2}(-(A_1)_1 + i (A_1)_t) =: Z_1 \quad \text{ and } \quad \frac{\partial \overline{\Phi_2}}{\partial z} = \frac{1}{2}((A_2)_1 - i (A_2)_t) =: Z_2
\end{align}
Furthermore, $\Phi_1$ and $\Phi_2$ have continuous representatives, which follows from the Dominated convergence theorem (DCT) applied to the Cauchy integral formula in the polar coordinate system at $\omega \in \mathbb{R} \times [0, L]$, as follows (the analogous argument applies to $\Phi_1$):
\begin{align}\label{fisoln}
\Phi_2(\omega) = \frac{1}{2\pi i} \int_\mathbb{C} \frac{\overline{Z_2}(z)}{z - \omega} dz \wedge d\bar{z} = \frac{1}{\pi} \int_0^\infty \int_0^{2 \pi} \overline{Z_2}(\omega + re^{i\theta}) e^{-i\theta} d\theta dr
\end{align}
So if $\omega_k \to \omega$, by the DCT we get that $\Phi_2(\omega_k) \to \Phi_2(\omega)$ and thus $\Phi_2$ is continuous. 


Our next aim is to compute the limit in \eqref{pluggedidentity} as $\tau \to \infty$ and $\epsilon \to 0$ for the solutions $u^\epsilon$ and $v^\epsilon$ instead of $u$ and $v$, respectively and $\tilde{A}^\epsilon$ instead of $\tilde{A}$. This integral splits into an integral over $\mathbb{R} \times M_0$, the limit of which we know and a remainder integral over $N^\epsilon$ of the following type, that we would like to prove is small in the limit as $\epsilon \to 0$:
\begin{align*}
\lim_{\tau \to \infty} \frac{1}{\tau} \int_{N^\epsilon} e^{-2i \lambda x_1} \big\langle{\tilde{A}^\epsilon, (-\tau + i\lambda) v_1^\epsilon \overline{v_2^\epsilon} dx_1 + \overline{v_2^\epsilon} dv_1^\epsilon}\big\rangle dV_{\tilde{g}}
\end{align*}

Firstly, observe that if $S \subset M_0$ is a compact submanifold with boundary and same dimension and $\gamma$ intersects the boundary of $S$ transversely, then
\begin{align*}
\lim_{\tau \to \infty}\int_{\{x_1'\} \times S} v'_1 \overline{v'}_2 \phi dV_{g_0} = \int_{\gamma^{-1}(S)} e^{\Psi_1 + \overline{\Psi}_2}e^{-2 \lambda t} dt
\end{align*}
for $x_1' \in J_0$, where $v'_1$ and $v_2'$ are some general Gaussian beams coming from our construction in Section \ref{sec5.1}, $\Psi_1$ and $\Psi_2$ are complex phases that satisfy the usual transport equations. Moreover, we have a similar formula involving the integrals of $\langle{\alpha, d\overline{v'}_2}\rangle v_1'$ and $\langle{\alpha, dv_1'}\rangle \overline{v'}_2$ for a one form $\alpha$ in the limit $\tau \to \infty$.

Secondly, recall that for almost all $x_1 \in \mathbb{R}$ we have $\partial M \pitchfork \{x_1\} \times M_0$, by applying Sard's theorem to the projection $\pi$; denote the set of such $x_1 \in J_0$ by $T$. This means that $\pi^{-1}(x_1) \cap \partial M$ is a manifold of dimension $n - 2$ for almost all $x_1$ and moreover that $N^\epsilon_{x_1}:= \pi^{-1}(x_1) \cap N^\epsilon$ is a manifold of dimension $n - 1$ with boundary for almost all $x_1$ (and similarly we set $M_{x_1}: = \pi^{-1}(x_1) \cap M$).

Thirdly, we claim that for almost all geodesics $\gamma$ in $M_0$ and for almost all $x_1 \in \mathbb{R}$, we have $\gamma \pitchfork \partial N_{x_1}^\epsilon$, where by $\gamma$ we mean the image of $\gamma$ and we identify subsets of $\{p\} \times M_0$ for some $p \in \mathbb{R}$ with subsets in $M_0$ as appropriate ($\epsilon > 0$ is fixed). To prove this, note first that the geodesics in $M_0$ are parametrised, by the influx boundary manifold $\Gamma := \partial_{+}SM_0$ which has dimension $(2n - 4)$. Furthermore, notice that the set of ``bad" geodesics, i.e. the ones that are tangent at some point to $\partial N_{x_1} ^\epsilon$, is of dimension $(2n - 5)$ (we choose a point and a unit tangent direction). Let us now define (for $x_1 \in T$):
\begin{align*}
\Gamma_{x_1} = \{\text{geodesics $\gamma \in \Gamma$ such that } \gamma \pitchfork \partial N^\epsilon_{x_1}\}
\end{align*}
and by the above dimension counting we have $\Gamma_{x_1}$ is of full measure in $\Gamma$. Let us consider the set
\begin{align*}
    A = \{(x_1, \gamma) \mid x_1 \in T \text{ and }\gamma \in \Gamma_{x_1}\} \subset J_0 \times \Gamma
\end{align*}
Since $\Gamma_{x_1}$ is of full measure in $\Gamma$ and $T$ is of full measure in $J_0$, we have $A$ is of full measure in $J_0 \times \Gamma$, by Fubini's theorem. Furthermore, again by Fubini's theorem applied to the indicator function $\chi_A$ of $A$, we conclude that for almost all $\gamma \in \Gamma$, the set $\{x_1 \mid x_1 \in J_0 \text{ and } \gamma \in \Gamma_{x_1}\}$ is of full measure in $J_0$; let us denote the set of such $\gamma$ by $\Gamma'$. This proves the claim, i.e. $\Gamma'$ is of full measure in $\Gamma$.

Moreover, notice that if we take a countable set of $\epsilon$, say $\epsilon_k \to 0$  for $k \in \mathbb{N}$, then the set of geodesics that tranversely intersect $\partial N^{\epsilon_i}_{x_1}$ for a.a. $x_1 \in J_0$ and all $i$ is of full measure, by taking a countable intersection.

We will also need the following claim: if $\gamma \in \Gamma'$, then we have $\Phi_i^\epsilon \to \Phi_i$ uniformly in $\mathbb{R} \times [0, L]$ for $i = 1, 2$ as $\epsilon \to 0$. This follows from \eqref{fiepsilonsoln} and \eqref{fisoln} in the polar coordinate form (the analogous argument works for $\Phi_1$ and $\Phi_1^\epsilon$): 
\begin{align}\label{fiuniform}
(\Phi_2^\epsilon - \Phi_2)(\omega) = \frac{1}{\pi} \int_0^\infty \int_0^{2\pi} \big(\overline{Z_2^\epsilon} - \overline{Z_2}\big)(\omega + re^{i\theta}) e^{-i\theta} d\theta dr
\end{align}
Notice that the support of $Z_2^\epsilon - Z_2$ lies in the set $S_\epsilon := \tilde{\gamma}^{-1}(N^\epsilon)$, where $\tilde{\gamma}: \mathbb{R} \times [0, L]$ maps $(x_1, t) \mapsto (x_1, \gamma(t))$. So we may write
\begin{align}\label{suppinclusion}
\text{supp}(Z_2^\epsilon - Z_2) \subset S_\epsilon = \bigcup_{x_1 \in J_0} \{x_1\} \times \gamma^{-1}(N^\epsilon_{x_1})
\end{align}
Therefore, if we define $M = \big(\text{supp}_{z, \epsilon}(|Z_2^\epsilon|) + \text{supp}_z(|Z_2|)\big)$, we have the bound
\begin{align}\label{uniformbound}
\big|\big(\Phi_2^\epsilon - \Phi_2\big)(\omega)\big| \leq \frac{M}{\pi} \int_0^\infty \int_0^{2\pi} \chi_{S_\epsilon} d\theta dr \leq 2 r_0 M + \frac{\text{area}(S_\epsilon)}{r_0}
\end{align}
for any $r_0 > 0$, where $\text{area}(S_\epsilon)$ is the $2$-dimensional Lebesgue measure. But by \eqref{suppinclusion}, Fubini and the DCT, we have:
\begin{align*}
\text{area}(S_\epsilon) = \int_{x_1} \int_{\gamma^{-1}(N_{x_1}^\epsilon)} dt dx_1 \to 0
\end{align*}
as $\epsilon \to 0$, since $\gamma \in \Gamma'$. Therefore, by taking $r_0$ small enough and then taking $\epsilon$ small enough, \eqref{uniformbound} gives a small uniform bound, which proves the claim.

Back to the main proof, for $\gamma \in \Gamma'$ we have
\begin{align*}
    \lim_{\tau \to \infty} \int_{N^\epsilon} \tilde{A}_1^\epsilon v_1^\epsilon \overline{v_2^\epsilon} dV_{\tilde{g}} = \lim_{\tau \to \infty} \int_{x_1} \int_{N_{x_1}^\epsilon} \tilde{A}_1^\epsilon v_1^\epsilon \overline{v_2^\epsilon} dV_{g_0} dx_1 = \int_{x_1} \int_{\gamma^{-1}(N^\epsilon_{x_1})} e^{\Phi_1^\epsilon + \overline{\Phi_2^\epsilon}} e^{-2\lambda t} \tilde{A}_1^\epsilon dt dx_1
\end{align*}
by Fubini, the first observation above and the Dominated convergence theorem. We may apply the DCT as $\lVert{v_i^\epsilon}\rVert_{L^2(\{x_1\} \times M_0)} = O(1)$ as $\tau \to \infty$ uniformly in $x_1 \in J_0$, for $i = 1, 2$. Furthermore, if we take $\epsilon = \epsilon_k$ with $\epsilon_k \to 0$ (e.g. $\epsilon_k = \frac{1}{k}$ for large enough $k$), we see that by the DCT (we drop the $k$ to lighten the notation):
\begin{align*}
    \lim_{\epsilon \to 0} \int_{x_1} \int_{\gamma^{-1}(N^\epsilon_{x_1})} e^{\Phi_1^\epsilon + \overline{\Phi_2^\epsilon}} e^{-2\lambda t} \tilde{A}_1^\epsilon dt dx_1 = 0
\end{align*}
since the length of $\gamma^{-1}(N^\epsilon_{x_1}) = o(1)$ as $\epsilon \to 0$, for a.a. $x_1 \in J_0$ (as $\gamma \in \Gamma'$) and the integrand is uniformly bounded. Analogously we obtain, by using Fubini, first observation and the DCT
\begin{multline*}
    \lim_{\tau \to \infty} \frac{1}{\tau} \int_{N^\epsilon} e^{-2i\lambda x_1} \overline{v_2^\epsilon} \langle{\tilde{A}^\epsilon, dv^\epsilon_1}\rangle_{\tilde{g}} dV_{\tilde{g}} = \lim_{\tau \to \infty} \frac{1}{\tau} \int_{x_1} \int_{N^\epsilon_{x_1}} e^{-2i\lambda x_1} \overline{v_2^\epsilon} \langle{\tilde{A}^\epsilon, dv^\epsilon_1}\rangle_{\tilde{g}} dV_{\tilde{g}}\\
    = i\int_{x_1} \int_{\gamma^{-1}(N_{x_1}^\epsilon)} e^{-2i\lambda x_1} \tilde{A}^\epsilon_t e^{\Phi_1^\epsilon + \overline{\Phi_2^\epsilon}} e^{-2\lambda t} dt dx_1
\end{multline*}
Note again that we may use the DCT as $\lVert{dv_i^\epsilon}\rVert_{L^2(\{x_1\} \times M_0)} = O(|\tau|)$ as $\tau \to \infty$ uniformly in $x_1 \in J_0$, for $i = 1, 2$. If we now take $\epsilon_k \to 0$, for the same reasons as before, we get
\[\lim_{\epsilon \to 0} i\int_{x_1} \int_{\gamma^{-1}(N_{x_1}^\epsilon)} e^{-2i\lambda x_1} \tilde{A}^\epsilon_t e^{\Phi_1^\epsilon + \overline{\Phi_2^\epsilon}} e^{-2\lambda t} dt dx_1 = 0\]

Going back to the identity \eqref{partialidentitystart}, taking $\tau \to \infty$ and combining with the two previous limits, we get:
\begin{align*}
    \int_{-\infty}^\infty \int_0^L e^{\Phi_1^\epsilon + \overline{\Phi_2^\epsilon}}e^{-2\lambda t} e^{-2i\lambda x_1} (\tilde{A}_1^\epsilon - i \tilde{A}_t^\epsilon) dtdx_1 = o_{\epsilon}(1)
\end{align*}
where $o_\epsilon(1)$ means $o(1)$ as $\epsilon \to 0$. As before, by using Stokes' theorem and integrating by parts over a simply connected $\Omega \subset \mathbb{R} \times [0, L]$ that contains the supports of $Z_i^\epsilon$ for $i = 1, 2$, together with inserting an anti-holomorphic function $h$ (the estimates above go through with $h e^{\Phi_1^\epsilon}$ instead of $e^{\Phi_1^\epsilon}$, as $h$ is independent of $\epsilon$), we obtain
\begin{align*}
\int_{\partial \Omega} e^{-2i \lambda (x_1 - it)} h e^{\Phi_1^\epsilon + \overline{\Phi_2^\epsilon}} d\bar{z} = o_\epsilon(1)
\end{align*}
and so by taking the limit $\epsilon \to 0$
\begin{align*}
\int_{\partial \Omega} e^{-2i \lambda (x_1 - it)} h e^{\Phi_1 + \overline{\Phi_2}} d\bar{z} = 0
\end{align*}
Now we repeat the argument of taking the logarithm from the proof of Theorem \ref{mainrecovery} (c.f. \eqref{logarithmstep0}) to get that
\begin{align*}
    \int_{\partial \Omega} e^{-2i\lambda (x_1 - it)} (\Phi_1 + \overline{\Phi_2}) d\bar{z} = 0
\end{align*}
So by going back to the $\epsilon$ limit and integrating by parts, we get (c.f. \eqref{logarithmstep})
\begin{align*}
    \int_\Omega e^{2i \lambda (x_1 + it)} (\tilde{A}_1^\epsilon + i\tilde{A}_t^\epsilon) dz d\bar{z} = o_\epsilon(1)
\end{align*}
Finally, by the Dominated convergence theorem we obtain
\begin{align*}
    \int_0^L e^{-\lambda r}(f + i\alpha(\dot{\gamma})) dr = 0
\end{align*}
with rescaling $\lambda$ and where $f$ and $\alpha$ are defined by \eqref{notation1} and \eqref{notation2} as before, for geodesics $\gamma$ in $\Gamma'$ (which is of full measure).

We claim that $f$ and $\alpha$ are in fact smooth. To show this, recall that the projection $\pi_2: \partial M \setminus \Gamma \to M_0$ is a local diffeomorphism by definition of $\Gamma$ -- therefore $\pi_2^{-1}(x')$ is a finite set of points for each $x'$ that we denote by $b_1(x') < \dotso < b_k(x')$ locally. Furthermore we set $a_1(x') = -N$, and $a_i(x') = b_{i-1}(x') + \epsilon$ for $i \geq 2$ where $\epsilon' > 0$ small enough so that $(b_i(x'), a_{i + 1}(x')] \times \{x'\} \subset M^c$ for $k - 1 \geq i \geq 1$, where $M^c$ is the complement of $M$. Therefore
\begin{align*}
    f(\lambda, x') = \sum_{i = 1}^k \int_{a_i(x')}^{b_i(x')} e^{i\lambda x_1} \tilde{A}_1(x_1, x') dx_1
\end{align*}
shows $f$ is smooth and similarly, so is $\alpha$. Same as before (formally), we get $\alpha = id'\beta$ and $f = \lambda \beta$ for some smooth $\beta$. By a computation and using $d'\alpha = 0$, we get
\begin{align*}
    \mathcal{F}\big(\partial_j \tilde{A}_k - \partial_k \tilde{A}_j\big)(\lambda, x') = \sum_{l = 1}^k e^{i \lambda b_l(x')} \Big(\frac{\partial b_l}{\partial x^k}(x') \tilde{A}_j(b_l(x'), x') - \frac{\partial b_l}{\partial x^j}(x') \tilde{A}_k(b_l(x'), x')\Big)
\end{align*}
for $j, k \geq 2$ and all $x'$ in a small open set and all $\lambda$. Note that the right hand side for fixed $x'$ is in $L^2 (\mathbb{R})$ if and only if the coefficients are zero; this implies that $\partial_j \tilde{A}_k = \partial_k \tilde{A}_j$ for a.a. $x_1$ and so $d'\tilde{A} = 0$ in $M$ by continuity.

Finally, by another computation and using $d'f + i \lambda \alpha = 0$, we have 
\begin{align*}
    \mathcal{F}\big(\partial_j \tilde{A}_1 - \partial_1 \tilde{A}_j\big)(\lambda, x') = -\sum_{l = 1}^k e^{i \lambda b_l(x')} \Big(\frac{\partial b_l}{\partial x^j}(x') \tilde{A}_1(b_l(x'), x') + \tilde{A}_j(b_l(x'), x')\Big)
\end{align*}
and we similarly conclude $\partial_j \tilde{A}_1 = \partial_1 \tilde{A}_j$ in $M$. Therefore, we globally have $d\tilde{A} = 0$.
\end{proof}
\begin{rem}\rm
In the case of a topologically non-trivial line bundle $E$, we can follow the lines of the proofs of Theorems \ref{main}, \ref{mainrecovery} and \ref{therem} to get that $d(A_1 - A_2) = 0$ (note that $\text{End E} = E \otimes E^*$ in this case is a trivial bundle, since we have the identity section, so $A_1 - A_2$ is a proper $1$-form on $M$). Namely, what one can do is to take the partition of unity used in the construction of the CGOs subordinate to $V_i$s and $W_j$s (see the equations \eqref{pou_argument} and the paragraph below it); now in each of these charts we may trivialise the bundle and by essentially re-running the last part of Theorem \ref{main} and Remark \ref{mainvector} dealing with the concentration properties, we get the limit of each individual term in the partition of unity; summing over again, we obtain the desired limit -- the equation \eqref{desired_limit}. Then the rest of the proof of Theorem \ref{mainrecovery} applies and we have a similar situation with Theorem \ref{therem}.
\end{rem}

\begin{rem}\rm
We have proved that Cauchy data uniquely determines $dA$, however ideally we would like to determine the connection up to gauge equivalence, which is finer than just determining $dA$. On simply-connected manifolds, we would have $A_2 - A_1 = dp = e^{-p} d(e^p)$ for some $p$ that we may arrange to vanish on one component of the boundary -- assuming the potentials are equal (or zero), the argument in Proposition \ref{Gextension} would imply that $e^p \equiv 1$ on the whole of $\partial M$. If additionally $\partial M$ is connected, we may recover a scalar potential, too (once we gauge transform one connection to the other, this would follow from the proof of Theorem 1.2 from \cite{CTA}). However, we can make the case without the potentials even on non simply-connected manifolds; the proof is contained in the next section and the idea is to consider $A_2 - A_1$ as a flat connection and to use a unique continuation principle.
\end{rem}

\section{Holonomy and Cauchy data}\label{holsec}

Given a manifold $M$ and a Hermitian vector bundle $E$ on it, equipped with a unitary connection $\nabla$, we can define the parallel transport along piecewise smooth curves in $M$, which is an isometry on the fibers. In particular, when this curve is a loop at a point $p$, we end up with an isometry of the fibre $E_p$, i.e. $P_{\gamma}: E_p \to E_p$ which preserves the Hermitian inner product. When $E = M \times \mathbb{C}^m$ with the standard structure, $P_\gamma$ is a unitary matrix. The \textit{holonomy group} at $p$ is defined as:
\begin{align*}
H_p(\nabla) = \{P_{\gamma}: E_p \to E_p \text{ }|\text{ } \gamma\text{ a closed loop at } p\}
\end{align*}
This naturally defines a group and moreover satisfies $P_{\gamma \cdot \gamma'} = P_\gamma \cdot P_{\gamma'}$ under path concatenation. We can also define the \textit{restricted} holonomy group as the group $H^0_p(\nabla)$ consisting of parallel transports along contractible loops -- which yields a surjective homomorphism $\rho^{\nabla}_p:\pi_1(M, p) \to H_p(\nabla)/H^0_p(\nabla)$ called the \textit{holonomy representation}. On a fixed connected manifold, these groups for varying points are all isomorphic upon conjugation by an appropriate element.

There is a close connection between the holonomy and the curvature. Namely, one can say that ``the curvature is an infinitesimal of the deviation of the holonomy"; more concretely, if we are given a parallelogram in a coordinate chart determined by two coordinate axes, say $x_1$ and $x_2$, then $F_{12} u = - \frac{\partial^2}{\partial s \partial t} T_{s, t}u$, where $F_{12}$ is the corresponding component of the curvature tensor and $T_{s, t}$ is the parallel transport along parallelogram at vertices $(0, 0)$, $(s, 0)$, $(s, t)$, $(0, t)$. For our purposes, we will need the fact that homotopic paths have the same holonomy if the curvature is zero.
\begin{lemma}
If the curvature $F_\nabla$ of $\nabla$ is zero, then $H^{0}_p(M) = 0$ for all $p \in M$.
\end{lemma}
\begin{proof}
Let $\sigma: I \times I \to M$ be a smooth homotopy between a loop $\gamma$ and the constant loop at $p \in M$, fixing the endpoints. We will make use of the identity:
\begin{align*}
\nabla_{\frac{\partial \sigma}{\partial x}} \nabla_{\frac{\partial \sigma}{\partial y}} V - \nabla_{\frac{\partial \sigma}{\partial y}} \nabla_{\frac{\partial \sigma}{\partial x}} V = F_\nabla\big(\frac{\partial \sigma}{\partial x}, \frac{\partial \sigma}{\partial y}\big) V
\end{align*}
where $V$ is any section. Let us put $V_{x, t} = T_{x, t} v$ for some $v \in E_p$, where $T_{x,t}$ is parallel transport along $\sigma(x, \cdot)$; also $\sigma(0, t) = p$ and $\sigma(1, t) = \gamma(t)$. Then we must have $0 = \nabla_{\frac{\partial \sigma}{\partial x}} \nabla_{\frac{\partial \sigma}{\partial t}} V_{x, t} = \nabla_{\frac{\partial \sigma}{\partial t}} \nabla_{\frac{\partial \sigma}{\partial x}} V_{x, t}$, which implies that $\nabla_{\frac{\partial \sigma}{\partial x}} V_{x, t}$ is parallel along $\sigma(x, \cdot)$ for all $x$. But $V_{x, 0} = v$ and $\sigma(x, 0) = p$ for all $x$ and so we have $\nabla_{\frac{\partial \sigma}{\partial x}} V_{x, 0} = 0$ for all $x$. By uniqueness of solution, we must have $\nabla_{\frac{\partial \sigma}{\partial x}} V_{x, t} \equiv 0$. Therefore, $V_{x, t}$ is parallel along $\sigma(\cdot, t)$ for all $t$. Since we know that $V_{0, 1} = T_{0, 1} v = v$ and $\sigma(x, 1) = p$ for all $x$, we must also have $V_{1, 1} = T_{1, 1} v = v = P_\gamma v$ and thus parallel transport along $\gamma$ is trivial.
\end{proof}

This means that for zero curvature, the holonomy representation is simply a map from $\pi_1$ to the holonomy group. As a warm up, let us point out some details about the construction of the parallel transport matrix. Namely, assume $E = M \times \mathbb{C}^m$ with a unitary connection $A$ has trivial holonomy and fix a point $p \in M$. Consider the matrix obtained by parallel transporting along curves emanating from $p$ and define $F(p') = P_{\gamma(p, p')}$ where $\gamma(p, p')$ is a path between $p$ and $p'$. Since the holonomy is trivial, we have $F$ well defined. Therefore, we have $dF + AF = 0$ for all $(x, v) \in TM$ and also $FF^* = Id$, since $A$ is unitary. Hence $F^{-1} A F + F^{-1} dF = 0$ and so $A$ is equivalent to the trivial connection and moreover the covariant derivative satisfies $F^{-1}(d + A)F = d$. Moreover, if we fix $p \in \partial M$ and assume that $\iota^*_{\Gamma} A = 0$ for a connected open set $\Gamma \subset \partial M$, we will have $F|_{\Gamma} = Id$, so $A$ and the trivial connection on $E$ will be gauge equivalent.

The following lemma is useful because most results on unique continuation for elliptic systems, which we will use in the proof of Theorem \ref{holtheorem}, work with usual normal derivative at the boundary (c.f. Remark \ref{UCP} below) and also for boundary determination results (c.f. \cite{LCW}).

\begin{lemma}\label{lemma}
Let $A$ and $B$ be two unitary connections on a Hermitian vector bundle $E$ over $M$. Consider the tubular neighbourhood $\partial M \times [0, \epsilon)$ of the boundary for some $\epsilon > 0$ and denote the normal distance coordinate (from $\partial M$) by $t$. Then $B$ is gauge equivalent to a unitary connection $B'$ via an automorphism $F$ of $E$ such that $F|_{\partial M} = Id$ and $(B'-A)(\frac{\partial}{\partial t}) = 0$ in the neighbourhood $\partial M \times [0, \delta)$ of the boundary, for some $\delta > 0$.

In particular, if $E = M \times \mathbb{C}^m$ we have gauges $F$ and $G$ for $A$ and $B$ respectively with $F|_{\partial M} = G|_{\partial M} = Id$, such that $A' = F^*A$ and $B' = G^*B$ satisfy $A'(\frac{\partial}{\partial t}) = B'(\frac{\partial}{\partial t}) = 0$ near the boundary.
\end{lemma}
\begin{proof}
Let us denote $B(\frac{\partial}{\partial t})$ by $B_t$. Then consider the following first order systems of differential equations, solving the parallel transport equations:
\begin{align*}
\frac{\partial F}{\partial t}(x', t) + A_t (x', t) F(x', t) = 0 \quad \text{ with } \quad F|_{\partial M} = Id\\
\frac{\partial G}{\partial t}(x', t) + B_t (x', t) G(x', t) = 0 \quad \text{ with } \quad G|_{\partial M} = Id
\end{align*} 
where $F$ and $G$ are $m \times m$ matrices, for $(x', t) \in U \times [0, \epsilon)$ for some coordinate chart $U \subset \partial M$. This has a unique smooth solution in $U \times [0, \delta')$, for some positive $\delta'$ with $\epsilon > \delta'$. Moreover, $F$ and $G$ are unitary, since $B_t$ is skew-Hermitian and if we define $H = GF^{-1}$ we have $B': = H^*B$ with $B'_t = A_t$ by the equations above:
\begin{align*}
    \frac{\partial H}{\partial t} = \frac{\partial G}{\partial t} F^{-1} + G \frac{\partial F^{-1}}{\partial t} = -B_tGF^{-1} + GF^{-1}A_t = HA_t - BH_t
\end{align*}
Moreover we see that $H:E_x \to E_x$ is defined independently of the chart for $x$ with distance less than $\delta'$ to the boundary and $(B'-A)_t = 0$.

Furthermore, there exists a $\delta > 0$ such that $H$ is close to identity in $\partial M \times [0, \delta)$, with $\delta < \delta'$. Then we may take a compactly supported function $\varphi$ on $[0, \delta')$, with $\varphi = 1$ on $[0, \delta)$, and define $\rho$ on $M$ by setting $\rho(x, t) = \varphi(t)$ in $\partial M \times [0, \delta')$ and zero elsewhere. Then we may define the unitary extension $\tilde{H} = e^{\rho \log F}$; clearly $\tilde{H}|_{\partial M \times [0, \delta)} = H$ and the globally defined $B': = \tilde{H}^*B$ satisfies the requirements.
\end{proof}

Let us briefly remark that in the next result, we will use the boundary determination that was mentioned in the introduction. For the scalar case, see Section 8 in \cite{LCW}; the result for $m \geq 2$ will appear in a forthcoming paper by the author. The basic idea is that the full jets of the quantities, such as the connection or the metric, can be restored from the full symbol of the pseudodifferential operator at the boundary determined by the DN map. Let us formulate the main theorem of this section more precisely (c.f. \cite{hol}, Theorem 6.1):

\begin{theorem}\label{holtheorem}
Let $E$ be a Hermitian vector bundle, equipped with two flat, unitary connections $A$ and $B$, and $\Gamma$ an open, non-empty subset of the boundary $\partial M$. Then the restricted DN maps agree, i.e. $\Lambda_A|_\Gamma(f) = \Lambda_B|_\Gamma(f)$ for all $f \in C_0^\infty(\Gamma; E|_\Gamma)$ if and only if $\iota^*_{\Gamma} (A - B) = 0$, the holonomy representations satisfy $\rho^A = \rho^B$ and the parallel transport matrices along any path with endpoints in $\Gamma$, with respect to $A$ and $B$ are equal\footnote{More precisely, given any $x, x' \in \Gamma$ and any path $\gamma$ between them, the parallel transport matrices $F, G: E_x \to E_{x'}$ with respect to $A$ and $B$ (respectively) along $\gamma$ are equal, i.e. $F = G$. This is to address the case when $\Gamma$ is potentially disconnected.}.
\end{theorem}
\begin{proof}
Let us firstly assume $\Lambda_A = \Lambda_B$ on $C_0^\infty(\Gamma; E|_\Gamma)$. We know this implies by boundary determination that $\iota^*_{\Gamma} (B - A) = 0$. Consider $p_1 \in \Gamma$ and a loop $\gamma$ starting at $p_1$. By standard differential topology, we can always homotopically perturb the curve such that we end up with two pieces of it: $\gamma_1: [0, 1] \to M$ starting at $p_1$ and ending at $p_2 \neq p_1 \in \Gamma$ such that $\gamma(0, 1) \subset \text{int }M$; and $\gamma_2: [1, 2] \to M$ starting at $p_2$ and ending at $p_1$ and $\text{Image}(\gamma_2) \subset \Gamma$. We moreover ask that $\gamma_1$ and $\gamma_2$ are embedded curves\footnote{We can always do this for curves in dimension $n \geq 3$ by using a version of the weak Whitney theorem to approximate; then we apply a result which says when we are close to a curve uniformly, we are homotopic to it -- for the case $n = 2$ see Remark \ref{embsurface}.}. In order to show that the holonomies are equal, it suffices to show the parallel transports along $\gamma_1$ are equal, as $\iota^*_{\Gamma} (B-A) = 0$.

We consider a tubular neighbourhood of $\gamma_1$; every such is of form $\mathcal{O} = \{p \in M^\circ \mid \text{dist}(p, \gamma_1) < \epsilon\} \cong (0, 1) \times B_{\epsilon}(0)$, where $B_{\epsilon}(0)$ is an $(n - 1)$-dimensional ball (every vector bundle over a contractible space is trivial). Therefore, we know $\mathcal{O}$ is simply connected and therefore has trivial holonomy $H_{p_1}(\mathcal{O}, B|_{\mathcal{O}}) = \{0\}$; here we also used that $B$ is flat and similarly for $A$. We consider $\epsilon > 0$ such that $\text{dist}(p_1, p_2) > 2\epsilon$, so that we have a cylindrical neighbourhood with disjoint ends. Denote $U_1 = \{p \in \Gamma : \text{dist}(p, p_1) < \epsilon\}$.

Now since both connections are flat and $\mathcal{O}$ is simply-connected, we get a unitary isomorphism $F$ between them: $F$ is obtained by taking parallel transport matrices from $p$ of both connections and composing them in a suitable way. We also have $F|_{U_1} = Id$, as $\iota^*_{\Gamma}(B - A) = 0$. Now we apply the hypothesis on the DN maps -- let $u_1$ and $u_2$ solve $\Lapl_A u_1 = \Lapl_B u_2 = 0$ with same boundary data $u_1|_{\Gamma} = u_2|_{\Gamma} = f$, so that $\nabla_\nu^A(u_1)|_{\Gamma} = \nabla_\nu^B(u_2)|_{\Gamma}$; here $\Lapl_A = \nabla_A^* \nabla_A$ and $\Lapl_B = \nabla_B^* \nabla_B$. Define $u := Fu_1$ in $\mathcal{O}$; we want to prove that $u = u_2$ on $\mathcal{O}$.

Since $F|_{U_1} = Id$, we have $u|_{U_1} = u_1|_{U_1} = u_2|_{U_1} = f|_{U_1}$. Also, using the introduction to this section and the definition of $F$, we have $F^*\nabla_B F = \nabla_A$ in $\mathcal{O}$ ($F$ is unitary). Therefore, we must have:
\begin{align}\label{conj_auto}
\Lapl_A = F^* \nabla^*_B F F^* \nabla_B F = F^* \Lapl_B F
\end{align}
which implies that $\Lapl_B u = 0$ in $\mathcal{O}$. Moreover, we have $\nabla^B_{\nu} u = \nabla^B_{\nu} u_2$ on $U_1$:
\begin{align*}
\nabla^{B}_{\nu} u &= (dF u_1 + F du_1)(\nu) + B(\nu)u = (FA(\nu)-B(\nu) F)u_1 + B(\nu)u + F \partial_{\nu} u_1 \\
&= \partial_{\nu} u_1 + A(\nu) u_1= \nabla^A_{\nu} u_1 = \nabla^B_{\nu} u_2
\end{align*}
Consequently, we have:
\begin{gather*}
\Lapl_B (u - u_2) = 0,\quad (u - u_2)|_{U_1} = 0 \quad \text{and} \quad \nabla^{B}_{\nu}|_{U_1} (u - u_2) = 0
\end{gather*}
so by a result concerning the \textit{unique continuation} properties of elliptic systems of equations (see Remark \ref{UCP} below), we must have $u \equiv u_2$ in $\mathcal{O}$; hence we must also have equality at $p_2$ by letting $p \in \mathcal{O}$ converge to $p_2$, i.e. $F(p_2) f(p_2) = f(p_2)$. Here $f$ is smooth and free to choose and therefore, we must have $F(p_2) = Id$. This concludes the proof that the holonomies are equal.

The same proof as above shows that given any $p_1, p_2 \in \Gamma$ and a path $\gamma$ between them, the parallel transport matrices along $\gamma$ of $A$ and $B$ agree, i.e. in the above notation we have $F(p_2) = Id$.

Conversely, to show that $A$ and $B$ have the same restricted DN maps under the given assumptions, just follow the  paragraph before the theorem ($B=0$ case); however, note that since we do not know that the holonomy is trivial, parallel transport from a point might not be well-defined, so we have to do something else. The idea is to provide a global horizontal section of the endomorphism bundle that is identity at the boundary and relate this with holonomy. 

Induce the standard unitary connection on the $\text{End} E$ bundle by $\nabla^{\text{End}} u = \nabla_B u - u \nabla_A$; one can easily check this new connection to be flat, as $A$ and $B$ are. Note that in a local trivialisation this is just $\hat{A}(R) = BR - RA$, where $\hat{A}$ is the new connection matrix and $R$ is a matrix. We would like to construct a unitary automorphism $U$ of $E$, such that $U|_{\Gamma} = Id$ and $U^*\nabla_A U = \nabla_B$. We do this as follows.

Fix $p_1 \in \Gamma$ as before and take a loop at $p_1$, homotope it as before and assume we are working in the tubular neighborhood $\mathcal{O}$. Then $A$ and $B$ are equivalent to a trivial connection over $\mathcal{O}$; take the parallel transport matrices $F$ and $G$ such that $dF + AF = dG + BG = 0$ in $\mathcal{O}$. Then one checks for $H = GF^{-1}$:
\begin{align*}
dH = dG F^{-1} + G dF^{-1} = -BGF^{-1} + GF^{-1}A = HA - BH
\end{align*}
One also sees that $H|_{U_1} = Id$ as $\iota^*_{\Gamma} (A - B) = 0$; also, as $\rho_A = \rho_B$ and the parallel transport along paths in $\Gamma$ is the same for $A$ and $B$, we have that $H|_{U_2} = Id$, too. Now as $H(\gamma_1(t))$ is parallel transport with respect to $\hat{A}$, we get the parallel transport of $Id$ along $\gamma_1$ at $p_2$ is $Id$; therefore parallel transport of $Id$ along $\gamma$ is trivial. So we may define $U(x)$ to be parallel transport with respect to $\hat{A}$ of $Id$ from $p_1$ to $x$, for every $x \in M$; the fact we get identity when we parallel transport between any two components at the boundary then gives $U|_\Gamma = Id$, which concludes the proof.
\end{proof}

Note that the proof above does not generalise if we add arbitrary potentials, since the local gauge isomorphism between two connections has no a priori reason to intertwine the potentials (see \eqref{conj_auto}). However, it generalises in the case $m = 1$ and $Q_A = Q_B$, since the group action is abelian in that case.



\begin{rem}\rm\label{linebundle}
Moreover, in the case of line bundles, it is true that for any two connections $A_1$ and $A_2$ for which we know $d(A_1 - A_2) = 0$: $\Lambda_{A_1}|_{\Gamma} = \Lambda_{A_2}|_{\Gamma}$ if and only if $\iota_{\Gamma}^*(A_1 - A_2) = 0$, the holonomy representation of $A_1 - A_2$ (on $M \times \mathbb{C}$) is trivial and the parallel transport maps with respect to $A_1 - A_2$ between boundary components in $\Gamma$ are equal to the identity. This can be easily seen from the above proof.
\end{rem}

\begin{rem}\rm\label{UCP}
The unique continuation result we are using follows from Theorem 2.3 in \cite{pat2015UCP}, which considers the case of the wave equation (covers our setting if we let $u$ independent of $t$) with the covariant normal derivative at the boundary and so solves our problem. However, it is not ideal since it gives more than we need. More adequate are techniques in Corollary 3.4 and Theorem 3.2 in \cite{tataruUCP} (although they do not use the covariant derivative), since for an elliptic operator, any smooth surface is pseudoconvex. See also Appendix B in \cite{mikkoUCP}.
\end{rem}

\begin{rem}\rm\label{embsurface}
In the case of surfaces, we need to be careful when approximating curves by embeddings -- we do not have enough space to get rid of possible self-intersections. However, there is a way around this by considering just the class of simple curves, by which we can represent generators of $\pi_1$ (see \cite{hol}, Section 6, for details). Furthermore, in \cite{hol} the Conjecture \ref{conjecture1} for Riemann surfaces and line bundles is proved, but with the extra bit of a potential added to the connection Laplacian (so the claim is more general in that case). There, the authors prove the identification of a potential \textit{before} identification of a connection (see the comment after the proof of Theorem \ref{holtheorem}). In our recovery Theorem \ref{mainrecovery}, we first prove the identification of a connection.
\end{rem}

Now we are in a good shape to prove the main theorem: all ingredients are ready. Theorems \ref{mainrecovery} and \ref{therem} almost finish the proof, however Theorem \ref{holtheorem} provides us with the necessary gauge.

\begin{proof}[Proof of Theorem \ref{maintheorem}]\label{mainproof}
Recall that we have $d(A_1 - A_2) = 0$ from Theorem \ref{mainrecovery} for full data and from Theorem \ref{therem} for partial data. By Remark \ref{linebundle}, we immediately get our gauge in both cases. This finishes the proof.
\end{proof}

\Addresses

\end{document}